\documentclass[12pt, reqno]{amsart}
\usepackage{amsmath,amssymb,pst-all,graphicx}
\usepackage{amscd, amsthm}
\usepackage{mathrsfs}
\usepackage[cmtip,arrow]{xy}
\usepackage{pb-diagram, pb-xy}
\usepackage{tikz-cd}
\usetikzlibrary{calc, intersections, through, backgrounds, bending, arrows.meta}
\usepackage{bm}
\usepackage{enumerate}
\usepackage{enumitem}
\usepackage[margin=25mm]{geometry}

\usepackage{pgfplots}
 
\pgfplotsset{
    compat=newest
}

\allowdisplaybreaks %displaystyleの数式で改ページを許可する．「\\*」での改行は改ページしない．

\newcommand{\CC}{\mathbb C}
\newcommand{\DD}{\mathbb{D}}

\newcommand{\RR}{\mathbb R}
\newcommand{\PP}{\mathbb P}

\newcommand{\TT}{\mathbb T}
\newcommand{\ZZ}{\mathbb Z}

\newcommand{\mcA}{\mathcal{A}}
\newcommand{\mcB}{\mathcal B}
\newcommand{\mcC}{\mathcal C}

\newcommand{\mcI}{\mathcal{I}}

\newcommand{\mcL}{\mathcal{L}}

\newcommand{\mcS}{\mathcal{S}}
\newcommand{\mcT}{\mathcal{T}}
\newcommand{\mcY}{\mathcal{Y}}

\newcommand{\mcV}{\mathcal{V}}

\newcommand{\gtc}{\mathfrak{c}}
\newcommand{\gtf}{\mathfrak{f}}
\newcommand{\gtg}{\mathfrak{g}}
\newcommand{\gth}{\mathfrak{h}}
\newcommand{\gtm}{\mathfrak{m}}

\newcommand{\BM}{\mathscr{B}}
\newcommand{\PM}{\mathscr{P}}

\newcommand{\Comb}{\mathop {\rm Cmb}\nolimits}

\newcommand{\SL}{\mathop {\rm SL}\nolimits}

\newcommand{\GL}{\mathop {\rm GL}\nolimits}

\newcommand{\Sing}{\mathop {\rm Sing}\nolimits}
\newcommand{\Str}{\mathop {\rm Str}\nolimits}

\newcommand{\id}{\mathop {{\rm id}}\nolimits}

\newcommand{\parti}{\mathfrak{p}}
\newcommand{\Parti}{\mathfrak{P}}
\newcommand{\sm}{\mathrm{sm}}
\newcommand{\pr}{\mathrm{pr}}
\newcommand{\Irr}{\mathrm{Irr}}

\newcommand{\modi}{\mathrm{m}}
\newcommand{\fin}{\mathrm{Fin}}
\newcommand{\plumb}{\mathrm{pg}}
\newcommand{\dplumb}{\mathrm{dpg}}

\newcommand{\norm}{\mathrm{nml}}
\newcommand{\Conn}{\mathrm{Conn}}

\newcommand{\AH}{\mcA}
\newcommand{\Ver}{\overline{\mcA}}
\newcommand{\arrowhead}{\begin{tikzpicture} \draw[-{Stealth[length=8pt]}] (0,0) -- (0.1,0);  \end{tikzpicture}}

\newcommand{\sign}{\mathrm{sgn}}

\newcommand{\W}{\mathrm{W}}

\renewcommand{\red}{}

\theoremstyle{oupplain}
\newtheorem{thm}{Theorem}[section] % 1st argument is your name for it
\newtheorem{lem}[thm]{Lemma}     % 2nd argument is what is printed
\newtheorem{cor}[thm]{Corollary}

\newtheorem{fact}[thm]{Fact}
\theoremstyle{oupdefinition}
\newtheorem{defin}[thm]{Definition}
\theoremstyle{oupremark}
\newtheorem{rem}[thm]{Remark}

\numberwithin{equation}{section}

\begin{document}

\title[A note on combinatorial type]{A note on combinatorial type and splitting invariants of plane curves}
\author{Taketo Shirane}
\address{Department of Mathematical Sciences, Faculty of Science and Technology, Tokushima University, 2-1 Minamijosanjima, Tokushima, 770-8506, Japan}
\email{shirane@tokushima-u.ac.jp}
%\date{
%%\today
%}
\maketitle

\begin{abstract}
Splitting invariants describe how a plane curve ``splits" by the pull-back under a Galois cover over the projective plane whose branch locus contains no component of the plane curve. 
They enable us to distinguish the embedded topology of several plane curves with the same fundamental group of the complements. 
In this note, we introduce a generalization of splitting invariants, called the \textit{$G$-combinatorial type}, for plane curves by using the modified plumbing graph defined by Hironaka. 
We prove the invariance of the $G$-combinatorial type under certain homeomorphisms based on the arguments of graph manifolds by Waldhausen and plumbing graphs by Neumann. 
Furthermore, we distinguish the embedded topology of \textit{quasi-triangular curves} by the $G$-combinatorial type, which are generalization of triangular curves studied by Artal, Cogolludo and {\red Mart\'in \cite{artal_cogolludo_martin_2019}}.  
\end{abstract}

%%%%%%%%%%%%%%%%%%%%%%%%%%%%%%%%
\section{Introduction}
%%%%%%%%%%%%%%%%%%%%%%%%%%%%%%%%

In this note, we call a reduced (possibly reducible) curve $\mcC$ on the complex projective plane $\PP^2:=\CC\PP^2$ a \textit{plane curve}, and call the homeomorphism class of the pair $(\PP^2,\mcC)$ the \textit{embedded topology} of $\mcC\subset\PP^2$. 
In the study of the embedded topology of plane curves, one of the main objects is {\red their classification}. 
For this, it is natural to consider the data consisting of the number of irreducible components of a plane curve $\mcC$, the degree {\red of each component and an embedded resolution of its singularities}, which is called the \textit{combinatorial type} (or \textit{combinatorics} for short) of $\mcC$ (see \cite[Definition 2.1]{abst2023} or Section~\ref{sec:comb} for details). 
It was mentioned in \cite{act2008} that Fact~\ref{fact:comb} below can be proved by the arguments in \cite{waldhausen_1967} and \cite{neumann1981}. 
However, there is no proof in the literature. 
In order to make clear what we obtain from \cite{waldhausen_1967} and \cite{neumann1981}, we give a precise proof of Fact~\ref{fact:comb} in this note (Section~\ref{sec:comb}). 
%which can be proved by \cite[III, Theorem~21]{brieskorn_knorrer_1986} and the results in \cite{waldhausen_1967} and \cite{neumann1981}  (see Section~\ref{sec:comb}). 

\begin{fact}\label{fact:comb}
Let $\mcC_1, \mcC_2\subset\PP^2$ be two plane curves. 
The {\red following} hold:
\begin{enumerate}
	\item $\mcC_1$ and $\mcC_2$ have the same combinatorics if and only if $\mcC_1$ and $\mcC_2$ have the same embedded topology in regular neighborhoods, that is, there is a homeomorphism $h\colon U(\mcC_1)\rightarrow U(\mcC_2)$ with $h(\mcC_1)=\mcC_2$ for some regular neighborhoods $U(\mcC_i)$ of $\mcC_i$ in $\PP^2$. 
	\item If $\mcC_1$ and $\mcC_2$ have the same embedded topology in $\PP^2$, then they have the same combinatorics. 
\end{enumerate}
\end{fact}

Zariski \cite{zariski_1929} %gave a counter-example for the converse of Fact~\ref{fact:comb} (ii). 
proved that the fundamental group $\pi_1(\PP^2\setminus C)$ for a sextic curve $C\subset\PP^2$ with $6$ cusps is isomorphic to the free product $(\ZZ/2\ZZ)\ast(\ZZ/3\ZZ)$ if there exists a conic passing through the $6$ cusps of $C$, and to $\ZZ/6\ZZ$ otherwise. 
This result shows that the converse of Fact~\ref{fact:comb} (ii) is false. 
By Artal \cite{artal_1994}, a pair $(\mcC_1,\mcC_2)$ of plane curves $\mcC_1,\mcC_2\subset\PP^2$  is called a \textit{Zariski pair} if $\mcC_1$ and $\mcC_2$ have the same combinatorics but different embedded topology. 
The fundamental group $\pi_1(\PP^2\setminus\mcC)$ for a plane curve $\mcC\subset\PP^2$ is an effective invariant, and one of main objects in the study of the embedded topology (cf.\ \cite{act2008}). 
On the other hand, there are examples of Zariski pairs $(\mcC_1,\mcC_2)$ with $\pi_1(\PP^2\setminus\mcC_1)\cong\pi_1(\PP^2\setminus\mcC_2)$, which are called \textit{$\pi_1$-equivalent Zariski pairs} (cf.\ \cite{amram_bannai_shirane_sinichkin_tokunaga_2023}, \cite{degtyarev_2008}, \cite{shimada2003, shirane2017}). 
Hence it is important to introduce invariants different from the fundamental group. 

Guerville-Ball\'e and Meilhan introduced the \textit{linking invariant} for plane curves in \cite{guerville_meilhan_2020}, which is derived from the linking number in the link theory. 
The linking invarinat is a generalization of \textit{$\mcI$-invariant} for line {\red arrangements} defined by Artal, Florens and Guerville-Ball\'e \cite{afg2017_I_inv}. 
For line arrangements, these invariants are generalized to more precise invariants, the \textit{loop-linking number} by Cadegan-Schlieper \cite{cad2018} and the \textit{homology inclusion} by Rodau \cite{rodau:2025aa} via the inclusion of the boundary manifold into the complement of a line arrangement. 

From an algebro-geometric view-point, Bannai \cite{bannai2016} and the author \cite{shirane2017, shirane2018, shirane2019} introduced several \textit{splitting invariants}, which is induced from the studies of \textit{splitting curves} with respect to a double cover by Artal and Tokunaga 
{\red 
et al. \cite{artal_carmona_cogolludo_tokunaga_2001}, \cite{artal_cogolludo_tokunaga_2007}, \cite{artal_tokunaga_2004}, \cite{tokunaga_2010}. 
} 
%{\color{orange} 
%From an algebro-geometric view-point, Bannai \cite{bannai2016} and the author \cite{shirane2017, shirane2018, shirane2019} introduced several \textit{splitting invariants}, which is induced from the studies of \textit{splitting curves} with respect to a double cover by Artal and Tokunaga \cite{artal_tokunaga_2004}, \cite{tokunaga_2010}. 
%} 
For two plane curves $\mcB,\mcC\subset\PP^2$ with no common component, and a Galois cover $\phi\colon X\to\PP^2$ branched along $\mcB$, a splitting invariant of $\mcC$ with respect to $\phi$ is some effective data of the configuration of components of $\phi^{-1}(\mcC\setminus\mcB)$. 
Both of the linking and splitting invariants {\red encode} how a plane curve $\mcC$ is ``entangled" with the other curve $\mcB$, and they {\red are not determined by} the fundamental group (cf.\ \cite{shirane2017} and \cite{guerville_shirane_2017}). 
Splitting invariants are effectively used for distinguishing the embedded topology (cf.\ \cite{amram_bannai_shirane_sinichkin_tokunaga_2023}, \cite{abst2023}, \cite{abst_MaxFlex_2023}, \cite{bannai_kawana_masuya_tokunaga_2022}). 

Let $G$ be a finite group, and let $\phi\colon X\to\PP^2$ be the $G$-cover induced by a surjection $\theta\colon\pi_1(\PP^2\setminus\mcB)\twoheadrightarrow G$. 
{\red 
Since the unbranched cover $X\setminus\phi^{-1}(\mcB)\to \PP^2\setminus\mcB$ is uniquely determined by $\theta$ topologically, the data of $\phi^{-1}(\mcC\setminus\mcB)$ coincide with that of $\phi^{-1}(h(\mcC)\setminus\mcB)$ for a homeomorphism $h\colon \PP^2\setminus\mcB\to\PP^2\setminus\mcB$ and a plane curve $\mcC\subset\PP^2$. 
However, it is not sure whether the branched cover $\phi$ is uniquely determined by $\theta$. Hence the splitting invariants do not contain any information of the ramification locus $\phi^{-1}(\mcB)$. 
} 
The aim of this note is to give a generalization of splitting invariants including information near $\phi^{-1}(\mcB)$,  called the \textit{$G$-combinatorial type} (or the \textit{$G$-combinatorics} for short). 
The idea is to investigate the boundary of a certain regular neighborhood of a plane curve, called the \textit{boundary manifold}, based on arguments of \textit{graph manifolds} \cite{waldhausen_1967}, \textit{plumbing graphs} \cite{neumann1981} and \textit{modified plumbing graphs} \cite{hironaka2000}. 
{\red An obstruction} to do this is that the plumbing graph naturally obtained from a plane curve (which corresponding to its combinatorics) may be non-normal in the sense of \cite{neumann1981}. 
In order to avoid this {\red obstruction}, we first discuss locally boundary manifolds near singularities of plane curves, and secondly consider the boundary manifold globally. 
As an application of $G$-combinatorics, we distinguish the embedded topology of \textit{quasi-triangular curves}, which is a generalization of triangular curves studied in \cite{artal_cogolludo_martin_2019} (Section~\ref{sec:qt_curve}). 
The embedded topology of quasi-triangular curves are not distinguished by splitting invariants.

%We mainly use the results in \cite{waldhausen_1967} and \cite{neumann1981} to prove the invariance of $G$-combinatorics for certain homeomorphisms. 
%Waldhausen \cite{waldhausen_1967} defined the \textit{graph manifolds} which are real $3$-manifolds constructed by gluing $S^1$-bundles over real $2$-manifolds, and classified them, with minor exceptions. 
%Neumann \cite{neumann1981} described a calculus for plumbing graphs which produces an algorithm for determining the homeomorphism class represented by given a plumbed manifold, where a plumbed manifold is a graph manifold associated with a plumbing graph. 
%A correspondence from graph manifolds to plumbed manifold is also given in \cite{neumann1981}, hence the class of plumbed manifolds is precisely the class of graph manifolds. 
%Plumbing graphs are useful to describe the topology of the boundary manifolds of normal crossing curves on smooth complex surfaces, where the boundary manifold of a curve on a complex surface is the boundary of a regular neighborhood of the curve. 
%In fact, Neumann \cite{neumann1981} applied the calculus of plumbing graphs to the topology of isolated singularities of complex surfaces and one-parameter families of complex curves. 
%Hironaka \cite{hironaka2000} defined \textit{modified plumbing graphs} based on \cite{neumann1981}, which is useful to describe the topology of certain algebraic curves on normal complex projective surfaces. 
%These works are important for our aim. 

In this note, we define a regular neighborhood $\mcT(\mcC)$ of a plane curve $\mcC$ as the algebraic neighborhood \cite{durfee_1983}, and we call its boundary $\BM(\mcC):=\partial\mcT(\mcC)$ the \textit{boundary manifold} of $\mcC$ (Subsection~\ref{subsec:reg_nbd}). 
We will obtain Theorem~\ref{thm:meridian} below to prove the invariance of $G$-combinatorics in Section~\ref{sec:comb}. 
Let $\mcC_1,\mcC_2\subset\PP^2$ be two plane curves, 
and let $\sigma_i\colon Y_i\to\PP^2$ be the minimal good embedded resolution of $\mcC_i$. 
Put $\mcC_{i}':=\sigma_i^{-1}(\mcC_i)$. 
Note that the combinatorics of $\mcC_i$ is redefined as the dual graph $\Comb(\mcC_i)$ of $\mcC_i'$ with marking $\Str_{\mcC_i}$ of strict transforms of components of $\mcC_i$ in \cite{abst2023} (see Subsection~\ref{subsec:comb} for details). 
We can define a meridian ${\red\gtm_{D}}\in\pi_1(\PP^2\setminus\mcC_i)$ of each component $\red D$ of $\mcC'_{i}$ in $Y_i\setminus\mcC_{i}'=\PP^2\setminus\mcC_i$, which have canonically orientation derived from the complex structure of $Y_i$ (cf.\  \cite[Definition 1.2]{act2008}, \cite[p.133]{shirane2019}). 
We call $\red\gtm_{D}^{-1}$ an \textit{anti-meridian} of a component $\red D$ of $\mcC_{i}'$, which is a meridian with reversed orientation. 
It is known that a homeomorphism $h\colon U(\mcC_1)\rightarrow U(\mcC_2)$ induces an isomorphism $\Psi_h^\infty\colon\pi_1(\BM(\mcC_1))\to\pi_1(\BM(\mcC_2))$. 
The following theorem follows from Theorem~\ref{thm:homeo_boundary}.

\begin{thm}\label{thm:meridian}
	Assume that there is a homeomorphism $h\colon U(\mcC_1)\rightarrow U(\mcC_2)$ with $h(\mcC_1)=\mcC_2$ for some regular neighborhoods $U(\mcC_i)$. 
	Then there exist an equivalence map $\varphi\colon\Comb(\mcC_1)\to\Comb(\mcC_2)$ and a homeomorphism $\Psi_\BM\colon\BM(\mcC_1)\to\BM(\mcC_2)$ such that
\begin{enumerate}
	\item $\Psi_{\BM\ast}=\Psi_h^\infty$; 
	\item there is $\alpha\in\{\pm1\}$ such that $\red\Psi_{\BM\ast}(\gtm_{D})=\gtm_{\varphi(D)}^\alpha$ up to conjugate for any irreducible component $\red D$ of $\mcC_1'$, where $\red D$ is regarded as also a vertex of $\Comb(\mcC_i)$.  
\end{enumerate}
\end{thm}

Let us state the main result of this note. 
Let $G$ be a finite group. 
Assume that there are surjective {\red homomorphisms} $\theta_i\colon \pi_1(\PP^2\setminus\mcC_i)\to G$. 
Then $\theta_i$ induce $G$-covers $\phi_i\colon X_i\to\PP^2$ and $\tilde\phi_i\colon \widetilde{X}_i\to Y_i$ branched along sub-curves of $\mcC_i$ and $\mcC'_{i}$, respectively. 
Put $\widetilde{\mcC}_{i}:=\tilde\phi_i^{-1}(\mcC'_{i})$. 
Since each irreducible component of $\widetilde{\mcC}_{i}$ is unibranched (cf.\ \cite[Chapter~II]{laufer1971}), we can define the dual graph $\Gamma_{i}^\theta$ of $\widetilde\mcC_{i}$, 
and we will define a \textit{$G$-combinatorics $\Comb_G(\mcC_i,\theta_i)$} of $\mcC_i$ with respect to $\theta_i$ based on $\Gamma_i^\theta$, which {\red describes} the topology of {\red the} boundary manifold of $\widetilde\mcC_{i}$ in $\widetilde X_i$ (Definition~\ref{def:G-comb}). 
For an equivalence map $\varphi\colon \Comb(\mcC_1)\to\Comb(\mcC_2)$, we will define a \textit{$G$-equivalence map $\varphi_G^\theta\colon \Comb_G(\mcC_1,\theta_1)\to\Comb_G(\mcC_2,\theta_2)$} with respect to $\varphi$ (Definition~\ref{def:G-equiv}). 
The following theorem follows from Theorem~\ref{thm:invariance_G-comb}. 

\begin{thm}\label{thm:invariance}
Let $\mcC_i\subset\PP^2$ be two plane curves with surjections $\theta_i\colon \pi_1(\PP^2\setminus\mcC_i)\to G$. 
Assume that there are a homeomorphism $h\colon \PP^2\to\PP^2$ with $h(\mcC_1)=\mcC_2$ and {\red an} automorphism $\tau\colon G\to G$ such that $\tau\circ\theta_1=\theta_2\circ h_\ast$, {\red where $h_\ast\colon \pi_1(\PP^2\setminus\mcC_1)\to\pi_1(\PP^2\setminus\mcC_2)$ is the isomorphism induced by $h$.  
Let} $\varphi\colon \Comb(\mcC_1)\to\Comb(\mcC_2)$ be the equivalence map induced by $h$. 
Then there is an equivalence map $\varphi_G^\theta\colon \Comb_G(\mcC_1,\theta_1)\to\Comb_G(\mcC_2,\theta_2)$ with respect to $\varphi$. 
\end{thm}

\begin{rem}
The $G$-combinatorics $\Comb_G(\mcC_i,\theta_i)$ of $\mcC_i\subset\PP^2$ with respect to $\theta_i\colon \pi_1(\PP^2\setminus\mcC_i)\to G$ can be algebraically constructed from the minimal resolution 
$\tilde{\sigma}'_{i}\colon \widetilde{X}'_{i}\to \widetilde{X}_i$ of the $G$-cover $\tilde\phi_i\colon \widetilde{X}_i\to Y_i$ and the pull-back ${\tilde{\sigma}_i^{\prime\ast}}\circ\phi_i^\ast(\mcC'_{i})$, and 
describe the topology of the boundary manifold of $\tilde\phi_i^{-1}(\mcC'_{i})$ in $\widetilde X_i$ (cf.\ Remark~\ref{rem:modified plumbing graph}). 
Hence it may be good to compare some topological invariant (like the linking invariant) and some algebraic data (like splitting invariants) through $G$-combinatorics. 
In fact, known splitting invariants can be constructed from $G$-combinatorics (see Section~\ref{sec:splitting_G-comb}). 
\end{rem}

This note is organized as follows: 
In Section~\ref{sec:preliminary}, we recall the definitions of graph manifolds, plumbing graphs and modified plumbing graphs, and several results in \cite{waldhausen_1967}, \cite{neumann1981}. 
In Section~\ref{sec:comb}, we give proofs of Fact~\ref{fact:comb} and Theorem~\ref{thm:meridian} by using arguments of \cite{waldhausen_1967} and \cite{neumann1981}. 
In Section~\ref{sec:G-comb}, we define the $G$-combinatorics based on the results in \cite{hironaka2000}, and prove Theorem~\ref{thm:invariance}. 
In Section~\ref{sec:splitting_G-comb}, we construct known splitting invariants from the $G$-combinatorics. 
In Section~\ref{sec:qt_curve}, we distinguish the embedded topology of quasi-triangular curves by $G$-combinatorics. 
In Appendix, we recall results of \cite{waldhausen_1967}, \cite{neumann1981} and several papers, which we use in this note. 

\medskip

\noindent
\textbf{Acknowledgements:} 
The author thanks Professor E.~Artal~Bartolo for his valuable comments. 
He also thanks the reviewers for their valuable comments and suggestions that helped improve this note. 
%{\red The author is supported by JSPS KAKENHI Grant Number JP21K03182.}

%%%%%%%%%%%%%%%%%%%%%%%%%%%%%%%%
\section{Preliminary}\label{sec:preliminary}
%%%%%%%%%%%%%%%%%%%%%%%%%%%%%%%%

We call a manifold of real dimension $n$ an \textit{$n$-manifold}. 
For a compact manifold $M$ and a union $F$ of closed subspaces of $M$, we choose a (finite) triangulation, in which each component of $F$ is a subcomplex. 
The closed star of $F$ in the second barycentric subdivision for the triangulation of $M$ is called a \textit{regular neighborhood} of $F$ in $M$, denoted by $U_M(F)$. 
If there is no confusion, $U_M(F)$ is simply denoted as $U(F)$. 
We may assume that a regular neighborhood is sufficiently small. 
For a subset $A\subset M$, let $A^\circ$ denote the interior of $A$ in $M$. 

In this note, a \textit{graph} $\Gamma$ means a collection of vertices $\mcV(\Gamma)$ and oriented edges $\mcY(\Gamma)$ together with maps 
\begin{align*}
	o&\colon \mcY(\Gamma)\longrightarrow \mcV(\Gamma), &
	t&\colon \mcY(\Gamma)\longrightarrow \mcV(\Gamma), &
	\bar{\ \ }&\colon  \mcY(\Gamma) \longrightarrow \mcY(\Gamma) \quad (y\longmapsto \overline{y})
\end{align*} 
satisfying $o(\overline{y})=t(y)$, $t(\overline{y})=o(y)$, $\overline{y}\ne y$ and $\overline{\overline{y}}=y$. 
For each $y\in\mcY(\Gamma)$, we call $o(y)$ the {\it initial vertex}, $t(y)$ the {\it terminal vertex}, and $\overline{y}$ the {\it inverse edge} of $y$. 
In a figure of a graph $\Gamma$, we depict the set $\{y,\overline{y}\}$ as a segment from $o(y)$ to $t(y)$ for each edge $y\in\mcY(\Gamma)$. 
Let 
$ \bm{d}\colon \mcV(\Gamma)\to\ZZ_{\geq0} $ 
be the \textit{degree map}, i.e., $\bm{d}(v)$ is the number of edges $y\in\mcY(\Gamma)$ with $t(y)=v$ for each $v\in\mcV(\Gamma)$.

Throughout this note, we put 
\begin{align*}
	I&:=[0,1]\subset\RR,  &
	{\red \DD^2}&:=\{z\in\CC\mid |z|\leq 1\}, &
	S^1&:=\{z\in \CC\mid |z|=1\}. 
\end{align*}
For $A=(\begin{smallmatrix} a&b\\c&d \end{smallmatrix})
\in\SL(2,\ZZ)$, $A$ also denotes the diffeomorphism $T\to T$ of the torus $T:=S^1\times S^1$ given by $A(t_1, t_2):=(t_1^at_2^b, t_1^ct_2^d)$. 
Throughout this note, we assume that the total space $M$ of an $S^1$-bundle $p\colon M\to B$ is orientable (the base space $B$ may be non-orientable). 

Let $B$ be a (compact) $2$-manifold of genus $g$ with $d$ boundary components $\partial B_1,\dots, \partial B_d\cong S^1$, with the convention that $g<0$ if $B$ is non-orientable. 
Let $p\colon M\to B$ be an $S^1$-bundle. 
For a boundary component $\partial M_i:=p^{-1}(\partial B_i)$ of $M$, an orientation preserving homeomorphism $\tau_i\colon S^1\times S^1\to\partial M_i$ is called a \textit{trivialization} of $\partial M_i$ if $p\circ\tau_i$ is a projection to either the first or second component of $S^1\times S^1$. 
Unless otherwise noted, we assume that $p\circ\tau_i$ is the projection to the first component: 
\[ p\circ \tau_i(t_1,t_2)=
t_1\in S^1\cong \partial {\red B_i}. 
\]
For an $S^1$-bundle $p\colon M\to B$, we fix a trivialization for each boundary component of $M$. 
Then the {\red \textit{Euler number} $e$ of the $S^1$-bundle $p$ is the cross-section obstruction (cf.\ \cite[p.\ 59]{HNK1971})}. 
In \cite{waldhausen_1967}, several properties of $S^1$-bundles over $2$-manifolds are proved. 
In Appendix~\ref{sec:S1-bdls}, results which we need in this note are stated.

%%%%%%%%%%%%%%%%%%%
\subsection{Graph manifolds}
%%%%%%%%%%%%%%%%%%%

We recall \textit{graph manifolds}, which are defined and classified (with minor exceptions) by Waldhausen \cite{waldhausen_1967}. 
In particular, we state several facts for graph manifolds with boundary. 
(In \cite{waldhausen_1967}, the same facts are proved for graph manifolds without boundary except for several cases.) 

Let $M$ be a $3$-manifold, and let $\TT:=T_1\cup\dots\cup T_n\subset M^\circ$ ($n\geq0$) be a disjoint union of tori. 
Let $U(\TT)$ be a regular neighborhood of $\TT$ in $M$. 
{\red We call $\TT$} a \textit{graph structure} for $M$ if every {\red connected} component of $M\setminus U(\TT)^\circ$ is homeomorphic to an $S^1$-bundle over $2$-manifolds. 
A $3$-manifold with a graph structure is called a \textit{graph manifold}. 

A graph structure $\TT$ of a $3$-manifold $M$ is said to be \textit{reduced} if  it does not satisfy {\red any of the ten conditions} in \cite[Definition~6.2]{waldhausen_1967}. We omit to state the ten conditions, but that $M\setminus U(\TT)^\circ$ has no component homeomorphic to $I\times S^1\times S^1$. 
A $3$-manifold with a fixed reduced graph structure is called a \textit{reduced graph manifold}. 
For reduced graph manifolds with boundary, Waldhausen proved the following theorem. 

\begin{thm}[cf. {\cite[Theorem~8.1]{waldhausen_1967}}]\label{thm:Wald-1}
For each $i=1,2$, let $M_i$ be a reduced graph manifold with boundary and a graph structure $\TT_i$. 
If there is a homeomorphism $h\colon M_1\to M_2$, then there is an isotopy $\gth_u\colon M_1\to M_2$ $(0\leq u\leq 1)$ such that $\gth_0=h$ and $\gth_1(\TT_1)=\TT_2$. 
\end{thm}

For a reduced graph manifold $M$, Waldhausen defined a weighted graph $\Gamma_\W(M)$, which is called \textit{W-graph} in \cite{neumann1981}, and proved Theorem~\ref{thm:Wald-2} below. 
The W-graph $\Gamma_\W(M)$ consists of the following data: 
 
 \begin{itemize}
	\item There is a bijection from $\mcV(\Gamma_\W(M))$ to the set of connected components of $M\setminus U(\TT)^\circ$, and let $M_v$ be the component corresponding to the vertex $v$ for each $v\in\mcV(\Gamma_\W(M))$;
	\item if $M_v\not\cong S^1\times {\red \DD^2}$ for $v\in\mcV(\Gamma_\W(M))$, $M_v$ has the unique $S^1$-bundle structure $p_v\colon M_v\to B_v$ up to isotopy by Theorem~\ref{thm:S1-bdl_gene}, and $v$ is assigned a triple of integers $(g_v,r_v,s_v)$  where $g_v$ is the genus of $B_v$, $r_v$ is the number of boundary components of $M$ contained in $M_v$, and $s_v$ is the {\red Euler} number of $p_v\colon M_v\to B_v$ if $r_v=0$, $s_v$ vanishes if $r_v>0$;
	\item each $v\in\mcV(\Gamma_\W(M))$ with $M_v\cong S^1\times {\red\DD^2}$ is left unweighted; 
	\item there is a bijection from $\mcY(\Gamma_\W(M))$ to the set of components of $\partial U(\TT)$, write $y\mapsto T_y\subset \partial U(\TT)$ for each $y\in\mcY(\Gamma_\W(M))$, satisfying 
\begin{itemize}
	\item $T_y\subset M_{t(y)}$ for each $y\in\mcY(\Gamma_\W(M))$; 
	\item for each $T\subset \TT$, there is $y\in\mcY(\Gamma_\W(M))$ such that $\partial U(T)=T_y\cup T_{\overline{y}}$
\end{itemize}
(note that $M$ can be regarded as the manifold given by gluing $M_{t(y)}$ and $M_{t(\overline{y})}$ via certain homeomorphisms $T_y\to T_{\overline{y}}$ for $y\in\mcY(\Gamma_\W(M))$ since $U(T)\cong I\times S^1\times S^1$ for each component $T\subset \TT$);
	\item each $y\in\mcY(\Gamma_\W(M))$ is assigned a pair $(\alpha_y,\beta_y)$ of integers satisfying 
	\begin{align*}
		&\gcd(\alpha_y,\beta_y)=1, \qquad 0\leq\beta_y<\alpha_y, \qquad \alpha_{\overline{y}}=\alpha_{y}, \qquad 
		\beta_{\overline{y}}\beta_y\equiv 1 \pmod{\alpha_y}
	\end{align*}
	and $\beta_y>0$ {if $y$ is adjacent to an unweighted vertex,} which is uniquely determined by the pasting of $M_{o(y)}$ and $M_{t(y)}$ at $T_y$; 
	\item each edge $y\in \mcY(\Gamma_\W(M))$ is assigned a sign $\varepsilon_y\in\{\pm1\}$, which depends on orientations of fibers of $M_{o(y)}$ and $M_{t(y)}$. Let $\Gamma_\W^\ast(M)$ be the full subgraph of $\Gamma_\W(M)$ given by removing $v\in\mcV(\Gamma_\W(M))$ with $g_v<0$. Then $\Gamma_\W(M)$ is assigned the map $\varepsilon_M\colon H_1(\Gamma_\W^\ast(M))\to\ZZ/2\ZZ$ defined by $\varepsilon_M(C)$ equal to the number modulo $2$ of $y\in\mcY(C)$ with $\varepsilon_y=-1$ for any cycle $C$ of $\Gamma_\W^\ast(M)$, which is uniquely determined. 
\end{itemize}
Moreover, the \textit{equivalence} of W-graphs is also defined in \cite{waldhausen_1967}. 
Under our notation of graphs, 
for reduced graph manifolds $M_i$ with boundary, W-graphs $\Gamma_\W(M_1)$ and $\Gamma_\W(M_2)$ are \textit{equivalent} if they satisfy one of the following conditions; 
\begin{itemize}
	\item there is an isomorphism $\varphi\colon \Gamma_\W(M_1)\to \Gamma_\W(M_2)$ of graphs preserving weights $(g_v,r_v,s_v)$ of $v\in\mcV(\Gamma_\W(M_1))$ and $(\alpha_y,\beta_y)$ of $y\in\mcY(\Gamma_\W(M_1))$, and $\varphi$ is compatible with $\varepsilon_{M_i}\colon  H_1(\Gamma_\W^\ast(M_i))\to\ZZ/2\ZZ$, i.e., $\varepsilon_{M_1}=\varepsilon_{M_2}\circ\varphi_\ast$; 
	\item the pair of $\Gamma_\W(M_1)$ and $\Gamma_\W(M_2)$ is that of the following graphs; 
	\begin{align}\label{eq:W-graphs} \begin{aligned} \begin{tikzpicture}
		\coordinate (a1) at (0,0);
		\coordinate (a2) at (-1.5,0);
		\coordinate (a3) at (1.5,0);
		\coordinate (b1) at (5,0);
		\draw [fill] (a1) circle [radius=2pt] node [below] {\footnotesize $(0,1,-)$};
		\draw [fill] (a2) circle [radius=2pt] node [above] {\footnotesize $-$};
		\draw [fill] (a3) circle [radius=2pt] node [above] {\footnotesize $-$};
		\draw [-{Classical TikZ Rightarrow[length=4pt]}] (a1) -- (-1,0);
		\draw [-{Classical TikZ Rightarrow[length=4pt]}] (a1) -- (1,0);
		\draw (a1) -- node [above] {\footnotesize$(2,1)$} (a2); 
		\draw (a1) -- node [above] {\footnotesize$(2,1)$} (a3); 
		\draw [fill] (b1) circle [radius=2pt] node [below] {\footnotesize $(-1,1,-)$};
		\node at (3,0) {and};
	\end{tikzpicture} \end{aligned} \end{align}
\end{itemize}
We call the isomorphism $\varphi\colon \Gamma_\W(M_1)\to\Gamma_\W(M_2)$ above an \textit{equivalence map of W-graphs}. 
{\red The proof of \cite[Theorem~9.4]{waldhausen_1967}} shows the following theorem. 

\begin{thm}[{cf.\ \cite[Theorem~9.4]{waldhausen_1967}}]\label{thm:Wald-2}
	Let $M_i$, $i=1,2$, be two oriented reduced graph manifolds with boundary and reduced graph structure $\TT_i$. 
	Then $M_1$ and $M_2$ are orientation-preserving homeomorphic if and only if the W-graphs $\Gamma_\W(M_1)$ and $\Gamma_\W(M_2)$ are equivalent. 
	
	Moreover, if $h\colon M_1\to M_2$ {\red is} an orientation preserving homeomorphism, and the pair of $\Gamma_\W(M_i)$ is not that of {\rm (\ref{eq:W-graphs})}, then there is an isotopy $\gth_u\colon M_1\to M_2$ $(0\leq u\leq 1)$ such that $\gth_0=h$ and $\gth_1$ satisfies the following conditions:
	\begin{enumerate}
		\item $\gth_1(\TT_1)=\TT_2$; 
		\item $\gth_1$ induces an equivalence map $\varphi\colon \Gamma_\W(M_1)\to\Gamma_\W(M_2)$ of W-graphs such that $\gth_1(M_{1,v})=M_{2,\varphi(v)}$ for each $v\in\mcV(\Gamma_\W(M_1))$, where $M_{i,v_i}$ is the component of $M_i\setminus U(\TT_i)^\circ$ corresponding to $v_i\in\mcV(\Gamma_\W(M_i))$; 
		\item $\gth_1$ maps every fiber of the $S^1$-bundle $p_{1,v}\colon M_{1,v}\to B_{1,v}$ to that of $p_{2,\varphi(v)}\colon M_{2,\varphi(v)}\to B_{2,\varphi(v)}$ if $M_{1,v}\not\cong {\red \DD^2}\times S^1$ for $v\in\mcV(\Gamma_\W(M_1))$; 
	\end{enumerate}
\end{thm}

%%%%%%%%%%%%%%%%%%%
\subsection{Plumbing graphs}\label{sec:plumbing_graph}
%%%%%%%%%%%%%%%%%%%

Neumann \cite{neumann1981} defined \textit{plumbed $3$-manifolds} corresponding to \textit{plumbing graphs} based on Waldhausen's work \cite{waldhausen_1967}. 
In the appendix of \cite{neumann1981}, he also defined \textit{decorated plumbing graphs} to describe \textit{boundary framed graph manifolds}, which are graph manifolds with fixed trivialization of each boundary component. 
In this note, we associate a decorated plumbing graph with the graph manifold constructed from the induced plumbing graph as mentioned in \cite[p.339]{neumann1981}, which is called the \textit{plumbed manifold}. 
We recall the (connected) decorated plumbing graph and such a graph manifold.

A \textit{decorated plumbing graph} $\Gamma_\dplumb=(\Gamma, \bm{g}, \bm{e}, \sign; \Ver,\AH)$ is a finite connected graph $\Gamma$ together with a partition $\mcV(\Gamma)=\Ver(\Gamma)\sqcup\AH(\Gamma)$ such that $\bm{d}(v)\leq 1$ for each $v\in\AH(\Gamma)$, and maps
\begin{align*}
	 \bm{g} &\colon  \Ver(\Gamma) \longrightarrow \ZZ, &
	 \bm{e} &\colon  \Ver(\Gamma) \longrightarrow \ZZ, &
	 \sign &\colon  \mcY(\Gamma) \longrightarrow \{\pm1\}
\end{align*}
satisfying $\sign(y)=\sign(\overline{y})$ for any $y\in\mcY(\Gamma)$. 
A vertex in $\AH(\Gamma)$ is called a \textit{boundary vertex}, and drawn in the graph as an arrow head $\arrowhead$. 
An edge $y\in\mcY(\Gamma)$ is called a \textit{$(+)$-edge} if $\sign(y)=+1$, and a \textit{$(-)$-edge} otherwise. 
In the case where $\AH(\Gamma)=\emptyset$, $\Gamma_\dplumb$ is a plumbing graph, and we denote it by $\Gamma_\plumb=(\Gamma,\bm{g}, \bm{e}, \sign)$.

Let $\Gamma_\dplumb=(\Gamma,\bm{g},\bm{e}, \sign; \Ver, \AH)$ be a decorated plumbing graph with $\Ver(\Gamma)\ne\emptyset$. 
The \textit{plumbed manifold} $\PM(\Gamma_\dplumb)$ for $\Gamma_\dplumb$ is the graph manifold constructed as follows: 
\begin{enumerate}[label=(\Roman*)]
	\item For $v\in\Ver(\Gamma)$, let $B_v$ be a compact $2$-manifold of genus $\bm{g}(v)$ with $\bm{d}(v)$ boundary components $\partial B_{y}$, labeled by edges $y\in\mcY(\Gamma)$ with $t(y)=v$; and let 
	$p_v\colon  M_v\to B_v$ 
	be an $S^1$-bundle, where, for each $y\in\mcY(\Gamma)$ with $t(y)=v$, the boundary component $T_{y}:=p_{v}^{-1}(\partial B_y)$  
	of $M_{v}$ has a fixed trivialization 
	$ \tau_y\colon S^1\times S^1\to T_y $
	such that the {\red Euler} number of $p_v$ is $\bm{e}(v)$; 
	\item the plumbed manifold $\PM(\Gamma_\dplumb)$ is constructed by gluing $M_{t(y)}$ and $M_{t(\overline{y})}$ at $T_y\subset M_{t(y)}$ and $T_{\overline{y}}\subset M_{t(\overline{y})}$ by the homeomorphism
	\begin{align}\label{eq:gluing torus} 
	T_{\overline{y}}\xrightarrow{\ \tau_{\overline{y}\ }^{-1}}  S^1\times S^1 \xrightarrow{\sign(y) J} S^1\times S^1\xrightarrow{\ \tau_y\ } T_y 	
	\end{align}

	for each $y\in\mcY$ with $o(y),t(y)\in\Ver(\Gamma)$, where $J:=(\begin{smallmatrix} 0&1\\1&0 \end{smallmatrix})$. 
\end{enumerate}
For $v\in\Ver(\Gamma)$, we call $p_v\colon M_v\to B_v$ (or $M_v$) the \textit{$S^1$-bundle of $v$}, $B_v$ the \textit{base $2$-manifold of $v$}, a fiber of $p_v$ a \textit{fiber of $v$}. 
An edge $y\in\mcY(\Gamma)$ with $o(y)\in\AH(\Gamma)$ corresponds to a boundary component $T_y$ of $\PM(\Gamma_\dplumb)$.  
Since the gluing maps reverse orientation, $\PM(\Gamma_\dplumb)$ inherits compatible orientations from all pieces $M_v$. 
The number of boundary components of $\PM(\Gamma_\dplumb)$ is equal to $\#\AH(\Gamma)$. 
By abuse of notation, let $T_y$ denote the resulting torus in $\PM(\Gamma_\dplumb)$ by gluing $T_y\subset M_{t(y)}$ and $T_{\overline{y}}\subset M_{t(\overline{y})}$, hence $T_y=T_{\overline{y}}$ in $\PM(\Gamma_\dplumb)$. 

In this note, we consider $\PM(\Gamma_\dplumb)$ as a closed submanifold of a graph manifold $M$. 
An edge $y\in\mcY(\Gamma)$ with $t(y)\in\AH(\Gamma)$ corresponds to the boundary component $T_y$ of $M\setminus\PM(\Gamma_\dplumb)^\circ$ which is identified with $T_{\overline{y}}\subset M_{o(t)}\subset\PM(\Gamma)$ in $M$, and $t(y)\in\AH(\Gamma)$ corresponds to an $S^1$-bundle part $M_{t(y)}$ of $M\setminus\PM(\Gamma)^\circ$ such that $T_y\subset M_{t(y)}$.

\begin{rem}\label{rem:original}
	Here we do not give the definition of plumbing graph introduced in \cite{neumann1981}. 
	Plumbed manifolds having boundary without trivialization are considered by using the plumbing graphs in \cite{neumann1981}. 
	In \cite[p.339]{neumann1981}, it is mentioned that each boundary vertex $v\in\AH(\Gamma)$ of a decorated plumbing graph $\Gamma_\dplumb$ can be regarded as a vertex $v$ of a plumbing graph with weights $\bm{g}(v)=0$, $\bm{e}(v)=0$ and $r_v=1$, where $r_v$ is the number of ``free" boundary components of the $S^1$-bundle $M'_v$, i.e., $v\in\AH(\Gamma)$ corresponds to $M_v'\cong I\times S^1\times S^1$ one of whose boundary components does not have trivialization. 
	In this note, we can regard such $M_v'$ as a submanifold of the above $S^1$-bundle part $M_v\subset M\setminus\PM(\Gamma)$.
\end{rem}

\begin{rem}\label{rem:edge weight}
For the aim of this note, it is enough to consider decorated plumbing graphs with $\sign(y)=+1$ for any $y\in\mcY(\Gamma)$ in many cases. 
If we assume that $\sign(y)=+1$ for each $y\in\mcY(\Gamma)$, then, by omitting the map $\sign\colon\mcY(\Gamma)\to\{\pm 1\}$, we write $\Gamma_\dplumb=(\Gamma, \bm{g}, \bm{e}; \Ver, \AH)$ or $\Gamma_\plumb=(\Gamma,\bm{g},\bm{e})$ if $\AH(\Gamma)=\emptyset$. 
\end{rem}

\medskip

\textit{Notation}. 
We follow \cite[Appendix]{neumann1981} to describe a decorated plumbing graph. 
Namely, a vertex $v\in\mcV(\Gamma)$ is written as follows:
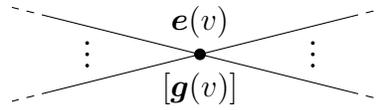
\[ \begin{tikzpicture}[yscale=1]
	\draw [fill] (0,0) circle [radius=2pt] node [above=2pt] {$\bm{e}(v)$} node [below=3pt] {$[\bm{g}(v)]$};
	\draw (0,0) -- (-2,0.32);
	\draw (0,0) -- (-2,-0.32);
	\draw (0,0) -- (2,0.32);
	\draw (0,0) -- (2,-0.32);
	\draw [dashed] (0,0) -- (-2.5,0.4);
	\draw [dashed] (0,0) -- (-2.5,-0.4);
	\draw [dashed] (0,0) -- (2.5,-0.4);
	\draw [dashed] (0,0) -- (2.5,0.4);
	\draw [fill] (-1.7,0) circle [radius=0.5pt];
	\draw [fill] (-1.7,0.15) circle [radius=0.5pt];
	\draw [fill] (-1.7,-0.15) circle [radius=0.5pt];
	\draw [fill] (1.7,0) circle [radius=0.5pt];
	\draw [fill] (1.7,0.15) circle [radius=0.5pt];
	\draw [fill] (1.7,-0.15) circle [radius=0.5pt];
\end{tikzpicture} \] 
We omit the weight of $\{y,\overline{y}\}$ if $\sign(y)=+1$, and we write $-$ if $\sign(y)=-1$. 
If $\bm{g}(v)=0$, we may omit it. 
For example, the left decorated plumbing graph is ``shorthand" for the right plumbing graph in the following graphs:
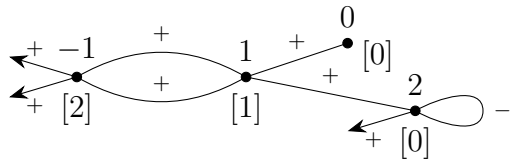
\[ \begin{tikzpicture}[yscale=0.8, xscale=0.7]
	\coordinate (a1) at (0,0);
	\coordinate (a2) at (2.5,0);
	\coordinate (a3) at (5,-0.5);
	\coordinate (a4) at (4,0.5);
	\coordinate (a5) at (4,-0.8);
	\coordinate (a6) at (-1, 0.3);
	\coordinate (a7) at (-1, -0.3);
	\draw [fill] (a1) circle [radius=2pt] node [below=2pt] {$[2]$} node [above=2pt] {$-1$};
	\draw [fill] (a2) circle [radius=2pt] node [below=2pt] {$[1]$} node [above=2pt] {$1$};
	\draw [fill] (a3) circle [radius=2pt] node [above=2pt] {$2$};
	\draw [fill] (a4) circle [radius=2pt] node [above=2pt] {$0$};
	\draw (a1) to [out=30, in=150] (a2);
	\draw (a1) to [out=-30, in=-150] (a2);
	\draw (a2) to (a3);
	\draw (a3) to [out=30, in=90] (6,-0.5) node [right] {\tiny $-$} to [out=-90, in=-30]  (a3);
	\draw [-{Stealth[length=7pt]}] (a1) to (a6);
	\draw [-{Stealth[length=7pt]}] (a1) to (a7);
	\draw (a2) to (a4);
	\draw [-{Stealth[length=7pt]}] (a3) to (a5);

	\begin{scope}[xshift=8.5cm]
			\coordinate (a1) at (0,0);
	\coordinate (a2) at (2.5,0);
	\coordinate (a3) at (5,-0.5);
	\coordinate (a4) at (4,0.5);
	\coordinate (a5) at (4,-0.8);
	\coordinate (a6) at (-1, 0.3);
	\coordinate (a7) at (-1, -0.3);
	\draw [fill] (a1) circle [radius=2pt] node [below=2pt] {$[2]$} node [above=2pt] {$-1$};
	\draw [fill] (a2) circle [radius=2pt] node [below=2pt] {$[1]$} node [above=2pt] {$1$};
	\draw [fill] (a3) circle [radius=2pt] node [below=2pt] {$[0]$} node [above=2pt] {$2$};
	\draw [fill] (a4) circle [radius=2pt] node [below=5pt, right=1pt] {$[0]$} node [above=2pt] {$0$};
	\draw (a1) to [out=30, in=150] node [above] {\tiny $+$} (a2);
	\draw (a1) to [out=-30, in=-150] node [above] {\tiny $+$} (a2);
	\draw (a2) to node [above] {\tiny $+$} (a3);
	\draw (a3) to [out=30, in=90] (6,-0.5) node [right] {\tiny $-$} to [out=-90, in=-30]  (a3);
	\draw [-{Stealth[length=7pt]}] (a1) to node [left=3pt, above] {\tiny $+$} (a6);
	\draw [-{Stealth[length=7pt]}] (a1) to node [left=3pt, below] {\tiny $+$} (a7);
	\draw (a2) to node [above] {\tiny $+$} (a4);
	\draw [-{Stealth[length=7pt]}] (a3) to node [left=3pt, below] {\tiny $+$} (a5);
	\end{scope}
\end{tikzpicture} \]
%is ``shorthand" for the following plumbing graph:
%\[ \begin{tikzpicture}[scale=0.9]
%	\coordinate (a1) at (0,0);
%	\coordinate (a2) at (2.5,0);
%	\coordinate (a3) at (5,-0.5);
%	\coordinate (a4) at (4,0.5);
%	\coordinate (a5) at (4,-0.8);
%	\coordinate (a6) at (-1, 0.3);
%	\coordinate (a7) at (-1, -0.3);
%	\draw [fill] (a1) circle [radius=2pt] node [below=2pt] {$[2]$} node [above=2pt] {$-1$};
%	\draw [fill] (a2) circle [radius=2pt] node [below=2pt] {$[1]$} node [above=2pt] {$1$};
%	\draw [fill] (a3) circle [radius=2pt] node [below=2pt] {$[0]$} node [above=2pt] {$2$};
%	\draw [fill] (a4) circle [radius=2pt] node [below=5pt, right=1pt] {$[0]$} node [above=2pt] {$0$};
%	\draw (a1) to [out=30, in=150] node [above] {\tiny $+$} (a2);
%	\draw (a1) to [out=-30, in=-150] node [above] {\tiny $+$} (a2);
%	\draw (a2) to node [above] {\tiny $+$} (a3);
%	\draw (a3) to [out=30, in=90] (6,-0.5) node [right] {\tiny $-$} to [out=-90, in=-30]  (a3);
%	\draw [-{Stealth[length=7pt]}] (a1) to node [left=3pt, above] {\tiny $+$} (a6);
%	\draw [-{Stealth[length=7pt]}] (a1) to node [left=3pt, below] {\tiny $+$} (a7);
%	\draw (a2) to node [above] {\tiny $+$} (a4);
%	\draw [-{Stealth[length=7pt]}] (a3) to node [left=3pt, below] {\tiny $+$} (a5);
%\end{tikzpicture} \]
A \textit{chain of length $k$} in $\Gamma_\dplumb$ is any part $\gtc$ of $\Gamma_\dplumb$ of the following form with arbitrary edge signs (after relabeling vertices)
\begin{align*} 
\begin{aligned}
&\begin{tikzpicture}
	\coordinate (a10) at (-1.5,0);
	\coordinate (a11) at (-1,0);
	\coordinate (a1) at (0,0); 
	\coordinate (a2) at (1.5,0);
	\coordinate (a20) at (2.3,0);
	\coordinate (a21) at (3.7,0);
	\coordinate (a3) at (4.5,0);
	\coordinate (a30) at (5.5,0);
	\coordinate (a31) at (6,0);
	\node at (8,0.1) {$(k\geq 0)$, \ or};
	\draw[fill] (a1) circle [radius=2pt] node [above] {$e_1$};
	\draw[fill] (a2) circle [radius=2pt] node [above] {$e_2$};
	\draw[fill] (a3) circle [radius=2pt] node [above] {$e_k$};
	\draw[dashed] (a10) to (a11); 
	\draw (a11) to (a20);
	\draw[dashed] (a20) to (a21); 
	\draw (a21) to (a30);
	\draw[dashed] (a30) to (a31); 
\end{tikzpicture}
	\\%%%%%%%%%%%%%%%%%%%%%%%%%%%%%%%%%
&\begin{tikzpicture}
	\coordinate (b10) at (-1.5,-1.3);
	\coordinate (b11) at (-1,-1.3);
	\coordinate (b1) at (0,-1.3); 
	\coordinate (b2) at (1.5,-1.3);
	\coordinate (b20) at (2.3,-1.3);
	\coordinate (b21) at (3.7,-1.3);
	\coordinate (b3) at (4.5,-1.3);
	\node at (8,-1.2) {$(k\geq 1)$,};
	\draw[fill] (b1) circle [radius=2pt] node [above] {$e_1$};
	\draw[fill] (b2) circle [radius=2pt] node [above] {$e_2$};
	\draw[fill] (b3) circle [radius=2pt] node [above] {$e_k$};
	\draw[dashed] (b10) to (b11); 
	\draw (b11) to (b20);
	\draw[dashed] (b20) to (b21); 
	\draw (b21) to (b3);
\end{tikzpicture}
\end{aligned}
\end{align*}
say $v_i$ the vertex on the above parts with $\bm{e}(v_i)=e_i$. 
The chain $\gtc$ above is said to be  \textit{incident to a boundary vertex} if either $v_1$ or $v_k$ is adjacent to a boundary vertex. 
An edge $y\in\mcY(\Gamma)$ is regarded as a chain of length $0$. 
A chain is \textit{maximal} if it can be included in no larger chain. 

\medskip

In \cite[Proposition~2.1]{neumann1981}, nine operations (R0) -- (R8) and their inverses to plumbing graphs are given, which do not change the diffeomorphism type of their plumbed $3$-manifolds (hence also homeomorphism type). 
In \cite[Appendix]{neumann1981}, (R0) -- (R7) are modified to \textit{F-calculus} for applying decorated plumbing graphs ((R8) is removed). Moreover, the \textit{F-normal forms} of decorated plumbing graphs are defined, which any decorated plumbing graph can be reduced to F-normal form using only F-calculus and their inverses.

\begin{rem}\label{rem:operation}
In this paper, we use the following five of F-calculus to obtain F-normal forms: 
%\[ \input{blowup_extrusion.tex} \]
\begin{align*} 
& \begin{tikzpicture}[scale=0.7]
	\coordinate (c1) at (0,5);
	\coordinate (c2) at (1.5,5);
	\coordinate (c3) at (3,5);
	\coordinate (c41) at (-1,5.3);
	\coordinate (c42) at (-1,4.7);
	\coordinate (c43) at (-1.5,5.45);
	\coordinate (c44) at (-1.5,4.55);
	\coordinate (c51) at (4,5.3);
	\coordinate (c52) at (4,4.7);
	\coordinate (c53) at (4.5,5.45);
	\coordinate (c54) at (4.5,4.55);

	\coordinate (d1) at (8,5);
	\coordinate (d2) at (9.5,5);
	\coordinate (d3) at (11,5);
	\coordinate (d41) at (7,5.3);
	\coordinate (d42) at (7,4.7);
	\coordinate (d43) at (6.5,5.45);
	\coordinate (d44) at (6.5,4.55);
	\coordinate (d51) at (12,5.3);
	\coordinate (d52) at (12,4.7);
	\coordinate (d53) at (12.5,5.45);
	\coordinate (d54) at (12.5,4.55);
	
	\draw[fill] (c1) circle [radius=2pt] node [above=2pt] {$e_i$} node [below=2.5pt] {$[g_i]$};
	\draw[fill] (c2) circle [radius=2pt] node [above=2pt] {$-1$};
%	\draw[fill] (c3) circle [radius=2pt] node [above=2pt] {$e_j$} node [below=2.5pt] {$[g_j]$};
	\draw (c1) -- (c2); %-- (c3);
	\draw (c1) -- (c41);
	\draw (c1) -- (c42);
	\draw[dashed] (c41) -- (c43);
	\draw[dashed] (c42) -- (c44);
%	\draw (c3) -- (c51);
%	\draw (c3) -- (c52);
%	\draw[dashed] (c51) -- (c53);
%	\draw[dashed] (c52) -- (c54);

	\draw [->] (5,5) -- (6,5);

	\draw[fill] (d1) circle [radius=2pt] node [above=2pt] {$e_i+1$} node [below=2.5pt] {$[g_i]$};
%	\draw[fill] (d3) circle [radius=2pt] node [above=2pt] {$e_j+1$} node [below=2.5pt] {$[g_j]$};
%	\draw (d1) -- (d3);
	\draw (d1) -- (d41);
	\draw (d1) -- (d42);
	\draw[dashed] (d41) -- (d43);
	\draw[dashed] (d42) -- (d44);
%	\draw (d3) -- (d51);
%	\draw (d3) -- (d52);
%	\draw[dashed] (d51) -- (d53);
%	\draw[dashed] (d52) -- (d54);

%	\draw [fill] (3.9,5.15) circle [radius=0.5pt];
%	\draw [fill] (3.9,5) circle [radius=0.5pt];
%	\draw [fill] (3.9,4.85) circle [radius=0.5pt];

	\draw [fill] (-0.9,5.15) circle [radius=0.5pt];
	\draw [fill] (-0.9,5) circle [radius=0.5pt];
	\draw [fill] (-0.9,4.85) circle [radius=0.5pt];
	
%	\draw [fill] (11.9,5.15) circle [radius=0.5pt];
%	\draw [fill] (11.9,5) circle [radius=0.5pt];
%	\draw [fill] (11.9,4.85) circle [radius=0.5pt];

	\draw [fill] (7.1,5.15) circle [radius=0.5pt];
	\draw [fill] (7.1,5) circle [radius=0.5pt];
	\draw [fill] (7.1,4.85) circle [radius=0.5pt];

	\node at (-2.7,5) {($\mathrm{R1}^+_0$)};

\end{tikzpicture}\\
&\, \begin{tikzpicture}[scale=0.7]
	\coordinate (c1) at (0,5);
	\coordinate (c2) at (1.5,5);
	\coordinate (c3) at (3,5);
	\coordinate (c41) at (-1,5.3);
	\coordinate (c42) at (-1,4.7);
	\coordinate (c43) at (-1.5,5.45);
	\coordinate (c44) at (-1.5,4.55);
	\coordinate (c51) at (4,5.3);
	\coordinate (c52) at (4,4.7);
	\coordinate (c53) at (4.5,5.45);
	\coordinate (c54) at (4.5,4.55);

	\coordinate (d1) at (8,5);
	\coordinate (d2) at (9.5,5);
	\coordinate (d3) at (11,5);
	\coordinate (d41) at (7,5.3);
	\coordinate (d42) at (7,4.7);
	\coordinate (d43) at (6.5,5.45);
	\coordinate (d44) at (6.5,4.55);
	\coordinate (d51) at (12,5.3);
	\coordinate (d52) at (12,4.7);
	\coordinate (d53) at (12.5,5.45);
	\coordinate (d54) at (12.5,4.55);
	
	\draw[fill] (c1) circle [radius=2pt] node [above=2pt] {$e_i$} node [below=2.5pt] {$[g_i]$};
	\draw[fill] (c2) circle [radius=2pt] node [above=2pt] {$-1$};
	\draw[fill] (c3) circle [radius=2pt] node [above=2pt] {$e_j$} node [below=2.5pt] {$[g_j]$};
	\draw (c1) -- (c2) -- (c3);
	\draw (c1) -- (c41);
	\draw (c1) -- (c42);
	\draw[dashed] (c41) -- (c43);
	\draw[dashed] (c42) -- (c44);
	\draw (c3) -- (c51);
	\draw (c3) -- (c52);
	\draw[dashed] (c51) -- (c53);
	\draw[dashed] (c52) -- (c54);

	\draw [->] (5,5) -- (6,5);

	\draw[fill] (d1) circle [radius=2pt] node [above=2pt] {$e_i+1$} node [below=2.5pt] {$[g_i]$};
	\draw[fill] (d3) circle [radius=2pt] node [above=2pt] {$e_j+1$} node [below=2.5pt] {$[g_j]$};
	\draw (d1) -- (d3);
	\draw (d1) -- (d41);
	\draw (d1) -- (d42);
	\draw[dashed] (d41) -- (d43);
	\draw[dashed] (d42) -- (d44);
	\draw (d3) -- (d51);
	\draw (d3) -- (d52);
	\draw[dashed] (d51) -- (d53);
	\draw[dashed] (d52) -- (d54);

	\draw [fill] (3.9,5.15) circle [radius=0.5pt];
	\draw [fill] (3.9,5) circle [radius=0.5pt];
	\draw [fill] (3.9,4.85) circle [radius=0.5pt];

	\draw [fill] (-0.9,5.15) circle [radius=0.5pt];
	\draw [fill] (-0.9,5) circle [radius=0.5pt];
	\draw [fill] (-0.9,4.85) circle [radius=0.5pt];
	
	\draw [fill] (11.9,5.15) circle [radius=0.5pt];
	\draw [fill] (11.9,5) circle [radius=0.5pt];
	\draw [fill] (11.9,4.85) circle [radius=0.5pt];

	\draw [fill] (7.1,5.15) circle [radius=0.5pt];
	\draw [fill] (7.1,5) circle [radius=0.5pt];
	\draw [fill] (7.1,4.85) circle [radius=0.5pt];

	\node at (-2.7,5) {($\mathrm{R1^{+}}$)};

\end{tikzpicture} \\
&\, \begin{tikzpicture}[scale=0.7]
	\coordinate (a1) at (0,2.5);
	\coordinate (a2) at (1.5,2.5);
	\coordinate (a3) at (3,2.5);
	\coordinate (a41) at (-1,2.8);
	\coordinate (a42) at (-1,2.2);
	\coordinate (a43) at (-1.5,2.95);
	\coordinate (a44) at (-1.5,2.05);
	\coordinate (a51) at (4,2.8);
	\coordinate (a52) at (4,2.2);
	\coordinate (a53) at (4.5,2.95);
	\coordinate (a54) at (4.5,2.05);

	\coordinate (b1) at (8,2.5);
	\coordinate (b2) at (9.5,2.5);
	\coordinate (b3) at (11,2.5);
	\coordinate (b41) at (7,2.8);
	\coordinate (b42) at (7,2.2);
	\coordinate (b43) at (6.5,2.95);
	\coordinate (b44) at (6.5,2.05);
	\coordinate (b51) at (12,2.8);
	\coordinate (b52) at (12,2.2);
	\coordinate (b53) at (12.5,2.95);
	\coordinate (b54) at (12.5,2.05);
	
	\draw[fill] (a1) circle [radius=2pt] node [above=2pt] {$e_i$} node [below=2.5pt] {$[g_i]$};
	\draw[fill] (a2) circle [radius=2pt] node [above=2pt] {$1$};
	\draw[fill] (a3) circle [radius=2pt] node [above=2pt] {$e_j$} node [below=2.5pt] {$[g_j]$};
	\draw (a1) -- node [above] {\tiny $-$} (a2) -- node [above] {\tiny $-$} (a3);
	\draw (a1) -- (a41);
	\draw (a1) -- (a42);
	\draw[dashed] (a41) -- (a43);
	\draw[dashed] (a42) -- (a44);
	\draw (a3) -- (a51);
	\draw (a3) -- (a52);
	\draw[dashed] (a51) -- (a53);
	\draw[dashed] (a52) -- (a54);

	\draw [->] (5,2.5) -- (6,2.5);

	\draw[fill] (b1) circle [radius=2pt] node [above=2pt] {$e_i-1$} node [below=2.5pt] {$[g_i]$};
	\draw[fill] (b3) circle [radius=2pt] node [above=2pt] {$e_j-1$} node [below=2.5pt] {$[g_j]$};
	\draw (b1) -- node [above] {\tiny $-$} (b3);
	\draw (b1) -- (b41);
	\draw (b1) -- (b42);
	\draw[dashed] (b41) -- (b43);
	\draw[dashed] (b42) -- (b44);
	\draw (b3) -- (b51);
	\draw (b3) -- (b52);
	\draw[dashed] (b51) -- (b53);
	\draw[dashed] (b52) -- (b54);

	\draw [fill] (3.9,2.65) circle [radius=0.5pt];
	\draw [fill] (3.9,2.5) circle [radius=0.5pt];
	\draw [fill] (3.9,2.35) circle [radius=0.5pt];

	\draw [fill] (-0.9,2.65) circle [radius=0.5pt];
	\draw [fill] (-0.9,2.5) circle [radius=0.5pt];
	\draw [fill] (-0.9,2.35) circle [radius=0.5pt];
	
	\draw [fill] (11.9,2.65) circle [radius=0.5pt];
	\draw [fill] (11.9,2.5) circle [radius=0.5pt];
	\draw [fill] (11.9,2.35) circle [radius=0.5pt];

	\draw [fill] (7.1,2.65) circle [radius=0.5pt];
	\draw [fill] (7.1,2.5) circle [radius=0.5pt];
	\draw [fill] (7.1,2.35) circle [radius=0.5pt];

	\node at (-2.7,2.5) {($\mathrm{R1}^-$)};

\end{tikzpicture} \\
& \begin{tikzpicture}[scale=0.7]
	\coordinate (h1) at (0.5,0);
	\coordinate (h3) at (3.5,0);
	\coordinate (h41) at (-0.5,0.3);
	\coordinate (h42) at (-0.5,-0.3);
	\coordinate (h43) at (-1.0,0.45);
	\coordinate (h44) at (-1.0,-0.45);

	\coordinate (g1) at (8.5,0);
	\coordinate (g3) at (11.5,0);
	\coordinate (g41) at (7.5,0.3);
	\coordinate (g42) at (7.5,-0.3);
	\coordinate (g43) at (7,0.45);
	\coordinate (g44) at (7,-0.45);
	
	\draw[fill] (h1) circle [radius=2pt] node [above=2pt] {$e_i$} node [below=2.5pt] {$[g_i]$};
	\draw[fill] (h3) circle [radius=2pt] node [above=2pt] {$-1$};
	\draw (h1) to [out=30, in=150] (h3) to [out=210, in=330] (h1);
	\draw (h1) -- (h41);
	\draw (h1) -- (h42);
	\draw[dashed] (h41) -- (h43);
	\draw[dashed] (h42) -- (h44);

	\draw [->] (5,0) -- (6,0);

	\draw [fill] (-0.4,0.15) circle [radius=0.5pt];
	\draw [fill] (-0.4,0) circle [radius=0.5pt];
	\draw [fill] (-0.4,-0.15) circle [radius=0.5pt];

	\draw[fill] (g1) circle [radius=2pt] node [above=3pt] {$e_i+2$} node [below=2.5pt] {$[g_i]$};
	\draw (g1) to [out=30, in=90] (g3) to [out=270, in=330] (g1);
	\draw (g1) -- (g41);
	\draw (g1) -- (g42);
	\draw[dashed] (g41) -- (g43);
	\draw[dashed] (g42) -- (g44);

	\draw [fill] (7.6,0.15) circle [radius=0.5pt];
	\draw [fill] (7.6,0) circle [radius=0.5pt];
	\draw [fill] (7.6,-0.15) circle [radius=0.5pt];

	\node at (-2.7,0) {(${\rm R1}^{++}$)};
\end{tikzpicture} \\
&\ \begin{tikzpicture}[scale=0.7]
	\coordinate (e1) at (0,-2.5);
	\coordinate (e3) at (2,-2.5);
	\coordinate (e41) at (-1,-2.8);
	\coordinate (e42) at (-1,-2.2);
	\coordinate (e43) at (-1.5,-2.95);
	\coordinate (e44) at (-1.5,-2.05);
	\coordinate (e51) at (3.2,-3);
	\coordinate (e52) at (3.2,-2);

	\coordinate (f1) at (9,-2.5);
	\coordinate (f41) at (8,-2.8);
	\coordinate (f42) at (8,-2.2);
	\coordinate (f43) at (7.5,-2.95);
	\coordinate (f44) at (7.5,-2.05);

	\draw[fill] (e1) circle [radius=2pt] node [above=2pt] {$e_i$} node [below=2.5pt] {$[g_i]$};
	\draw[fill] (e3) circle [radius=2pt] node [above=2pt] {$-1$};
	\draw[fill] (e51) circle [radius=2pt] node [right=2pt] {$-2$};
	\draw[fill] (e52) circle [radius=2pt] node [right=2pt] {$-2$};
	\draw (e1) -- (e3);
	\draw (e1) -- (e41);
	\draw (e1) -- (e42);
	\draw[dashed] (e41) -- (e43);
	\draw[dashed] (e42) -- (e44);
	\draw (e3) -- (e51);
	\draw (e3) -- (e52);

	\draw [->] (5.5,-2.5) -- (6.5,-2.5);

	\draw[fill] (f1) circle [radius=2pt] node [above=2pt] {$e_i$} node [below=12.5pt, right=-7pt] {$[g_i \# (-1)]$};
	\draw (f1) -- (f41);
	\draw (f1) -- (f42);
	\draw[dashed] (f41) -- (f43);
	\draw[dashed] (f42) -- (f44);

%	\node at (-2.7,5) {($\mathrm{R1^{+}}$)};
%	\node at (-2.7,2.5) {($\mathrm{R1}^-$)};
%	\node at (-2.7,0) {(${\rm R1}^{++}$)};
	\node at (-2.7,-2.5) {(R2)};
	
	\draw [fill] (-0.9,-2.65) circle [radius=0.5pt];
	\draw [fill] (-0.9,-2.5) circle [radius=0.5pt];
	\draw [fill] (-0.9,-2.35) circle [radius=0.5pt];

	\draw [fill] (8.1,-2.65) circle [radius=0.5pt];
	\draw [fill] (8.1,-2.5) circle [radius=0.5pt];
	\draw [fill] (8.1,-2.35) circle [radius=0.5pt];
	
\end{tikzpicture} 
\end{align*}
where $g\#(-1):=-2g-1$ if $g\geq 0$, and $g\#(-1):=g-1$ otherwise. 
In \cite{neumann1981}, the operations (${\rm R1}_0^+$), (${\rm R1}^+$), (${\rm R1}^-$) and (${\rm R1}^{++}$) are called \textit{blowing-downs} and their reverse are called \textit{blowing-ups}, and (R2) is called an \textit{$\RR\PP^2$-absorption} and its reverse is called an \textit{$\RR\PP^2$-extrusion}. 
Note that the blowing-down (${\rm R1}^{++}$) is different from the blowing-down of the dual graph for an algebraic curve on a complex surface. 
\end{rem}

%\begin{rem}
%For a general decorated plumbing graph $\Gamma_\dplumb$ and its F-normal form $\Gamma_\dplumb^\norm$, a canonical homeomorphism $h\colon\PM(\Gamma_\dplumb)\to\PM(\Gamma_\dplumb^\norm)$ does not preserve the trivializations of boundary components since a blowing-up (down) ($\mathrm{R}1^\pm$) on an edge $y\in\mcY(\Gamma_{\dplumb})$ changes the trivialization of the boundary component $T_y$, and such operation for $y$ with $t(y)\in\AH(\Gamma_{\dplumb})$ may appear to obtain $\Gamma_\dplumb^\norm$ (cf. \cite[Theorem~4.1]{neumann1981}). 
%However, the F-normal form works in our case (see proofs of Lemmas~\ref{lem:orientation_gen_sing} and \ref{lem:Homeo_Sing}). 
%
%When one wants to preserve trivializations of boundary components, it is enough to treat any vertex $v$ adjacent to a boundary vertex like a vertex $v$ with $\bm{g}(v)>0$ or $\bm{d}(v)>2$ (i.e., to remove such vertices from any chain) since preserving the trivialization of $T_y$ for an edge $y$ implies that fibers of $M_{t(y)}$ are preserved (cf. Theorem~\ref{thm:S1-bdl_gene}). 
%\end{rem}

\begin{rem}\label{rem:sign}
For a decorated plumbing graph $\Gamma_\dplumb=(\Gamma,\bm{g},\bm{e},\sign;\Ver, \AH)$, let $\Gamma^{\ast}$ be the full subgraph of $\Gamma$ defined by all vertices $v\in\Ver(\Gamma)$ with $\bm{g}(v)\geq0$ and all boundary vertices in $\AH(\Gamma)$. 
As mentioned in \cite[p.340]{neumann1981}, the information contained in the edge weights is completely coded up to the F-calculus corresponding to (R0) in \cite[Proposition~2.1]{neumann1981} by a homomorphism $\varepsilon_{\Gamma_\dplumb}\colon H_1(\Gamma^{\ast},\AH(\Gamma))\to\ZZ/2\ZZ$	which assigns to any relative cycle $C$ of $(\Gamma^{\ast},\AH(\Gamma))$ the number modulo $2$ of $(-)$-edges on $C$. 
\end{rem}

\begin{defin}\label{def:intersection_form}
	Let $\Gamma_\dplumb=(\Gamma,\bm{g},\bm{e};\Ver, \AH)$ be a decorated plumbing graph with $\sign(y)=+1$ for each $y\in\mcY(\Gamma)$. 
	The \textit{intersection form} $S(\Gamma_\dplumb)$ of $\Gamma_\dplumb$ can be defined as in \cite[Section~8]{neumann1981}. 
	Namely, if $\Ver(\Gamma)=\{v_1,\dots,v_r\}$, then $S(\Gamma_\dplumb):=(a_{ij})_{i,j=1,\dots,r}$ with 
	\begin{align*}
		a_{ij}:=\left\{\begin{array}{ll}
			\bm{e}(v_i) + \#\{y\in\mcY(\Gamma)\mid o(y)=t(y)=v_i\} & (i=j), \\[0.3em]
			\#\{ y\in\mcY(\Gamma) \mid o(y)=v_i, t(y)=v_j \} & (i\ne j). 
		\end{array}\right.
	\end{align*}
\end{defin}

In Appendix~\ref{sec:pgm}, we recall properties of plumbing graphs which are used in this note. 

%%%%%%%%%%%%%%%%%%%
\subsection{Modified plumbing graphs}
%%%%%%%%%%%%%%%%%%%

Hironaka \cite{hironaka2000} defined \textit{modified plumbing graphs} to describe normal surface-curve pairs $(X,\mcC)$ which are pairs of a complex normal surface $X$ and an algebraic curve $\mcC\subset X$. 
We will use modified plumbing graph to define a $G$-combinatorics for a plane curve. 
In this subsection, we recall the definition of the modified plumbing graph without boundary vertices. 

Let $\fin(S^1\times S^1)$ be the set of finite unbranched {\red coverings} from $S^1\times S^1$ to itself. 
A \textit{modified plumbing graph} $\Gamma_\plumb^\modi=(\Gamma, \bm{g}, \bm{e}, \bm{m})$ is a plumbing graph $\Gamma_\plumb=(\Gamma,\bm{g},\bm{e})$ satisfying $\sign(y)=+1$ for each $y\in\mcY(\Gamma)$ together with a map $\bm{m}\colon \mcY(\Gamma) \to \fin(S^1\times S^1)$ 
such that 
\begin{enumerate}
	\item by choosing certain bases of tori $S^1\times S^1$, the unbranched coverings $\bm{m}(y)\colon S^1\times S^1\to S^1\times S^1$ are represented by  matrices in $M_2(\ZZ)$ of the following form; 
	\begin{align}\label{eq:modi_matrix} \bm{m}(y)=\begin{pmatrix}  c(y) & 0 \\ b(y) & a(y) \end{pmatrix}\colon  S^1\times S^1\longrightarrow S^1\times S^1, \end{align}
	where $0\leq b(y)<a(y)$ and $c(y)>0$; and 
	\item the induced maps $\bm{m}(y)_\ast$ and $(J\,\bm{m}(\overline{y}))_\ast$ from $\pi_1(S^1\times S^1)$ to $\pi_1(S^1\times S^1)$ have the same image in $\pi_1(S^1\times S^1)\cong\ZZ\oplus\ZZ$. 
\end{enumerate}

\begin{rem}
The above definition of $\bm{m}(y)$ is different from $m(y)$ in \cite{hironaka2000}. 
This difference comes from the difference of trivializations for boundary components of $S^1$-bundles as follows:
For an $S^1$-bundle $p\colon M\to B$ over a $2$-manifold $B$ with boundary, 
each boundary $T_i\subset \partial M$ is trivialized by a homeomorphism $\tau_i\colon S^1\times S^1\to T_i$ such that $p\circ\tau_i$ is the projection onto the second component in \cite{hironaka2000}, but onto the first one in this note according to \cite{neumann1981}, i.e., $\bm{m}(y)=Jm(y)J$. 
Hence $a(y), b(y)$ and $c(y)$ in this note play the same role of those in \cite{hironaka2000}. 
\end{rem}

Given a modified plumbing graph $\Gamma_\plumb^\modi=(\Gamma,\bm{g},\bm{e},\bm{m})$, an associated graph manifold $\PM(\Gamma_\plumb^\modi)$ is defined as follows: 
For each $v\in\mcV(\Gamma)$, let $B_v$ be a compact $2$-manifold of genus $\bm{g}(v)$ with boundary components, labeled by $y\in\mcY(\Gamma)$ with $t(y)=v$, 
and let $p_v\colon M_v\to B_v$ be an $S^1$-bundle, where $M_v$ is an oriented $3$-manifold with trivializations at boundary components such that the {\red Euler} number of $p_v$ is $\bm{e}(v)$. 
For each $y\in\mcY(\Gamma)$, let $T_y\subset M_{t(y)}$ be the boundary component associated to $y$ with trivialization $\tau_y\colon S^1\times S^1\to T_y$, and 
put $R_y:=\bm{m}(y)^{-1} J \bm{m}(\overline{y})$. 
By the condition (ii) of $\bm{m}$, we have $R_y\in\GL(2,\ZZ)$, $\det R_y=-1$, and $R_y$ represent homeomorphisms $T_{\overline{y}}\to T_y$: 
\[ \begin{tikzpicture}
	\node (b1) at (0,0) {$S^1\times S^1$};
	\node (b2) at (2.5,0) {$S^1\times S^1$};
	\node (a1) at (-2.5,1.3) {$T_{\overline{y}}$};
	\node (a2) at (5,1.3) {$T_y$};
	\node (c1) at (0,1.3) {$S^1\times S^1$};
	\node (c2) at (2.5,1.3) {$S^1\times S^1$};
	\draw [->] (a1) -- node [above] {\footnotesize $\tau_{\overline{y}}^{-1}$} (c1);
	\draw [->] (c2) -- node [above] {\footnotesize $\tau_{y}$} (a2);
	\draw [->] (c1) -- node [above] {\footnotesize $R_{y}$} (c2);
	\draw [->] (b1) -- node [above] {\footnotesize $J$} (b2);
	\draw [->] (c1) -- node [left] {\footnotesize $\bm{m}(\overline{y})$} (b1);
	\draw [->] (c2) -- node [right] {\footnotesize $\bm{m}(y)$} (b2);
\end{tikzpicture} \]
The manifold $\PM(\Gamma_\plumb^\modi)$ is obtained by gluing $M_{t(y)}$ and $M_{t(\overline{y})}$ by $R_y$.

\begin{rem}
For $y\in\mcY(\Gamma)$, $\bm{m}(y)\colon S^1\times S^1\to S^1\times S^1$ coincides with the map induced by the map
\[ \CC^2\setminus\{u_1v_1\ne0\}\ni (u_1,v_1)\longmapsto(u_1^{c(y)},u_1^{b(y)}v_1^{a(y)})=(u_2,v_2)\in\CC^2\setminus\{u_2v_2\ne0\}\] 
by regarding $S^1\times S^1$ as $\{|u_i|=|v_i|=1\}$. 
This map appears in \cite[Chapter 2]{laufer1971} to construct a resolution of singularities of $2$-dimensional analytic spaces. 
Note that the dual graph of the exceptional divisor on the resolution is a chain, and $\PM(\Gamma_\plumb^\modi)$ can be constructed as the plumbed manifold $\PM(\Gamma_\plumb)$ of $\Gamma_\plumb$, where $\Gamma_\plumb$ is the plumbing graph constructed by replacing each edge of $\Gamma_\plumb^\modi$ with the chain of the exceptional divisor. 
By the resolution in \cite{laufer1971}, we can check that the basis of $H_1(T_y,\ZZ)$ giving the matrix (\ref{eq:modi_matrix}) corresponding to the basis given by~$\Gamma_\plumb$. 
\end{rem}

%%%%%%%%%%%%%%%%%%%%%%%%%%%%%%%%
\section{Combinatorial type of plane curves}\label{sec:comb}
%%%%%%%%%%%%%%%%%%%%%%%%%%%%%%%%

In this section, we prove Fact~\ref{fact:comb} and Theorem~\ref{thm:meridian}. 
We first recall the combinatorial type of a plane curve. 
The combinatorial type is defined in \cite{act2008} as the data consisting of the number of irreducible components, the degrees and singularities of irreducible components, and configuration of components. 
In \cite{abst2023}, it is redefined by using the resolution graph of plane curves, which we explain below. 

\subsection{Combinatorical type}\label{subsec:comb}
Let $\mcC\subset\PP^2$ be a plane curve of degree $d$, and let $\Irr(\mcC)$ denote the set of irreducible components of $\mcC$. 
Let $\sigma\colon Y\to\PP^2$ be the minimal good embedded resolution of the singularities $\Sing(\mcC)$ of $\mcC$, i.e., the minimal sequence of blowing-ups which {\red resolves} $\Sing(\mcC)$ and such that $\sigma^{-1}(\mcC)$ is a simple normal crossing divisor. 
Put 
\[ {\mcC'}:=\sigma^{-1}(\mcC).  \]
In \cite{abst2023}, the \textit{combinatorial type} (or the \textit{combinatorics} for short) $\Comb(\mcC)$ of the plane curve $\mcC$ is redefined as the $3$-tuple $\Comb(\mcC):=(\Gamma_\mcC, \Str_\mcC, \bm{e}_\mcC)$, where 
\begin{itemize}
	\item $\Gamma_\mcC$ is the dual graph of $\red \mcC'$, 
	\item $\Str_\mcC\subset\mcV(\Gamma_\mcC)$ is the set of vertices corresponding to the strict transforms of irreducible components of $\mcC$, and 
	\item $\bm{e}_\mcC$ is the map $\bm{e}_\mcC\colon \mcV(\Gamma_\mcC)\to\ZZ$ of self-intersection numbers. 
\end{itemize}
Since ${\mcC'}$ is simple normal crossing, each edge {\red of} $\Gamma_\mcC$ is not a loop. 
We can identify the sets $\Irr({\mcC'})$ and $\Irr(\mcC)$ of irreducible components with $\mcV(\Gamma_\mcC)$ and $\Str_\mcC$, respectively, 
\begin{align*} 
\Irr({\mcC'})&=\mcV(\Gamma_\mcC), &
\Irr(\mcC)&=\Str_\mcC. 
\end{align*}
Each node $P\in\Sing({\mcC'})$ corresponds to a pair of edges $\{y, \overline{y}\}$ for some $y\in\mcY(\Gamma_\mcC)$. 
For each $v\in\mcV(\Gamma_\mcC)$, let $\red D'_v$ denote the irreducible component of ${\mcC'}$ corresponding to $v$. 
By contracting exceptional components of $\sigma$, we can compute the degrees of irreducible components of $\mcC$ and the map 
$\bm{g}_\mcC\colon \mcV(\Gamma_\mcC)\to\ZZ$ 
of genera from $\Comb(\mcC)$. 
Hence the combinatorics $\Comb(\mcC)$ can be regarded as the plumbing graph 
\[ \Comb(\mcC)=(\Gamma_\mcC, \bm{g}_\mcC,\bm{e}_\mcC; \Str_\mcC) \]
together with the marking $\Str_\mcC\subset\mcV(\Gamma_\mcC)$. 
Let $\PM(\mcC)$ denote the plumbed manifold associated with $\Comb(\mcC)$:
\[ \PM(\mcC):=\PM(\Gamma_\mcC,\bm{g}_\mcC,\bm{e}_\mcC). \]
For two plane curves $\mcC_1, \mcC_2\subset\PP^2$, an \textit{equivalence map} $\varphi\colon \Comb(\mcC_1)\to\Comb(\mcC_2)$ is an isomorphism $\varphi\colon \Gamma_{\mcC_1}\to\Gamma_{\mcC_2}$ of graphs such that $\varphi(\Str_{\mcC_1})=\Str_{\mcC_2}$, $\bm{g}_{\mcC_1}=\bm{g}_{\mcC_2}\circ\varphi$ and $\bm{e}_{\mcC_1}=\bm{e}_{\mcC_2}\circ\varphi$. 
We say that $\mcC_1$ and $\mcC_2$ \textit{have the same combinatorics} if there is an equivalence map $\varphi\colon \Comb(\mcC_1)\to\Comb(\mcC_2)$. 

\subsection{The boundary manifold $\BM(\mcC)$}\label{subsec:reg_nbd}

According to \cite{durfee_1983}, we define an algebraic neighborhood $\mcT(\mcC)$ of a plane curve $\mcC\subset\PP^2$ as follows:
Let $f\in\CC[x_0,x_1,x_2]$ be a homogeneous polynomial of degree $d$ defining $\mcC$. 
Let $\eta_f\colon \PP^2\to\RR$ be the map defined by
\[ \eta_f(x_0,x_1,x_2):=%\colon \PP^2\ni (x_0:x_1:x_2)\mapsto 
\frac{|f(x_0,x_1,x_2)|^2}{(|x_0|^2+|x_1|^2+|x_2|^2)^d}\in\RR  \]
for $(x_0:x_1:x_2)\in\PP^2$. 
Let $\mcT(\mcC):=\{(x_0:x_1:x_2)\in\PP^2\mid \eta_f(x_0,x_1,x_2)\leq\delta\}$  for sufficiently small $\delta>0$, which is called an \textit{algebraic neighborhood} of $\mcC$. 
%\[ \mcT(\mcC):=\left\{ (x:y:z)\in\PP^2 \ \middle| \ \frac{|f(x,y,z)|^2}{(|x|^2+|y|^2+|z|^2)^d}\leq \delta \right\}, \]
%where $\delta>0$ is sufficiently small positive real number. 
By \cite{lojasiewicz1964}, $\mcT(\mcC)$ is a regular neighborhood of $\mcC$. 
By \cite{durfee_1983}, $\mcT(\mcC)$ is unique up to isotopy as an algebraic neighborhood. 
However, the author is not sure whether $\mcT(\mcC)$ is unique as a regular neighborhood.  
Let $\BM(\mcC)$ be the boundary of $\mcT(\mcC)$,
$ \BM(\mcC):=\partial\mcT(\mcC)$, 
called the \textit{boundary manifold} of $\mcC$. 

\begin{rem}\label{rem:isotopy_reg_nbd}
	{\red For any regular neighborhood $U(\mcC)$ of $\mcC$, we may assume that $\mcT(\mcC)\subset U(\mcC)$ by taking sufficiently small $\delta>0$}. 
\end{rem}

Let $\sigma\colon Y\to\PP^2$ be the minimal good embedded resolution of $\mcC$, and put $\mcC':=\sigma^{-1}(\mcC)$. 
Since ${\red \eta_f}\circ\sigma\colon Y\to\RR$ is admissible in the sense of \cite{mumford1961}, 
we can check that $\sigma^{-1}(\mcT({\mcC}))$ is given by plumbing tubular neighborhoods $\mcT({\red D'})$ of irreducible components ${\red D'}$ of ${\mcC'}$ (cf.\ \cite[Section~8]{HNK1971} and \cite[Section~1]{mumford1961}). 
Note that, for each irreducible component ${\red D'}\subset{\mcC'}$, the {\red Euler} number of the $S^1$-bundle $\red \overline{p}_D\colon \partial\mcT(D')\to D'$ is equal to the self-intersection number of ${\red D'}$. 
{\red 
From this point of view, $\BM(\mcC)$ is constructed from $\Comb(\mcC)$ as follows. 
For $v\in\mcV(\Gamma_\mcC)$, put 
\begin{align*}\label{eq:S^1-bdls mcC} 
	B_v&:=D'_v\setminus U(\Sing(\mcC'))^\circ, 
	& 
	M_v&:=\overline{p}_{D_v}^{-1}(B_v), 
	&
	p_v&:=\overline{p}_{D_v}|_{M_v}\colon M_v\longrightarrow B_v. 	
\end{align*}
For $y\in\mcY(\Gamma_\mcC)$, let $\partial B_{y}$ be the connected component of $\partial B_{t(y)}$ around the node $P_y$ of $\mcC'$ corresponding to $y$, and put $T_y:=p_{t(y)}^{-1}(\partial B_{y})$. 
We may assume that $\mcC'$ and $D'_{t(y)}$ are locally defined at $P_y$ by $z_1z_2=0$ and $z_2=0$, respectively, in $\DD^2\times \DD^2$ with coordinates $(z_1,z_2)$, 
and that $T_y$ is a torus given by $|z_1|=|z_2|=\delta$ for small $\delta>0$. 
Let 
%\begin{align}\label{eq:trivialization mcC} 
$\tau_y\colon S^1\times S^1\to T_y$	
%\end{align}
be the trivialization given by $(z_1,z_2)\mapsto (\delta z_1,\delta z_2)$. 
Then $\BM(\mcC)$ is constructed by gluing $M_{t(y)}$ and $M_{t(\overline{y})}$ by (\ref{eq:gluing torus}) with $\sign(y)=+1$. 
Hence $\BM(\mcC)$ is a graph manifold with graph structure $\TT_\mcC:=\bigcup_y T_y$, where $T_{\overline{y}}$ is identified with $T_y$ in $\BM(\mcC)$ for each $y\in\mcY(\Gamma_\mcC)$. 
We call $\TT_\mcC\subset\BM(\mcC)$ the \textit{graph structure with respect to $\Comb(\mcC)$}. 
}
It follows from the definition of $\PM(\mcC)$ that $\BM(\mcC)$ is homeomorphic to $\PM(\mcC)$. 
Therefore, if two plane curves $\mcC_1,\mcC_2\subset\PP^2$ have the same combinatorics, then there is a homeomorphism $h\colon\mcT(\mcC_1)\to\mcT(\mcC_2)$ with $h(\mcC_1)=\mcC_2$. 

\subsection{Embedded topology in regular neighborhoods}

Let $\mcC_1,\mcC_2\subset\PP^2$ be two plane curves, and let $\sigma_i\colon Y_i\to\PP^2$ be the minimal good embedded resolution of $\mcC_i$. 
We put ${\mcC}'_{i}:=\sigma_i^{-1}(\mcC_i)$, and identify the combinatorics $\Comb(\mcC_i)$ with the corresponding plumbing graph, write 
\[ \Comb(\mcC_i)=(\Gamma_i,\bm{g}_i,\bm{e}_i;\Str_i). \]
For $v_i\in\Str_i$ (resp. $v_i\in\mcV(\Gamma_i)$), let $\red D_{v_i}$ (resp. $\red{D}'_{v_i}$) be the irreducible component of $\mcC_i$ (resp. ${\mcC}'_i$) corresponding to $v_i$. 
{\red 
For $v_i\in\mcV(\Gamma_i)$ and $y_i\in\mcY(\Gamma_i)$, let $M_{v_i}$, $B_{v_i}$, $p_{v_i}$, $T_{y_i}$ and $\tau_{y_i}$ be as in Subsection~\ref{subsec:reg_nbd}. 
}
Let $\red \TT_i=\bigcup_{y_i} T_{y_i}\subset\BM(\mcC_i)$ be the graph structure with respect to $\Comb(\mcC_i)$. 
%Note that each $P_i\in\Sing({\mcC}'_i)$ (resp. $v_i\in\mcV(\Gamma_i)$) corresponds to a torus $\red T_{P_i}$ in $\TT_i$ (resp. a connected component $\red M_{v_i}$ of $\BM(\mcC_i)\setminus U(\TT_i)^\circ$). 
%For $v_i\in\mcV(\Gamma_i)$ and $y_i\in\mcY(\Gamma_i)$, let 
%\[ \red p_{v_i}\colon M_{v_i}\longrightarrow B_{v_i}, \qquad \tau_{y_i}\colon S^1\times S^1\longrightarrow T_{y_i} \]
%be the $S^1$-bundle of $v_i$ and the trivialization for $y_i$, respectively, where 
%{\red $B_{v_i}$ is the complement of $U(\Sing(\mcC'_i))^\circ$ in $D'_{v_i}$, $B_{v_i}:={D}'_{v_i}\setminus U(\Sing({\mcC}'_i))^\circ$}, and $\red T_{y_i}$ is the boundary component of $\red M_{t(y_i)}$ corresponding to~$y_i\in~\mcY(\Gamma_i)$. 
%%For $y_i\in\mcY(\Gamma_i)$, let $T_{i,y_i}$ be the boundary component of $M_{i,t(y_i)}$ corresponding to $y_i$, and let 
%%\[ \tau_{i,y_i}\colon S^1\times S^1\to T_{i,y_i} \]
%%be the trivialization for $y_i$. 

Assume that there is a homeomorphism $h\colon U(\mcC_1)\to U(\mcC_2)$ with $h(\mcC_1)=\mcC_2$ for some regular neighborhoods $U(\mcC_i)$ of $\mcC_i$. 
As in the proof of \cite[Theorem~2.2]{artal_cogolludo_martin_2019}, $h$ induces an isomorphism 
\[ \Psi_h^\infty\colon \pi_1(\BM(\mcC_1))\longrightarrow\pi_1(\BM(\mcC_2)). \]
The isomorphism $\Psi_h^\infty$ can be constructed as follows: 
Let $\mcT(\mcC_2)\subset h(U(\mcC_1))$ be an algebraic neighborhood of $\mcC_2$, and let $\mcT(\mcC_1)\subset h^{-1}(\mcT(\mcC_2))$ be an algebraic neighborhood of $\mcC_1$. 
Then there is an algebraic neighborhood $\mcT'(\mcC_2)\subset h(\mcT(\mcC_1))$. 
Since $\BM(\mcC_i)$ is homotopic to $\mcT(\mcC_i)\setminus\mcC_i$, we obtain the homomorphism 
\[ \Psi_h^\infty\colon\pi_1(\BM(\mcC_1))\cong\pi_1(\mcT(\mcC_1)\setminus\mcC_1)\overset{h_\ast}{\longrightarrow}\pi_1(\mcT(\mcC_2)\setminus\mcC_2)\cong\pi_1(\BM(\mcC_2)). \] 
Since we can similarly construct the inverse map of $\Psi_h^\infty$, it is {\red an} isomorphism.  

In \cite{artal_cogolludo_martin_2019}, it is also proved by using \cite[Corollary~6.5]{waldhausen_1968} that there exists a homeomorphism $\Psi_\BM\colon \BM(\mcC_1)\to\BM(\mcC_2)$ with $\Psi_{\BM\,\ast}=\Psi_h^\infty$ in several cases. 
Here, we construct such a homeomorphism $\Psi_\BM\colon \BM(\mcC_1)\to\BM(\mcC_2)$ homotopic to $h|_{\BM(\mcC_1)}$ in general case, and prove that $\Psi_\BM$ maps meridians of irreducible components of ${\mcC}'_1$ to those of ${\mcC}'_2$. 

\begin{rem}\label{rem:comb_smooth}
We can construct a map $h'\colon\BM(\mcC_1)\to\BM(\mcC_2)$ homotopic to $h|_{\BM(\mcC_1)}$, which is homeomorphic outside of small regular neighborhoods of singular points, as follows: 

By \cite[Theorem~21]{brieskorn_knorrer_1986}, $\mcC_1$ and $\mcC_2$ have the same singularity at each $P_1\in\mcC_1$ and $h(P_1)\in\mcC_2$. 
By the isotopy theorem for closed tubular neighborhoods \cite[Theorem~6.5]{hirsch_1976}, we may assume that $h$ maps $\BM(\mcC_1)\setminus U_{\Sing,1}$ to $\BM(\mcC_2)\setminus U_{\Sing,2}$ homeomorphically for certain regular neighborhoods $U_{\Sing,i}:=U(\Sing(\mcC_i))$. 
Let $\mcT_i(\mcC_2)$ ($i=1,2$) be two algebraic neighborhoods of $\mcC_2$ such that $\mcT_1(\mcC_2)\subset h(\BM(\mcC_1))\subset \mcT_2(\mcC_2)$. 
By a certain deformation retract $\mcT_2(\mcC_2)\to\mcT_1(\mcC_2)$, we obtain {\red the} desired map $h'\colon\BM(\mcC_1)\to \BM(\mcC_2)$ with $h'(\BM(\mcC_1)\cap U_{\Sing,1})=\BM(\mcC_2)\cap U_{\Sing,2}$. 

In particular, if $\mcC_1$ is smooth, then $\mcC_2$ is also smooth, and $h'$ is {\red a homeomorphism}. This implies $\deg\mcC_1=\deg\mcC_2$, i.e., $\mcC_1$ and $\mcC_2$ have the same combinatorics. 
% by \cite[III, Theorem~21]{brieskorn_knorrer_1986}, and it follows from the isotopy theorem for closed tubular neighborhoods \cite[Theorem~6.5]{hirsch_1976} that there is a homeomorphism $\Psi_\BM\colon\BM(\mcC_1)\to\BM(\mcC_2)$ which is isotopic to $h|_{\BM(\mcC_1)}$. 
%This implies that $\deg\mcC_1=\deg\mcC_2$, i.e., $\mcC_1$ and $\mcC_2$ have the same combinatorics. 
Hence we assume that $\mcC_i$ are singular below. 
\end{rem}

%\medskip

%%%%%%%%%%%%%%%
Assume that $\mcC_i$ are singular. 
We locally construct a homeomorphism of boundary manifolds around each singular point homotopic to $h$. 
Let $P_1$ be a singular point  of $\mcC_1$, and 
put $P_2:=h(P_1)$. 
Let $U(P_i)\subset\PP^2$ be a regular neighborhood of $P_i$ for each $i=1,2$. 
We may assume that   
\begin{align*}\label{eq:local_tb_nbd}
\BM(\mcC_i,P_i):=\BM(\mcC_i)\cap U(P_i) \subset \BM(\mcC_i)\setminus\bigcup_{v_i\in\Str_i} {\red M_{v_i}^\circ}
\end{align*}
 is a connected component of the complement of the union of $\{{\red M_{v_i}^\circ}\mid v_i\in\Str_i\}$ in $\BM(\mcC_i)$. 
We can regard $\BM(\mcC_i,P_i)$ as a boundary framed graph manifold associated with the decorated plumbing graph 
\[ \Gamma_\dplumb(\mcC_i,P_i)=(\Gamma_{i,P_i}, \bm{g}_{i,P_i}, \bm{e}_{i,P_i};\Ver,  \AH) \] 
defined as follows: 
If $P_i\in\mcC_i$ is not {\red a node at which two distinct components intersect}, then we define $\Gamma_\dplumb(\mcC_i,P_i)$ by the following steps;
\begin{enumerate}[label=(\Roman*)]
	\item let $\Ver(\Gamma_{i,P_i})$ be the subset of $\mcV(\Gamma_i)$ consisting of vertices $v_i$ with $\sigma_i({\red D'_{v_i}})=\{P_i\}$;  
	\item $\mcY(\Gamma_{i,P_i})$ is the subset of $\mcY(\Gamma_i)$ consisting of edges $y_i$ with either $t(y_i)\in \Ver(\Gamma_{i,P_i})$ or $o(y_i)\in \Ver(\Gamma_{i,P_i})$; 
	\item {\red $\AH(\Gamma_{i,P_i})$ is in bijection with the set of edges $y_i\in \mcY(\Gamma_{i,P_i})$ with $t(y_i)\notin \Ver(\Gamma_{i,P_i})$ as edges of $\Gamma_i$; let $v_{i,y_i}$ be the vertex of $\AH(\Gamma_{i,P_i})$ corresponding to $y_i$;}
	\item for each $y_i\in\mcY(\Gamma_{i,P_i})$, let $t(y_i)$ be the same terminal vertex $t(y_i)$ of $y_i\in\mcY(\Gamma_i)$ if $t(y_i)\in \Ver(\Gamma_{i,P_i})$ as {\red a vertex} of $\Gamma_i$; and 
	put $t(y_i):=v_{i,y_i}\in\AH(\Gamma_{i,P_i})$ {\red otherwise}; 
	\item for each $y_i\in\mcY(\Gamma_{i,P_i})$, put $o(y_i):=t(\overline{y}_i)$ as edges of $\Gamma_{i,P_i}$; 
	\item $\bm{g}_{i,P_i}, \bm{e}_{i,P_i}\colon \Ver(\Gamma_{i,P_i})\to\ZZ$ are the restrictions of $\bm{g}_i$ and $\bm{e}_i$, respectively.
\end{enumerate}
Note that $\Gamma_{\dplumb}(\mcC_i,P_i)$ corresponds to the dual graph of the minimal good embedded resolution of the singularity $P_i\in\mcC_i$ locally. 
Hence $\Gamma_{i,P_i}$ has no cycle. 

According to \cite[Chapter~V]{eisenbud_1986}, if $P_i\in\mcC_i$ is a node at which two distinct components intersect, we take the blowing-up at $P_i$; replace {\red $\Comb(\mcC_i)$ with the plumbing graph $\Gamma_{\plumb}(\mcC_i)$ given by the blowing-up (${\rm R1}^+$) of $\Comb(\mcC_i)$ at the edge corresponding to $P_i$}; and define $\Gamma_\dplumb(\mcC_i,P_i)$ as above. 
Namely, if $P_i$ is a node of $\mcC_i$, then $\Gamma_\dplumb(\mcC_i,P_i)$ is the following decorated plumbing graph:
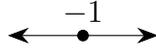
\begin{align*}\label{graph:smooth_node}
\begin{aligned}
\begin{tikzpicture} 
\draw[-{Stealth[length=8pt]}] (0,0) -- (1,0); 
\draw [-{Stealth[length=8pt]}] (0,0) -- (-1,0);
\draw[fill] (0,0) circle [radius=2pt] node [above] {$-1$}; 
\end{tikzpicture}
\end{aligned}
 \end{align*}

Let $\TT_{i,P_i}$ be the graph structure of $\BM(\mcC_i,P_i)$ with respect to $\Gamma_\dplumb(\mcC_i,P_i)$. 
We denote the F-normal form of $\Gamma_\dplumb(\mcC_i,P_i)$ by $\Gamma_\dplumb^\norm(\mcC_i,P_i)$. 
For $v_i\in\Ver(\Gamma_{i,P_i})$ and $y_i\in\mcY(\Gamma_{i,P_i})$, we may define $\red p_{v_i}\colon M_{v_i}\to B_{v_i}$, $\red T_{y_i}$ and $\red \tau_{y_i}\colon S^1\times S^1\to T_{y_i}$ as the same things for $\Gamma_{\plumb}(\mcC_i)$. 
Let $\TT_{i.P_i}\subset\BM(\mcC_i,P_i)$ be the graph structure with respect to $\Gamma_{\dplumb}(\mcC_i,P_i)$. 

Note that $\AH(\Gamma_{i,P_i})$ corresponds to the set of local branches of $\mcC_i$ at $P_i$. 
Hence $h$ induces a bijection 
\[ \xi_{h,P_1}\colon\AH(\Gamma_{1,P_1})\longrightarrow\AH(\Gamma_{2,P_2}) \] 
since $h(\mcC_1)=\mcC_2$. 
%Since $h(\mcT(\mcC_1)\cap U(P_1))$ and $\mcT(\mcC_2)\cap U(P_2)$ are good neighborhoods of $P_2$ with regard to $\mcC_2$, there is a homotopy $\gth_u\colon \mcT(\mcC_1)\cap U(P_1)\to U(\mcC_2)$ ($0\leq u\leq 1$) such that 
%\begin{itemize}
%	\item $\gth_0=h|_{\mcT(\mcC_1)\cap U(P_1)}$, 
%	\item $\gth_u(\mcT(\mcC_1)\cap U(P_1)\setminus\mcC_1)\subset U(P_2)\setminus\mcC_2$,
%	\item $\gth_1(\mcT(\mcC_1)\cap U(P_1))=\mcT(\mcC_2)\cap U(P_2)$  
%\end{itemize}
%by the argument of \cite[Section II. B]{prill_1967}. 
%Let $h_{P_1}\colon\BM(\mcC_1,P_1)\to\BM(\mcC_2,P_2)$ be the restriction of $\gth_1$. 
Let $h_{P_1}\colon\BM(\mcC_1,P_1)\to\BM(\mcC_2,P_2)$ be the restriction of $h'\colon\BM(\mcC_1)\to\BM(\mcC_2)$ constructed in Remark~\ref{rem:comb_smooth}. 
We have the induced isomorphism
\[ h_{P_1\ast}\colon \pi_1(\BM(\mcC_1,P_1))\longrightarrow\pi_1(\BM(\mcC_2,P_2)) \]
of fundamental groups $\pi_1(\BM(\mcC_i,P_i))$ such that $h_{P_1\ast}(i_{y_1\ast}(\pi_1({\red T_{y_1}})))=i_{y_2\ast}(\pi_1({\red T_{y_2}}))$ for each $y_i\in\mcY(\Gamma_{i,P_i})$ with $t(y_1)\in\AH(\Gamma_{1,P_1})$ and $t(y_2)=\xi_{h,P_1}(t(y_1))$, where $i_{y_i}\colon {\red T_{y_i}}\to\BM(\mcC_i,P_i)$ is the inclusion map. 
Note that it follows from \cite[III, Theorem~21]{brieskorn_knorrer_1986} that there exists an isomorphism 
\begin{align}\label{eq:isom_dplumb} 
\varphi_{P_1}\colon\Gamma_{\dplumb}(\mcC_1,P_1)\longrightarrow \Gamma_\dplumb(\mcC_2,P_2) 
\end{align} 
satisfying $\varphi_{P_1}|_{\AH(\Gamma_{1,P_1})}=\xi_{h,P_1}$, and hence $\BM(\mcC_i,P_i)$ are homeomorphic. 
\begin{rem}\label{rem:fact_comb}
We can prove from $\varphi_{P_1}$ for each $P_1\in\Sing(\mcC_1)$ that $\mcC_1$ and $\mcC_2$ have the same combinatorics if they have the same embedded topology in regular neighborhoods. 
Indeed, if $v_i\in\Str_i$ satisfy $\red h(D_{v_1})=D_{v_2}$, then we have $\bm{g}_1(v_1)=\bm{g}_2(v_2)$, and 
\[ \bm{g}_i(v_i)=\frac{(d_i-1)(d_i-2)}{2}-{\red \sum_{P_i\in\Sing(D_{v_i})}\delta_{D_{v_i},P_i}},   \]
where $d_i:=\deg {\red D_{v_i}}$, and $\delta_{\red D_{v_i},P_i}\in\ZZ$ is the delta invariant of $\red D_{v_i}$ at $P_i\in\Sing({\red D_{v_i}})$. 
The existence of $\varphi_{P_1}$ implies $\sum\delta_{\red D_{v_1},P_1}=\sum\delta_{\red D_{v_2},P_2}$. 
If $d_1\leq 2$, then we obtain $d_1=d_2$ by Remark~\ref{rem:comb_smooth}. 
If $d_1>2$, then $d_1=d_2$ follows from $d_i>0$. 
Hence we have $\bm{e}_1(v_1)=\bm{e}_2(v_2)$ by $d_1=d_2$ and the existence of $\varphi_{P_1}$. 
Therefore, we can construct an equivalence map $\varphi\colon\Comb(\mcC_1)\to\Comb(\mcC_2)$. 
The arguments in Subsection~\ref{subsec:reg_nbd} and this remark prove Fact~\ref{fact:comb}. 
\end{rem}

We here construct a homeomorphism $\gtg_{P_1}\colon\BM(\mcC_1,P_1)\to\BM(\mcC_2,P_2)$ which is homotopic to $h_{P_1}$, and satisfies certain properties. 
Put $E:=(\begin{smallmatrix}	1&0\\0&1 \end{smallmatrix})$, the identity matrix. 

\begin{lem}\label{lem:Homeo_smooth_node}
Assume that $P_1$ is a node of $\mcC_1$. 
Then $P_2$ is a node of $\mcC_2$, and there exists a homeomorphism $\gtg_{P_1}\colon\BM(\mcC_1,P_1)\to\BM(\mcC_2,P_2)$ such that 
\begin{enumerate}
	\item $\gtg_{P_1}$ is homotopic to $h_{P_1}$;
	\item $\gtg_{P_1}{\red (T_{y_1})=T_{y_2}}$ for $y_i\in\mcY(\Gamma_{i,P_i})$ with $t(y_i)\in\AH(\Gamma_{i,P_i})$ and $t(y_2)=\xi_{h,P_1}(t(y_1))$; 
	\item $\gtg_{P_1}$ preserves trivializations of the boundaries, i.e., for $y_i\in\mcY(\Gamma_{i,P_i})$ with $t(y_i)\in\AH(\Gamma_{i,P_i})$ and $t(y_2)=\xi_{h,P_1}(t(y_1))$, ${\red \tau_{y_2}^{-1}}\circ\gtg_{P_1}\circ{\red \tau_{y_1}}\colon S^1\times S^1\to S^1\times S^1$ is given by a matrix $\alpha E$ for some $\alpha=\pm1$. 
	\item for $v_i\in\Ver(\Gamma_{i,P_i})$, $\gtg_{P_1}$ preserves fibers of $v_1$ and $v_2$, i.e., for any fiber $F_1$ of $\red p_{v_1}$, $\gtg_{P_1}(F_1)$ is a fiber of $\red p_{v_2}$; 
\end{enumerate}
Moreover, $\alpha=1$ if and only if $h$ maps a meridian of $\red D_{t(y_1)}$ to a meridian of $\red D_{t(y_2)}$ for some $y_i\in\mcY(\Gamma_{i,P_i})$ with $t(y_2)=\xi_{h,P_1}(t(y_1))$,
where arrowheads in $\AH(\Gamma_{i,P_i})$ are regarded as vertices in $\Str_{i}$. 
\end{lem}
\begin{proof}
By \cite[III, Theorem~21]{brieskorn_knorrer_1986}, $P_2$ is also a node of $\mcC_2$. 
We may identify $\BM(\mcC_i,P_i)$ with $I\times S^1\times S^1$ so that $h_{P_1}(\{j\}\times S^1\times  S^1)=\{j\}\times S^1\times S^1$ for each $j\in\partial I$. 
Let $v_i\in\Ver(\Gamma_{i,P_i})$ be the unique vertex for each $i=1,2$. 
For $i=1,2$ and $j\in \partial I$, let $y_{ij}\in\mcY(\Gamma_{i,P_i})$ be the edge satisfying $w_{ij}:=t(y_{ij})\in\AH(\Gamma_{i,P_i})$ and ${\red T_{y_{ij}}}=\{j\}\times S^1\times S^1$. 
The trivialization $\tau_{ij}:={\red \tau_{y_{ij}}}$ is determined by a meridian $\gtm_{ij}$ of $\red D_{w_{ij}}$ and a fiber $\gtf_{i}$ of $\red p_{v_i}\colon M_{v_i}\to B_{v_i}$. 
It follows from $\bm{e}_{i,P_i}(v_i)=-1$ that $\gtf_i=\gtm_{i0}\gtm_{i1}$ in $\pi_1(\BM(\mcC_i,P_i))$. 
By \cite[Proposition~3.4 and Remark~3.5]{abst2023}, there is $\alpha\in\{\pm1\}$ such that $h_{P_1\ast}(\gtm_{1j})=\gtm_{2j}^\alpha$ for each $j\in\partial I$. 
Since $\bm{e}_1(v_1)=\bm{e}_2(v_2)=-1$, there is a homeomorphism $\gtg_{P_1}\colon\BM(\mcC_1,P_1)\to\BM(\mcC_2,P_2)$ preserving fibers and $\gtg_{P_1\ast}(\gtm_{1j})=\gtm_{2j}^\alpha$ for $j\in\partial I$. 
Since $h_{P_1\ast}=\gtg_{P_1\ast}$, and the universal covering of a torus is not compact, $h_{P_1}$ is homotopic to $\gtg_{P_1}$ by \cite[Theorem~I]{olum1953}. 
\end{proof}

Note that the homeomorphism $\gtg_{P_1}$ in Lemma~\ref{lem:Homeo_smooth_node} is orientation-preserving by the condition (iii). 
We consider the case where $P_1$ is not a node of $\mcC_1$. 

\begin{lem}[{cf.\ \cite[Theorem~8.2]{neumann1981}}]\label{lem:Fnormal}
If $P_i\in\mcC_i$ is a singular point which is not node, then the F-normal form $\Gamma_\dplumb^\norm(\mcC_i,P_i)$ is obtained from $\Gamma_\dplumb(\mcC_i,P_i)$ by applying the following operations to $\Gamma_\dplumb(\mcC_i,P_i)$ wherever possible: 
\[ \begin{tikzpicture}[scale=0.7]
%%%%%%%%%%%%%%%%%%%%%%%%%%%%%%%%%%%%%%%%%%%%%
	\coordinate (a1) at (-1,2.5);
	\coordinate (a2) at (0.5,2.5);
	\coordinate (a3) at (1.5,3);
	\coordinate (a4) at (1.5,2);
	\coordinate (a11) at (-2,2.8);
	\coordinate (a12) at (-2,2.2);
	\coordinate (a13) at (-2.5,2.95);
	\coordinate (a14) at (-2.5,2.05);
	
	\draw [fill] (a1) circle [radius=2pt] node [above=2pt] {\footnotesize $e_i$} node [below=2.5pt] {\footnotesize $[g_i]$};
	\draw [fill] (a2) circle [radius=2pt] node [above=2pt] {\footnotesize $e$};
	\draw [fill] (a3) circle [radius=2pt] node [right=2pt] {\footnotesize $-2$};
	\draw [fill] (a4) circle [radius=2pt] node [right=2pt] {\footnotesize $-2$};
	\draw (a1) -- (a11);
	\draw (a1) -- (a12);
	\draw (a3) -- (a2) -- (a4);
	\draw[dashed] (a11) -- (a13);
	\draw[dashed] (a12) -- (a14);
	\draw [fill] (-1.9,2.5) circle [radius=0.5pt];
	\draw [fill] (-1.9,2.35) circle [radius=0.5pt];
	\draw [fill] (-1.9,2.65) circle [radius=0.5pt];
	\draw %[thick] 
	(a1) -- (a2);
	
	\draw [->] (3,2.5) -- (4,2.5);
	
	\coordinate (a1) at (6,2.5);
	\coordinate (a2) at (7.5,2.5);
	\coordinate (a3) at (9,2.5);
	\coordinate (a11) at (5,2.8);
	\coordinate (a12) at (5,2.2);
	\coordinate (a13) at (4.5,2.95);
	\coordinate (a14) at (4.5,2.05);
	
	\draw [fill] (a1) circle [radius=2pt] node [above=2pt] {\footnotesize $e_j$} node [below=2.5pt] {\footnotesize $[g_j]$};
	\draw [fill] (a2) circle [radius=2pt] node [above=2pt] {\footnotesize $e+1$};
	\draw [fill] (a3) circle [radius=2pt] node [above=2pt] {\footnotesize $0$} node [below=2pt] {\footnotesize $[-1]$};
	\draw (a1) -- (a11);
	\draw (a1) -- (a12);
	\draw[dashed] (a11) -- (a13);
	\draw[dashed] (a12) -- (a14);
	\draw [fill] (5.1,2.5) circle [radius=0.5pt];
	\draw [fill] (5.1,2.35) circle [radius=0.5pt];
	\draw [fill] (5.1,2.65) circle [radius=0.5pt];
	\draw %[thick] 
	(a1) -- (a3);
	
	\node at (11.5,2.5) {$(e\leq -3)$, };

%%%%%%%%%%%%%%%%%%%%%%%%%%%%%%%%%%%%%%%%%%%%%

	\coordinate (a1) at (-2,0);
	\coordinate (a2) at (1.5,0);
	\coordinate (a3) at (2.5,0.5);
	\coordinate (a4) at (2.5,-0.5);
	\coordinate (a5) at (-1,0);
	\coordinate (a6) at (0,0);
	\coordinate (a11) at (-3,0.3);
	\coordinate (a12) at (-3,-0.3);
	\coordinate (a13) at (-3.5,0.45);
	\coordinate (a14) at (-3.5,-0.45);
	
	\draw [fill] (a1) circle [radius=2pt] node [above=2pt] {\footnotesize $e_j$} node [below=2.5pt] {\footnotesize $[g_j]$};
	\draw [fill] (a2) circle [radius=2pt] node [above=2pt] {\footnotesize $-2$};
	\draw [fill] (a3) circle [radius=2pt] node [right=2pt] {\footnotesize $-2$};
	\draw [fill] (a4) circle [radius=2pt] node [right=2pt] {\footnotesize $-2$};
	\draw [fill] (a5) circle [radius=2pt] node [above=2pt] {\footnotesize $-2$};
	\draw [fill] (a6) circle [radius=2pt] node [above=2pt] {\footnotesize $-2$};
	\draw (a1) -- (a11);
	\draw (a1) -- (a12);
	\draw (a3) -- (a2) -- (a4);
	\draw[dashed] (a11) -- (a13);
	\draw[dashed] (a12) -- (a14);
	\draw [fill] (-2.9,0) circle [radius=0.5pt];
	\draw [fill] (-2.9,0.15) circle [radius=0.5pt];
	\draw [fill] (-2.9,-0.15) circle [radius=0.5pt];
	\draw %[thick] 
	(a1) -- (0.3,0);
	\draw [dashed, %thick
	] (0.3,0) -- (1.2,0);
	\draw %[thick] 
	(1.2,0) -- (a2);
	\draw (-1.2, -0.2) to [out=345, in=195] (1.7, -0.2);
	\node at (0.25,-0.7) {\footnotesize $b$};
	
	\draw [->] (4,0) -- (5,0);
	
	\coordinate (a1) at (7,0);
	\coordinate (a3) at (9,0);
	\coordinate (a11) at (6,0.3);
	\coordinate (a12) at (6,-0.3);
	\coordinate (a13) at (5.5,0.45);
	\coordinate (a14) at (5.5,-0.45);
	
	\draw [fill] (a1) circle [radius=2pt] node [above=2pt] {\footnotesize $e_j+1$} node [below=2.5pt] {\footnotesize $[g_j]$};
	\draw [fill] (a3) circle [radius=2pt] node [above=2pt] {\footnotesize $b$} node [below=2pt] {\footnotesize $[-1]$};
	\draw (a1) -- (a11);
	\draw (a1) -- (a12);
	\draw %[thick] 
	(a3) -- (a1);
	\draw[dashed] (a11) -- (a13);
	\draw[dashed] (a12) -- (a14);
	\draw [fill] (6.1,0) circle [radius=0.5pt];
	\draw [fill] (6.1,0.15) circle [radius=0.5pt];
	\draw [fill] (6.1,-0.15) circle [radius=0.5pt];
	
	\node at (12,0) {$\left\{ \begin{array}{l} b\geq 1, \\ b \mbox{ \rm maximal}. \end{array} \right.$ };

\end{tikzpicture} \]
Furthermore, $\Gamma_\dplumb^\norm(\mcC_i,P_i)$ satisfies the {\red following}: 
\begin{enumerate}
	\item All edges of $\Gamma_\dplumb^\norm(\mcC_i,P_i)$ are $(+)$-edges. 
	\item Any vertex $v_i$ of $\Gamma_\dplumb^\norm(\mcC_i,P_i)$ satisfies $\bm{g}(v_i)\geq -1$. 
	\item Any vertex $v_i$ of $\Gamma_\dplumb^\norm(\mcC_i,P_i)$ with $\bm{g}(v_i)=-1$ satisfies $\bm{d}(v_i)=1$, $\bm{e}(v_i)\geq0$, and if $\bm{e}(v_i)=0$, the maximal chain adjoining $v_i$ has length $\geq 1$. 
	\item $\Gamma_\dplumb(\mcC_i,P_i)$ is uniquely determined by $\Gamma_\dplumb^\norm(\mcC_i,P_i)$.  
\end{enumerate}
\end{lem}
\begin{proof}
Since the intersection form of $\Gamma_\dplumb(\mcC_i,P_i)$ is negative definite (cf.\ \cite{mumford1961}), we can apply the argument of \cite[Section~8]{neumann1981}. 
By \cite[Theorem~8.2]{neumann1981}, $\Gamma_\dplumb^\norm(\mcC_i,P_i)$ is obtained by the operations in the statement since $\Gamma_\dplumb(\mcC_i,P_i)$ has boundary vertices. 
It is clear that (i) holds. 
Since any vertex of $\Gamma_\dplumb(\mcC_i,P_i)$ has non-negative genus, any vertex $v_i$ of $\Gamma_\dplumb^\norm(\mcC_i,P_i)$ with $\bm{g}(v_i)\leq-1$ is given by the operations in the statement. 
Therefore, we obtain (ii), (iii) and (iv).  
\end{proof}

\begin{rem}
	The first operation in Lemma~\ref{lem:Fnormal} can be done by $-e-1$ times blowing-ups (${\rm R1}^-$), an $\RR\PP^2$-absorption (R2), a blowing-up (${\rm R1}^+$) and $-e-2$ times blowing-downs (${\rm R1}^+$) (see Remark~\ref{rem:operation}). 
The second operation in Lemma~\ref{lem:Fnormal} can be done by a blowing-up (${\rm R1}^-$), $b-1$ times blowing-downs (${\rm R1}^+$) and an $\RR\PP^2$-absorption (R2). 
\end{rem}

\begin{lem}\label{lem:existence of gtg'}
	Assume that $P_1\in\Sing(\mcC_1)$ is not a node. 
	Then there exists a homeomorphism $\gtg_{P_1}'\colon\BM(\mcC_1,P_1)\to\BM(\mcC_2,P_2)$ homotopic to $h_{P_1}$. 
\end{lem}
\begin{proof}
Let $\Gamma_\dplumb^\norm(\mcC_i,P_i)$ be the F-normal form of $\Gamma_\dplumb(\mcC_i,P_i)$. 
By Lemma~\ref{lem:Fnormal}, $\Gamma_\dplumb^\norm(\mcC_i,P_i)$ is connected. 
For each $i=1,2$, let $\mcS_i$ be a subset of edges of $\Gamma_\dplumb^\norm(\mcC_i,P_i)$ such that $\mcS_i$ contains precisely one edge of each maximal chain of $\Gamma_\dplumb^\norm(\mcC_i,P_i)$ which is not incident to a boundary vertex (see \cite[\S 4]{neumann1981} or Appendix \ref{sec:pg vs Wg}). 
By Theorem~\ref{thm:red_str}, $\red \TT_{\mcS_i}:=\bigcup_{y_i\in\mcS_i}T_{y_i}$ is the reduced graph structure of $\BM(\mcC_i,P_i)$. 
Hence $\BM(\mcC_i,P_i)$ can be regarded as a reduced graph manifold. 
Then $\BM(\mcC_i,P_i)$ is irreducible by \cite[Theomre~7.1]{waldhausen_1967}. 
Since $\BM(\mcC_i,P_i)$ is not a solid torus, $\BM(\mcC_i,P_i)$ is boundary-irreducible and sufficiently large by \cite[Lemma~7.2]{waldhausen_1967}. 
Hence it follows from \cite[Theorem~6.1]{waldhausen_1968} that there is a homeomorphism $\gtg_{P_1}'\colon\BM(\mcC_1,P_1)\to\BM(\mcC_2,P_2)$ such that $\gtg'_{P_1}$ is homotopic to $h_{P_1}$. 
\end{proof}

We prove that the homeomorphism $\gtg'_{P_1}$ is orientation-preserving. 

\begin{lem}\label{lem:orientation_gen_sing}
Assume that $P_1\in\Sing(\mcC_1)$ is not a node of $\mcC_1$, and let $\gtg_{P_1}'$ be the homeomorphism in Lemma~\ref{lem:existence of gtg'}. 
Then $\gtg'_{P_1}$ is orientation-preserving. 
\end{lem}
\begin{proof}
Assume that $\gtg_{P_1}'$ is orientation-reversing. 
Let $y_{i1}\in\mcY(\Gamma_{i,P_i})$ be edges with $t({\red y_{i1}})\in\AH(\Gamma_{i,P_i})$ and $\xi_{h,P_1}(t({\red y_{11}}))={\red y_{21}}$. 
Since the isomorphism $\varphi_{P_1}$ satisfies $\varphi_{P_1}|_{\AH(\mcC_{1,P_1})}=\xi_{h,P_1}$, 
the maximal chains incident to $\red y_{i1}$ are the same form,
\begin{align}
\begin{aligned}\label{eq:max_chain}
\begin{tikzpicture}
	\coordinate (a0) at (0,0) node [left] {$v_{i1}$}; 
	\coordinate (a1) at (1.5,0);
	\coordinate (a2) at (3,0);
	\coordinate (a21) at (3.5,0);
	\coordinate (a22) at (4.5,0);
	\coordinate (a3) at (5,0);
	\coordinate (b0) at (6.5,0);
	\coordinate (b11) at (7.5,0.2);
	\coordinate (b12) at (8,0.3);
	\coordinate (b21) at (7.5,-0.2);
	\coordinate (b22) at (8,-0.3);
	\draw[fill] (7.7,0.13) circle [radius=0.5pt];
	\draw[fill] (7.7,0) circle [radius=0.5pt];
	\draw[fill] (7.7,-0.13) circle [radius=0.5pt];
	\draw[fill] (a1) circle [radius=2pt] node [above] {$-b_1$};
	\draw[fill] (a2) circle [radius=2pt] node [above] {$-b_2$};
	\draw[fill] (a3) circle [radius=2pt] node [above] {$-b_k$};
	\draw[fill] (b0) circle [radius=2pt] node [above] {$v_{i2}$};
	\draw [-{Stealth[length=8pt]}] (a21) -- (a1) -- node [below] {$y_{i1}$} (a0);
	\draw (a22) -- (a3) -- node [below] {$y_{i2}$} (b0);
	\draw[dashed] (b21) -- (b22);
	\draw (b21) -- (b0) -- (b11);
	\draw[dashed] (b11) -- (b12);
	\draw[dashed] (a21) -- (a22);
\end{tikzpicture} 
\end{aligned}
\end{align}
where $-b_j\leq -2$ are the {\red Euler} numbers of corresponding vertices, and $y_{i2}\in\mcY(\Gamma_{i,P_i})$ are edges such that $v_{i2}=t(y_{i2})$ and $\bm{d}(v_{i2})\geq 3$ for $i=1,2$. 
Let $v_{i1}\in\Str_i$ be the vertices corresponding to $t(y_{i1})\in\AH(\Gamma_{i,P_i})$. 
Let $T_{ij}$ be the boundary component $\red T_{y_{ij}}$ of $M_{v_{ij}}$ corresponding $y_{ij}$ for each $i,j=1,2$. 
By \cite[Lemma 5.3]{neumann1981}, $T_{i1}$ and $T_{i2}$ are identified in $\BM(\mcC_i)$ via homeomorphism given by the matrix 
\[ B_k:=\begin{pmatrix}
d_k & c_k \\ - d_{k-1} & -c_{k-1}
\end{pmatrix}, \]
where $c_i, d_i\in\ZZ$ are coprime integers such that $c_i/d_i$ is equal to the continued fraction $[b_1,\dots,b_i]$ for each $i=1,\dots,k$ (see \cite[Lemma 5.2]{neumann1981} for details). 
Note that $c_{-1}=0$, $c_0=1$, $d_{-1}=-1$, $d_0=0$, and $\det B_k=-1$. 
By \cite[Proposition~3.4]{abst2023} and \cite[Theorem~5.5]{waldhausen_1967}, we may assume that $\red \tau_{y_{2j}}\circ \gtg'_{P_1}\circ\tau_{y_{1j}}^{-1}$ is the homeomorphism of $S^1\times S^1$ given by the matrix 
\[ H_j:=\begin{pmatrix}
	\alpha_j&0\\ \beta_j & -\alpha_j
\end{pmatrix}  \]
for each $j=1,2$, where $\alpha_j\in\{\pm1\}$ and $\beta_j\in\ZZ$. 
Hence we have the following commutative diagram: 
\begin{align} 
\begin{aligned}\label{eq:diagram_boundary}
\begin{tikzpicture}
	\node (a1) at (-1.2,0.5) {$T_{11}$};
	\node (a2) at (-1.2,-0.5) {$T_{12}$};
	\node (b1) at (1.2,0.5) {$T_{21}$};
	\node (b2) at (1.2,-0.5) {$T_{22}$};
	\node (a3) at (-4.2,1.3) {$S^1\times S^1$};
	\node (a4) at (-4.2,-1.3) {$S^1\times S^1$};
	\node (b3) at (4.2,1.3) {$S^1\times S^1$};
	\node (b4) at (4.2,-1.3) {$S^1\times S^1$};
	\draw[double distance=2pt] (a1) -- (a2);
	\draw[double distance=2pt] (b1) -- (b2);
	\draw[->] (a1) -- node [above] {\footnotesize $\gtg_{P_1}'$} (b1);
	\draw[->] (a2) -- node [above] {\footnotesize $\gtg_{P_1}'$} (b2);
	\draw[->] (a3) -- node [left] {\footnotesize $B_k$} (a4);
	\draw[->] (b3) -- node [right] {\footnotesize $B_k$} (b4);
	\draw[->] (a3) -- node [above] {\footnotesize $H_1$} (b3);
	\draw[->] (a4) -- node [above] {\footnotesize $H_2$} (b4);
	\draw[->] (a3) -- node [below] {\footnotesize $\red \tau_{y_{11}}$} (a1);
	\draw[->] (a4) -- node [above] {\footnotesize $\red \tau_{y_{12}}$} (a2);
	\draw[->] (b3) -- node [right=8pt, below] {\footnotesize $\red \tau_{y_{21}}$} (b1);
	\draw[->] (b4) -- node [right=8pt, above] {\footnotesize $\red \tau_{y_{22}}$} (b2);
\end{tikzpicture} 
\end{aligned}
 \end{align}
The $(1,2)$ entry of $B_kH_1B_k^{-1}=H_2$ is 
$(2\alpha_1 d_k+\beta_1 c_k)c_k=0$. 
Since $\gcd(c_k,d_k)=1$, $c_k\geq k+1$ and $d_k\geq k$ by \cite[Lemma 5.2 (ii)]{neumann1981}, we obtain either $d_k=0$ or $c_k=2$. 
If $d_k=0$, then $k=0$, $\beta_1=0$, and $H_2=-H_1$. 
If $c_k=2$, then we have $k=1$, $d_k=1$, $b_1=2$, $\beta_1=-\alpha_1$ and $H_2=-H_1$. 
Thus, we can construct $\Gamma_\dplumb^\norm(\mcC_2,P_2)$ from $-\Gamma_\dplumb^\norm(\mcC_1,P_1)$ by Theorem~\ref{thm:ori_rev}. 

By the above fact, we can apply the argument of \cite[Section 9]{neumann1981} as follows. 
By Theorem~\ref{thm:ori_rev} (ii) and Lemma~\ref{lem:Fnormal} (iii), $\Gamma_\dplumb^\norm(\mcC_i,P_i)$ for $i=1,2$ have no vertex $v$ with $\bm{g}(v)=-1$. 
Hence $\Gamma_\dplumb(\mcC_i,P_i)=\Gamma_\dplumb^\norm(\mcC_i,P_i)$ for $i=1,2$  by Lemma~\ref{lem:Fnormal}, and $\sign(y_i)=+1$ for any $y_i\in\mcY(\Gamma_{i,P_i})$. 
Since $\Gamma_{i,P_i}$ is the dual graph of embedded resolution of the singularity $P_i\in\mcC_i$, $\Gamma_{i,P_i}$ has no cycle for each $i=1,2$. 
Moreover, there is $v_1\in\mcV(\Gamma_{1,P_i})$ with $\bm{d}(v_1)\geq3$ since $P_1$ is not node. 
Let 
\[ \Gamma_{\dplumb}^{1,v_1}=(\Gamma_{1,P_1}^{v_1},\bm{g}_{1,P_1}^{v_1},\bm{e}_{1,P_1}^{v_1};\Ver,\AH)\] 
be the full subgraph of $\Gamma_\dplumb(\mcC_1,P_1)$ such that $\mcV(\Gamma_{1,P_1}^{v_1})$ consists of $v_1$, all vertices on maximal chains adjoining $v_1$ and all boundary vertices adjacent to $v_1$. 
The decorated plumbing graph $\Gamma_\dplumb^{1,v_1}$ is a star-shaped decorated plumbing graph, and $\PM(\Gamma_\dplumb^{1,v_1})$ is a Seifelt manifold. 
Let $\widetilde{\Gamma}_\dplumb^{1,v_1}$ be the F-normal form of $-\Gamma_\dplumb^{1,v_1}$. 
By Theorem~\ref{thm:ori_rev}, $\widetilde{\Gamma}_\dplumb^{1,v_1}$ is also a star-shaped decorated plumbing graph, and is a subgraph of $\Gamma_\dplumb(\mcC_2,P_2)$. 
Since the intersection form $S(\Gamma_\dplumb^{1,v_1})$ is negative definite, the {\red Euler} number of the Seifelt manifold $\PM(\Gamma_\dplumb^{1,v_1})$ is negative by Theorem~\ref{thm:Seifelt2}. 
Thus the {\red Euler} number of $\PM(\widetilde{\Gamma}_\dplumb^{1,v_1})=\PM(-\Gamma_\dplumb^{1,v_1})$ is positive by the equality (\ref{eq:Seiflt}), which is contradiction to Theorem~\ref{thm:Seifert1} since $\widetilde{\Gamma}_\dplumb^{1,v_1}$ is a subgraph of $\Gamma_\dplumb(\mcC_2,P_2)$, and the intersection form $S(\widetilde{\Gamma}_\dplumb^{1,v_1})$ is negative definite. 
Therefore, $\gtg'_{P_1}$ is orientation-preserving. 
\end{proof}

\begin{lem}\label{lem:Homeo_Sing}
Assume that $P_1\in\Sing\mcC_1$ is not a node of $\mcC_1$. 
There exists a homeomorphism $\gtg_{P_1}\colon\BM(\mcC_1,P_1)\to\BM(\mcC_2,P_2)$ such that 
\begin{enumerate}
	\item $\gtg_{P_1}$ is homotopic to $h_{P_1}$;
	\item $\gtg_{P_1}(\TT_{1,P_1})=\TT_{2,P_2}$;
	\item $\gtg_{P_1}({\red M_{v_1}})={\red M_{\varphi_{P_1}(v_1)}}$ for each $v_1\in\Ver(\Gamma_{1})$, where $\varphi_{P_1}\colon \Gamma_{\dplumb}(\mcC_1,P_1)\to\Gamma_{\dplumb}(\mcC_2,P_2)$ is the isomorphism in {\rm (\ref{eq:isom_dplumb})}. 
	\item there is $\alpha\in\{\pm1\}$ such that ${\red \tau_{y_2}\circ\gtg_{P_1}\circ\tau_{y_1}}\colon S^1\times S^1\to S^1\times S^1$ is given by $\alpha E$ for any $y_1\in\mcY(\Gamma_{1,P_1})$ with $t(y_1)\in\AH(\Gamma_{1,P_1})$ and $y_2:=\varphi_{P_1}(y_1)$; 
	\item for each $v_1\in\Ver(\Gamma_{1,P_1})$, $\gtg_{P_1}$ preserves fibers of $v_1$ and $v_2=\varphi_{P_1}(v_1)\in\Ver(\Gamma_{2,P_2})$. 
\end{enumerate}
Moreover, $\alpha=1$ if and only if $h$ maps a meridian of $\red D_{t(y_1)}$ to a meridian of $\red D_{t(y_2)}$ for some $v_i\in\Ver(\Gamma_{i,P_i})$ with $t(y_i)\in\AH(\Gamma_{i,P_i})$ and $y_2=\varphi_{P_1}(y_1)$. 
\end{lem}

\begin{proof}
We construct {\red the} desired homeomorphism $\gtg_{P_1}$ by deforming $\gtg'_{P_1}$ in Lemma~\ref{lem:existence of gtg'}. 
Let $\mcS_i$ be the subset of edges of $\Gamma_\dplumb^\norm(\mcC_i,P_i)$ as in the proof of Lemma~\ref{lem:existence of gtg'}, and put $\red \TT_{\mcS_i}:=\bigcup_{y_i\in\mcS_i}T_{y_i}$ which is a reduce graph structure of $\BM(\mcC_i,P_i)$. 
By \cite[Theorem~8.1]{waldhausen_1967} and Lemma~\ref{lem:orientation_gen_sing}, there is an isotopy $\gth_{u}\colon \BM(\mcC_1,P_1)\to\BM(\mcC_2,P_2)$ ($0\leq u\leq 1$) which is constant on $\partial\BM(\mcC_i,P_i)$, and satisfies $\gth_{0}=\gtg'_{P_1}$ and $\red \gth_1(\TT_{\mcS_1})=\TT_{\mcS_2}$.

Let $y_{i1}\in\mcY(\Gamma_{i,P_i})$ be edges satisfying $t(y_{i1})\in\AH(\Gamma_{i,P_i})$ and $\varphi_{P_1}(y_{11})=y_{21}$. 
For each $i=1,2$, let $\gtc_i$ be the maximal chain incident to $t(y_{i1})$, and let $y_{i2}\in\mcY(\Gamma_{i,P_i})$ be the edge on $\gtc_i$ such that $\bm{d}(t(y_{i2}))>2$ as in (\ref{eq:max_chain}). 
By the same argument in the proof of Lemma~\ref{lem:orientation_gen_sing}, 
we may assume that $\red \tau_{y_{2j}}\circ\gtg'_{P_1}\circ\tau_{y_{1j}}^{-1}$ is given by 
\[ H_j:=\begin{pmatrix}
	\alpha_j&0\\ \beta_j & \alpha_j
\end{pmatrix}  \]
for each $j=1,2$, where $\alpha_j\in\{\pm1\}$ and $\beta_j\in\ZZ$. 
It follows from the diagram (\ref{eq:diagram_boundary}) that $\alpha_1=\alpha_2$ and $\beta_j=0$ for $j=1,2$. 

For $y_i\in\mcS_i$ ($i=1,2$) with $\red \gth_{1}(T_{y_1})=T_{y_2}$, the maximal chains of $\Gamma_\dplumb^\norm(\mcC_i,P_i)$ containing $y_i$ have the same number of edges. 
Since each maximal chain of $\Gamma_\dplumb^\norm(\mcC_i,P_i)$ corresponds to a part homeomorphic to $I\times S^1\times S^1$ in $\BM(\mcC_i,P_i)$, 
there is an isotopy $\gth_{u}\colon \BM(\mcC_1,P_1)\to\BM(\mcC_2,P_2)$ ($1\leq u\leq 2$) which is constant on $\partial\BM(\mcC_i,P_i)$, and satisfies $\gth_{2}(\TT_{1,P_1}^\norm)=\TT_{2,P_2}^\norm$, 
where $\TT_{i,P_i}^\norm$ is the graph structure corresponding to $\Gamma_\dplumb^\norm(\mcC_i,P_i)$. 
Tori in $\TT_{i,P_i}$ exchanged by the operations in Lemma~\ref{lem:Fnormal} are in the $S^1$-bundle over the M\"obius band studied in \cite[Section~3]{waldhausen_1967}, and the tori give reduced graph structure of the $S^1$-bundle. 
Hence there is an isotopy $\gth_{u}\colon \BM(\mcC_1,P_1)\to\BM(\mcC_2,P_2)$ ($2\leq u\leq 3$) by Theorem~\ref{thm:Wald-2}, which is constant on $\partial\BM(\mcC_i,P_i)$, and satisfies $\gth_{3}(\TT_{1,P_1})=\TT_{2,P_2}$.

If $v_{1}\in\Ver(\Gamma_{1,P_1})$ is not on any chain of $\Gamma_\dplumb(\mcC_1,P_1)$, and $v_2:=\varphi_{P_1}(v_1)\in\Ver(\Gamma_{2,P_2})$, then, for a fiber $\red F_{v_1}$ of $v_1$,  the image $\red \gth_{3}(F_{v_1})$ is ambient isotopic in $\red M_{v_2}$ to a fiber $\red F_{v_2}$ of $v_2$ by Theorem~\ref{thm:S1-bdl_gene}. 
If $v_1\in\Ver(\Gamma_{1,P_1})$ is on a chain of $\Gamma_\dplumb(\mcC_1,P_1)$, then 
the argument of \cite[p.322, 323]{neumann1981} shows that,  for a fiber $\red F_{v_1}$ of $v_1$, the image $\red \gth_{3}(F_{v_1})$ is homotopic in $\red M_{v_2}$ to a fiber $\red F_{v_2}$ of $v_2$ (cf. \cite[Chapter~II]{laufer1971}). 

For $y_1\in\mcY(\Gamma_{1,P_1})$ with $t(y_1), t(\overline{y}_1)\notin\AH(\Gamma_{1,P_1})$, 
we can identify $\red T_{\overline{y}_1}$ with $\red T_{y_1}$ in $\BM(\mcC_1,P_1)^\circ$. 
Let $y_{i,1},\dots,y_{i,r}\in\mcY(\Gamma_{i,P_i})$ be edges such that 
\begin{align*}
	& t(y_{i,j})\in\Ver(\Gamma_{i,P_i}),  && y_{i,j'}\notin\{y_{i,j},\overline{y}_{i,j}\} \quad \mbox{if $j\ne j'$}, 
	\\
	& \varphi_P(y_{1,j})=y_{2,j}, && \mcY(\Gamma_{i,P_i})=\{y_{i,j}, \overline{y}_{i,j}\mid j=1,\dots,r\}. 
\end{align*}  
By Lemma~\ref{lem:S1-bdl_torus}, there is an isotopy $\gth_{u}$ ($3\leq u\leq 4$) from $\gth_{3}$, which is constant on $\partial\BM(\mcC_i,P_i)$ and preserves graph structures $\TT_{i,P_i}$, such that $\gth_{4}$ maps every fiber of $t(\overline{y}_{1,j})$ in $\red T_{\overline{y}_{1,j}}$ to a fiber of $t(\overline{y}_{2,j})$ in $\red T_{\overline{y}_{2,j}}$. 
Since each fiber $F_{i,j}\subset {\red T_{y_{i,j}}}$ of $t(y_{i,j})$ intersects at just one point with each fiber $\red \overline{F}_{i,j}\subset T_{\overline{y}_{i,j}}=T_{y_{i,j}}$ of $t(\overline{y}_{i,j})$, there is an isotopy $\gth_{u}$ ($4\leq u\leq 5$) from $\gth_{4}$ such that 
\begin{align*}
	& \gth_{u}(\TT_{1,P_1})=\TT_{2,P_2}, 
	&& \gth_{u}(\overline{F}_{1,j})=\gth_{4}(\overline{F}_{1,j})  
\end{align*}
for each fiber $\overline{F}_{1,j}\subset {\red T_{\overline{y}_{1,j}}}$ of $t(\overline{y}_{1,j})$ and any $4\leq u\leq 5$, and for each fiber $\red F_{1,y}\subset T_{y_1}$ of $\red t(y_1)$ with $\red y_1\in\mcY(\Gamma_{1,P_1})$, $\red \gth_5(F_{1,y_1})$ is a fiber of $\red t(\varphi_{P_1}(y_1))$. 
By Lemma~\ref{lem:S1-bdl_solidtorus}, \ref{lem:S1-bdl_over_annulus}, Theorem~\ref{thm:S1-bdl_gene} and Remark~\ref{rem:S1-bdl_gene}, there is an isotopy $\gth_{u}$ ($5\leq u\leq 6$) from $\gth_{5}$ such that 
\begin{align*}
	& \gth_{u}|_{\TT_{1,P_1}}=\gth_{5}|_{\TT_{1,P_1}}, 
	&& \gth_{u}|_{\partial\BM(\mcC_1,P_1)}=\gth_{5}|_{\partial\BM(\mcC_1,P_1)} 
\end{align*}
for any $5\leq u\leq6$, 
and $\gth_{6}$ maps every fiber of all $v_1\in\Ver(\Gamma_{1,P_1})$ to that of $\varphi_{P_1}(v_1)$. 
Put $\gtg_{P_1}:=\gth_{6}$, which satisfies (i), (ii), (iii) and (v). 
Furthermore, it follows from \cite[Remark~3.5]{abst2023} that $\gtg_{P_1}$ satisfies (iv). 
\end{proof}

By identifying $\red Y_i\setminus \mcC_i'$ with $\PP^2\setminus\mcC_i$, let $\red \gtm_{v_i}\in\pi_1(\PP^2\setminus\mcC_i)$ be a meridian of $\red D'_{v_i}\subset \mcC_i'$ for $v_i\in\mcV(\Gamma_i)$. 
Theorem~\ref{thm:meridian} follows from the next theorem. 

\begin{thm}\label{thm:homeo_boundary}
	Let $\mcC_1,\mcC_2\subset\PP^2$ be two plane curves. 
	Assume that there is a homeomorphism $h\colon U(\mcC_1)\to U(\mcC_2)$ with $h(\mcC_1)=\mcC_2$. 
	Let $\varphi\colon\Comb(\mcC_1)\to\Comb(\mcC_2)$ be the equivalence map constructed in Remark~\ref{rem:fact_comb}. 
	Then there exists a homeomorphism $\Psi_\BM\colon\BM(\mcC_1)\to\BM(\mcC_2)$ such that  
	\begin{enumerate}
		\item $\Psi_{\BM}$ is homotopic to the restriction of $h$ in $U(\mcC_2)\setminus\mcC_2$; 
		\item $\red \Psi_\BM(T_{y_1})=T_{\varphi(y_1)}$ for each $y_1\in\mcY(\Gamma_1)$;
		\item $\red \Psi_\BM(M_{v_1})=M_{\varphi(v_1)}$ for each $v_1\in\mcV(\Gamma_1)$;
		\item there is $\alpha\in\{\pm1\}$ such that $\red \Psi_{\BM\ast}(\gtm_{v_1})=\gtm_{v_2}^\alpha$ up to conjugate for any $v_i\in\mcV(\Gamma_i)$ with $v_2=\varphi(v_1)$. 
	\end{enumerate}
In particular, $\Psi_{\BM\ast}=\Psi_h^\infty$. 
\end{thm}
\begin{proof}
Let $h'\colon \BM(\mcC_1)\to\BM(\mcC_2)$ be the map constructed in Remark~\ref{rem:comb_smooth}.  
%By the isotopy theorem for closed tubular neighborhoods \cite[Theorem~6.5]{hirsch_1976}, we may assume that $h(M_{1,v_1})=M_{2,v_2}$ for each $v_1\in\Str_1$ and $v_2=\varphi(v_2)$. 
For $v_1\in\Str_1$ and $v_2=\varphi(v_1)$, since $\bm{e}_1(v_1)=\bm{e}_2(v_2)$, we may also assume that $h'$ preserves fibers and trivializations of boundary components of $\red p_{v_i}\colon M_{v_i}\to B_{v_i}$. 
By Lemma~\ref{lem:Homeo_smooth_node} and \ref{lem:Homeo_Sing}, 
we obtain a desired homeomorphism $\Psi_\BM\colon\BM(\mcC_1)\to\BM(\mcC_2)$ after a certain homotopy near $\bigcup_{P_2\in\Sing(\mcC_2)}U(P_2)$. 
\end{proof}

%%%%%%%%%%%%%%%%%%%%%%%%%%%%%%%%
\section{$G$-combinatorial type of plane curves}\label{sec:G-comb}
%%%%%%%%%%%%%%%%%%%%%%%%%%%%%%%%

In this section, we define \textit{$G$-combinatorial type} as a modified plumbing graph together with certain information, which describe the embedded topology of ${\phi}^{-1}(\mcC)$ in its regular neighborhood for a plane curve $\mcC\subset\PP^2$ and {\red the $G$-cover $\phi\colon X\to\PP^2$ given by a surjection $\theta\colon\pi_1(\PP^2\setminus\mcC)\to G$}. 

Let $\mcC\subset\PP^2$ be a plane curve, and let $\sigma\colon Y\to\PP^2$ be the minimal good embedded resolution of $\mcC$, and put $\mcC':=\sigma^{-1}(\mcC)$. 
Let $\Comb(\mcC)=(\Gamma_\mcC,\bm{g}_\mcC,\bm{e}_\mcC;\Str_\mcC)$ be the combinatorics of $\mcC$, let $p_v\colon M_v\to B_v$ and $T_y\subset M_{t(y)}$ be the $S^1$-bundle of $v\in\mcV(\Gamma_\mcC)$ and the torus corresponding to $y\in\mcY(\Gamma_\mcC)$, respectively, as in the previous section. 
Let $G$ be a finite group. 
We assume that there is a surjective homomorphism $\theta\colon \pi_1(\PP^2\setminus\mcC)\to G$. 
Then $\theta$ induces the $G$-covers $\phi\colon X\to\PP^2$ and $\tilde{\phi}\colon \widetilde{X}\to Y$ branched along sub-curves of $\mcC$ and $\mcC'$, respectively. 
The minimal resolution $\tilde{\sigma}'\colon\widetilde{X}'\to \widetilde{X}$ is given in the proof of \cite[Proposition~2.2]{laufer1971} explicitly. 
By Stein factorization, there is a morphism $\tilde{\sigma}\colon \widetilde{X}\to X$ such that the following diagram is commutative: 
\begin{align*}%\label{eq:diagram}
\begin{aligned}
\begin{tikzpicture}
	\node (a1) at (0, 0) {$\PP^2$};
	\node (a2) at (2.5, 0) {$Y$};
	\node (b1) at (0, 1.3) {$X$};
	\node (b2) at (2.5, 1.3) {$\widetilde{X}$};
	\node (b3) at (5, 1.3) {$\widetilde{X}'$};
	\draw (a1) node [left=6pt] {$\,\mcC\,\subset$};
	\draw (a2) node [right=5pt] {$\supset\, \mcC'$};
	\draw [->] (b1) to node [left] {\footnotesize $\phi$} (a1);
	\draw [->] (b2) to node [right] {\footnotesize $\tilde{\phi}$} (a2);
	\draw [->] (b2) to node [above] {\footnotesize $\tilde{\sigma}$} (b1);
	\draw [->] (a2) to node [above] {\footnotesize $\sigma$} (a1);
	\draw [->] (b3) to node [above] {\footnotesize $\tilde{\sigma}'$} (b2);
\end{tikzpicture}
\end{aligned}
\end{align*}
Put $\widetilde{\mcC}:=\tilde{\phi}^{-1}(\mcC')\subset \widetilde{X}$, and $\BM^\theta(\mcC):=\phi^{-1}(\BM(\mcC))$. 
We call $\BM^\theta(\mcC)$ a \textit{boundary manifold} of $\widetilde{\mcC}$ in $\widetilde{X}$. 
The modified plumbing graph $\Gamma_\plumb^\modi(\mcC,\theta)$ corresponding to $\BM^\theta(\mcC)$ is defined in \cite{hironaka2000}, which can be constructed as follows:

It follows from the minimal resolution $\tilde{\sigma}'$ that $\widetilde{\mcC}\subset \widetilde{X}$ consists of just two branches locally at each ${P}\in\Sing(\widetilde{\mcC})=\tilde{\phi}^{-1}(\Sing(\mcC'))$. 
Since each irreducible component of $\mcC'$ is smooth, if $D_1\cap D_2\ne\emptyset$ for $D_1,D_2\in\Irr(\widetilde{\mcC})$ with $D_1\ne D_2$, then $\tilde{\phi}(D_1)\ne\tilde{\phi}(D_2)$. 
Hence we can define the dual graph $\Gamma_\mcC^\theta$ of $\widetilde{\mcC}$ as follows by using an arbitrary total order $\prec$ on $\Irr(\mcC')$:
\begin{itemize}
	\item $\mcV(\Gamma_\mcC^\theta):=\left\{v_D \ \middle| \  D\in\Irr\big( \widetilde{\mcC} \big)\right\}$ 
	and $\mcY(\Gamma_\mcC^\theta):=\left\{ y_{{P}}, \  \overline{y}_{P} \ \middle| \ {P}\in\Sing(\widetilde{\mcC}) \right\}$,
	\item if ${P}\in D_1\cap D_2$ for $D_1,D_2\in\Irr(\widetilde{\mcC})$ with $\red \tilde{\phi}(D_1)\precneqq \tilde{\phi}(D_2)$, then  $o(y_P):=v_{D_1}$ and $t(y_P):=v_{D_2}$. 
\end{itemize}
Let $i_v\colon M_v\to \PP^2\setminus\mcC$ and $i_y\colon T_y\to\PP^2\setminus\mcC$ be inclusions for each $v\in\mcV(\Gamma_\mcC)$ and $y\in\mcY(\Gamma_\mcC)$, respectively, and put
\begin{align*}
\begin{aligned}
	G_v&:=\theta\circ {i_v}_\ast\big(\pi_1(M_v)\big),%\subset G & & \mbox{for each $v\in\mcV(\Gamma_\mcC)$, }
	&
	H_v&:=\theta\circ {i_v}_\ast\Big( \big\langle  [F_v]  \big\rangle  \Big), %\subset G & & \mbox{for each $v\in\mcV(\Gamma_\mcC)$, }
	&
	G_y&:=\theta\circ {i_y}_\ast\big(\pi_1(T_y)\big)%\subset G && \mbox{for each $y\in\mcY(\Gamma_\mcC)$, }
\end{aligned}
\end{align*}
for each $v\in\mcV(\Gamma_\mcC)$ and $y\in\mcY(\Gamma_\mcC)$,
where $F_v$ is a fiber of $p_v$, and $\langle [F_v]\rangle$ is the subgroup of $\pi_1(M_v)$ generated by the class $[F_v]$. 
The subgroups $G_v$, $H_v$ and $G_y$ are uniquely determined up to conjugate in $G$. 
Since the kernel of $\theta_y:=\theta\circ {i_y}_\ast$ has a finit index in $\pi_1(T_y)\cong\ZZ\oplus\ZZ$ for each $y\in\mcY(\Gamma_\mcC)$, we have $\ker \theta_y\cong \ZZ\oplus\ZZ$. 
We take the basis of $\pi_1(T_y)$ according to the trivialization $\tau_y\colon S^1\times S^1\to T_y$ for each $y\in\mcY(\Gamma_\mcC)$. 
We fix the unique basis of $\ker \theta_y\cong \ZZ\oplus\ZZ$ so that the inclusion 
$\ker\theta_y\to\pi_1(S^1\times S^2)$ induced by $\tau_y$
%\[ \ker \theta_y\to\pi_1(T_y) \overset{{\tau_y}_\ast^{-1}}{\longrightarrow} \pi_1(S^1\times S^1) \]
is represent by a matrix
\begin{align}\label{eq:matrix} 
\bm{m}_\theta(y):=\begin{pmatrix} c(y) & 0 \\ b(y) & a(y) \end{pmatrix} 
\end{align}
with $0\leq b(y)<a(y)$ and $c(y)>0$. 
We can regard $\bm{m}_\theta$ as a map $\bm{m}_\theta\colon \mcY(\Gamma_\mcC)\to \fin(S^1\times S^1)$. 
Since the following diagram is commutative, 
\[ \begin{tikzpicture}
	\node (a1) at (0,1.3) {$S^1\times S^1$};
	\node (a2) at (4,1.3) {$S^1\times S^1$};
	\node (b1) at (0,0) {$T_{\overline{y}}$};
	\node (b2) at (4,0) {$T_{y}$};
	\draw [->] (a1) to node [above] {\footnotesize $J=(\begin{smallmatrix} 0&1\\1&0 \end{smallmatrix})$} (a2);
	\draw [->] (b1) to node [above] {\footnotesize $\sim$} (b2);
	\draw [->] (a1) to node [left] {\footnotesize $\tau_{\overline{y}}$} (b1);
	\draw [->] (a2) to node [right] {\footnotesize $\tau_{y}$} (b2);
\end{tikzpicture}\]
$\bm{m}_\theta(y)_\ast$ and $(J\bm{m}_\theta(\overline{y}))_\ast$ have the same image in $\pi_1(S^1\times S^1)\cong\ZZ\oplus\ZZ$. 
We define $\bm{g}_\mcC^\theta\colon \mcV(\Gamma_\mcC^\theta)\to\ZZ$, $\bm{e}_\mcC^\theta\colon \mcV(\Gamma_\mcC^\theta)\to \ZZ$ and $\bm{m}_\mcC^\theta\colon \mcY(\Gamma_\mcC^\theta)\to\fin(\ZZ\times\ZZ)$ by 
\begin{align*} 
\bm{g}_\mcC^\theta(w)&:= \frac{1}{2}\left( 2- \frac{\# G_{v}}{\# H_{v}} \Big(2-2\bm{g}_\mcC(v)-\bm{d}(v)\Big) -\sum_{\substack{y\in\mcY(\Gamma_\mcC), \ t(y)=v}} \frac{\# G_{v}}{\# G_y} \right), \\
\bm{e}_\mcC^\theta(w)&:=\frac{\# G_v}{(\# H_v)^2} \bm{e}_\mcC(v)-\sum_{\substack{y\in\mcY(\Gamma_\mcC), \  t(y)=v}}\frac{b(y)}{a(y)}\frac{ \# G_v}{\#G_y}, \\
\bm{m}_\mcC^\theta(z)&:=\bm{m}_\theta(y)%=\begin{pmatrix} c(y) & 0 \\ b(y) & a(y) \end{pmatrix} 
\end{align*}
for each $w\in\mcV(\Gamma_\mcC^\theta)$ and $z\in\mcY(\Gamma_\mcC^\theta)$, where $v:=\pr_\mcC^\theta(w)$ and $y:=\pr_\mcC^\theta(z)$ (see \cite[Lemma~8.2 and 9.1]{hironaka2000}). 
Since $H_v\subset G_v$ for each $v\in\mcV(\Gamma_\mcC)$, the image of $\bm{g}_\mcC^\theta$ is in $\ZZ$.  
Note that, for each $D\in\Irr(\widetilde{\mcC})$, $\bm{e}_\mcC^\theta(v_D)$ (resp. $\bm{g}_\mcC^\theta(v_D)$) is equal to the self-intersection number (resp. the genus) of the strict transformation $D'\subset \widetilde{X}'$ of $D$. 
Hence 
\[ \Gamma^\modi_\plumb(\mcC,\theta):=(\Gamma_\mcC^\theta,\bm{g}_\mcC^\theta,\bm{e}_\mcC^\theta,\bm{m}_\mcC^\theta) \] 
is a modified plumbing graph. 
By \cite[Section~5]{hironaka2000}, $\BM^\theta(\mcC)$ with the graph structure $\TT_\mcC^\theta:=\phi^{-1}(\TT_\mcC)$ is the plumbed manifold $\PM(\Gamma_\plumb^\modi(\mcC,\theta))$. 

\begin{rem}\label{rem:modified plumbing graph}
	$\Gamma_\plumb^\modi(\mcC,\theta)$ can be constructed from the algebraic data of $\tilde{\phi}\colon\widetilde{X}\to Y$ and $\widetilde{\mcC}':=(\tilde{\sigma}')^{-1}(\widetilde{\mcC})$ in the minimal resolution $\widetilde{X}'$ as follows: 
	
	The graph $\Gamma_\mcC^\theta$ is obtained from the dual graph $\widetilde{\Gamma}_\mcC'$ of $\widetilde{\mcC}'$ by contracting all vertices corresponding to the exceptional divisors of $\tilde{\sigma}'$. 
	The maps $\bm{e}_\mcC^\theta$ and $\bm{g}_\mcC^\theta$ are induced by the self-intersection numbers and the genera on $\widetilde{\Gamma}_\mcC'$ as mentioned above. 
	
	For $y\in\mcY(\Gamma_\mcC^\theta)$, $\bm{m}_{\mcC}^\theta(y)$ is described as follows. 
	For $v\in\mcV(\Gamma_\mcC^\theta)$, let $\omega_v\in\ZZ_{>0}$ be the ramification index of $\tilde{\phi}$ at the component of $\widetilde{\mcC}\subset \widetilde{X}$ corresponding to $v$. 
	We may assume by \cite{laufer1971} that the contracted part in $\widetilde{\Gamma}_\mcC'$ over $y$ is the following tree
	\[ \begin{tikzpicture}
		\coordinate (a1) at (0,0);
		\coordinate (b1) at (1,0);
		\coordinate (b2) at (2,0);
		\coordinate (b3) at (4.5,0);
		\coordinate (a2) at (5.5,0);
		
		\draw (a1) -- (2.5,0);
		\draw[dashed] (b2) -- (b3);
		\draw (a2) -- (4,0);
		\draw (-0.5,0.15) -- (a1) -- (-0.5,-0.15);
		\draw[dashed] (-1,0.3) -- (a1) -- (-1,-0.3);
		\draw (6,0.15) -- (a2) -- (6,-0.15);
		\draw[dashed] (6.5,0.3) -- (a2) -- (6.5,-0.3);
		\draw[fill] (a1) circle [radius=2pt] node[below]{$t(y)'$}; 
		\draw[fill] (b1) circle [radius=2pt] node[above]{$-k_1$}; 
		\draw[fill] (b2) circle [radius=2pt] node[above]{$-k_2$}; 
		\draw[fill] (b3) circle [radius=2pt] node[above]{$-k_s$}; 
		\draw[fill] (a2) circle [radius=2pt] node[below]{$o(y)'$}; 
		\draw[fill] (-0.8,0.12) circle [radius=0.5pt];
		\draw[fill] (-0.8,0) circle [radius=0.5pt];
		\draw[fill] (-0.8,-0.12) circle [radius=0.5pt];
		\draw[fill] (6.3,0.12) circle [radius=0.5pt];
		\draw[fill] (6.3,0) circle [radius=0.5pt];
		\draw[fill] (6.3,-0.12) circle [radius=0.5pt];
	\end{tikzpicture} \]
	for $k_i\geq 2$ for $i=1,\dots,s$, where $t(y)',o(y)'\in\mcV(\widetilde{\Gamma}_\mcC')$ are the vertices corresponding to $t(y),o(y)\in\mcV(\Gamma_\mcC^\theta)$, respectively. 
	Let $\alpha_y,\beta_y\in\ZZ$ be coprime integers satisfying $\beta_y\ne0$ and
	\[ \frac{\alpha_y}{\beta_y}=[k_1,k_2,\dots,k_s],%:=k_1-\cfrac{1}{k_2-\cfrac{1}{\ddots-\cfrac{1}{k_s}}}. 
	\]
	where $[k_1,\dots,k_s]$ is the continuous fraction defined in \cite[p.318]{neumann1981}. 
	Then we can check from the proof of \cite[Proposition~2.2]{laufer1971} that $\bm{m}_\mcC^\theta(y)$ is the matrix in (\ref{eq:matrix}) with $a(y)=\omega_{o(y)}\alpha_y$, $b(y)=\omega_{o(y)}\beta_y$ and $c(y)=\omega_{t(y)}$. 
%	the following matrix:
%	\[ \bm{m}_\mcC^\theta(y)=\begin{pmatrix} \omega_{t(y)} & 0 \\ \omega_{o(y)}\beta_y & \omega_{o(y)}\alpha_y \end{pmatrix}. \]
\end{rem}

The $G$-cover $\tilde{\phi}\colon\widetilde{X}\to Y$ induces {\red the} map $\pr_\mcC^\theta\colon \Gamma_\mcC^\theta\to\Gamma_\mcC$ defined by 
\[ \pr_\mcC^\theta(v_D):=v_{\tilde{\phi}(D)} ,\qquad \pr_\mcC^\theta(y_P):=y_{\tilde{\phi}(P)}, \qquad \pr_\mcC^\theta(\overline{y_P}):=\overline{y_{\tilde{\phi}(P)}} \]
for each $D\in\Irr(\widetilde{\mcC})$ and $P\in\Sing(\widetilde{\mcC})$, 
where $v_{\tilde{\phi}(D)}\in\mcV(\Gamma_\mcC)$ and $y_{\tilde{\phi}(P)}\in\mcY(\Gamma_\mcC)$ are the vertex and the edge corresponding to $\tilde{\phi}(D)\in\Irr(\mcC')$ and $\tilde{\phi}(P)\in\Sing(\mcC')$, respectively, 
such that 
\[ o(\pr_\mcC^\theta(y_P))=\pr_\mcC^\theta(o(y_P)), \qquad t(\pr_\mcC^\theta(y_P))=\pr_\mcC^\theta(t(y_P)). \]
Moreover, the action of $G$ on $\widetilde{X}$ induces {\red the} action $\rho_\mcC^\theta\colon G\times \Gamma_\mcC^\theta\to\Gamma_\mcC^\theta$ defined by 
\[ \rho_\mcC^\theta(g,v_D):=v_{g{\cdot}D}, \qquad \rho_\mcC^\theta(g,y_P):=y_{g{\cdot}P}, \qquad \rho_\mcC^\theta(g,\overline{y_P}):=\overline{y_{g{\cdot}P}} \]
for each $g\in G$, $D\in\Irr(\widetilde{\mcC})$ and $P\in\Sing(\widetilde{\mcC})$.

\begin{defin}\label{def:G-comb}
With the above notation, we define the \textit{$G$-combinatorial type} (or the \textit{$G$-combinatorics} for short) $\Comb_G(\mcC,\theta)$ of $\mcC$ with respect to $\theta$ as the following $6$-tuple 
\begin{align}\label{eq:G-combinatorics} 
\Comb_G(\mcC,\theta):=(\Gamma_\mcC^\theta, \bm{g}_\mcC^\theta, \bm{e}_\mcC^\theta, \bm{m}_\mcC^\theta; \rho_\mcC^\theta,\pr_\mcC^\theta). 
\end{align}
For simplicity, let $g{\cdot} \bullet$ denote $\rho_\mcC^\theta(g, \bullet)$ for each $g\in G$ and $\bullet\in\mcV(\Gamma_\mcC^\theta)\cup\mcY(\Gamma_\mcC^\theta)$. 
\end{defin}

For each $w\in\mcV(\Gamma_\mcC^\theta)$, let $p_w^\theta\colon M_w^\theta\to B_w^\theta$ denote the $S^1$-bundle of $w$, where $M_w^\theta$ is a closed subset of $\BM^\theta(\mcC)$. 
For each $z\in\mcY(\Gamma_\mcC^\theta)$, let $T_z^\theta\subset M_{t(z)}^\theta$ be the boundary component of $M_{t(z)}^\theta$ corresponding to $z$. 
By the definition of $\pr_\mcC^\theta$ and $\rho_\mcC^\theta$, we have 
\begin{align*}
	\tilde{\phi}(M_w^\theta)&=M_v\subset\BM(\mcC), & 
	\tilde{\phi}(T_z^\theta)&=T_y\subset M_{t(y)}, &
	g{\cdot}M_w^\theta &=M_{g{\cdot}w}^\theta, &
	g{\cdot}T_z^\theta &=T_{g{\cdot}z}^\theta 
\end{align*}
for each $w\in\mcV(\Gamma_\mcC^\theta)$ and $z\in\mcY(\Gamma_\mcC^\theta)$, where $v:=\pr_\mcC^\theta(w)$ and $y:=\pr_\mcC^\theta(z)$. 

For two plane curves $\mcC_i$ with surjections $\theta_i\colon \pi_1(\PP^2\setminus\mcC_i)\to G$ ($i=1,2$), we describe $\Comb_G(\mcC_i,\theta_i)$ by 
$ \Comb_G(\mcC_i,\theta_i)=(\Gamma_i^{\theta},\bm{g}_i^{\theta}, \bm{e}_i^\theta, \bm{m}_i^\theta; \rho_{i}^\theta, \pr_i^\theta)$. 

\begin{defin}\label{def:G-equiv}
For two plane curves $\mcC_i$ with surjections $\theta_i\colon \pi_1(\PP^2\setminus\mcC_i)\to G$ and an equivalent map $\varphi\colon \Comb(\mcC_1)\to\Comb(\mcC_2)$, a \textit{$G$-equivalence map $\varphi_G^\theta\colon \Comb_G(\mcC_1,\theta_1)\to\Comb_G(\mcC_2,\theta_2)$ with respect to $\varphi$} is an isomorphism $\varphi_G^\theta\colon \Gamma_{1}^{\theta}\to\Gamma_{2}^{\theta}$ of graphs such that 
 \begin{itemize}
 	\item $\bm{g}_{1}^{\theta}=\bm{g}_{2}^{\theta}\circ\varphi_G^\theta$, \ 
 	$\bm{e}_{1}^{\theta}=\bm{e}_{2}^{\theta}\circ\varphi_G^\theta$, \ 
 	$\bm{m}_{1}^{\theta}=\bm{m}_{2}^{\theta}\circ\varphi_G^\theta$, 
 	\item there is an automorphism $\tau\colon G\to G$ such that 
 	\[\varphi_G^\theta(g{\cdot}w)=\tau(g){\cdot}\varphi_G^\theta(w), \qquad \varphi_G^\theta(g{\cdot}z)=\tau(g){\cdot}\varphi_G^\theta(z)\] for each $w\in\mcV(\Gamma_{1}^{\theta})$ and $z\in\mcY(\Gamma_{1}^{\theta})$, 
 	\item $\varphi\circ\pr_{1}^{\theta}=\pr_{2}^{\theta}\circ\varphi_G^\theta$. 
 \end{itemize}
\[ \begin{tikzpicture}
	\node (a1) at (0,0) {$\Gamma_1$};
	\node (a2) at (3,0) {$\Gamma_2$};
	\node (b1) at (0,1.3) {$\Gamma_1^\theta$};
	\node (b2) at (3,1.3) {$\Gamma_2^\theta$};
	\draw[->] (a1) -- node [above] {\footnotesize $\varphi$} (a2);
	\draw[->] (b1) -- node [above] {\footnotesize $\varphi_G^\theta$} (b2);
	\draw[->] (b1) -- node [right] {\footnotesize $\pr_1^\theta$} (a1);
	\draw[->] (b2) -- node [left] {\footnotesize $\pr_2^\theta$} (a2);
	\draw[->] (b1) to [out=150, in=210, looseness=4] node [left] {\footnotesize $G$} (b1);
	\draw[->] (b2) to [out=30, in=-30, looseness=4] node [right] {\footnotesize $\tau(G)$} (b2);
\end{tikzpicture} \]
\end{defin}

\begin{thm}\label{thm:invariance_G-comb}
Let $\mcC_i\subset\PP^2$ be two plane curves with surjections $\theta_i\colon \pi_1(\PP^2\setminus\mcC_i)\to G$. 
Assume that there are a homeomorphism $h\colon \PP^2\to\PP^2$  and an automorphism $\tau\colon G\to G$ such that $h(\mcC_1)=\mcC_2$ and $\tau\circ\theta_1=\theta_2\circ h_\ast$. 
Let $\varphi\colon \Comb(\mcC_1)\to\Comb(\mcC_2)$ be the equivalence map induced by $h$. 
Then there is a $G$-equivalence map $\varphi_G^\theta\colon \Comb_G(\mcC_1,\theta_1)\to\Comb_G(\mcC_2,\theta_2)$ with respect to~$\varphi$. 
\[ \begin{tikzpicture}
	\node (a1) at (0,1.3) {$\pi_1(\PP^2\setminus\mcC_1)$};
	\node (a2) at (3.5,1.3) {$\pi_1(\PP^2\setminus\mcC_2)$};
	\node (b1) at (0,0) {$G$};
	\node (b2) at (3.5,0) {$G$};
	
	\draw [->] (a1) -- node [above] {\footnotesize $h_\ast$} (a2);
	\draw [->] (b1) -- node [above] {\footnotesize $\tau$} (b2);
	\draw [->] (a1) -- node [left] {\footnotesize $\theta_1$} (b1);
	\draw [->] (a2) -- node [right] {\footnotesize $\theta_2$} (b2);
\end{tikzpicture} \]
\end{thm}
\begin{proof}
Let $\sigma_i\colon Y_i\to\PP^2$ be the minimal good embedded resolution of $\mcC_i$, and put $\mcC'_{i}:=\sigma_i^{-1}(\mcC_i)$. 
We  identify $Y_i\setminus\mcC_i'$ with $\PP^2\setminus\mcC_i$ via $\sigma_i$. 
Let $\tilde{\phi}_i\colon \widetilde{X}_i\to Y_i$ be the $G$-covers induced by $\theta_i$ for $i=1,2$, and put $\widetilde{\mcC}_{i}:=\tilde{\phi}_i^{-1}(\mcC'_{i})$. 
Let $\iota_i'\colon \BM(\mcC_i)\to Y_i\setminus\mcC'_i$ be the inclusion map, and let $\Psi_\BM\colon\BM(\mcC_1)\to\BM(\mcC_2)$ be the homeomorphism constructed from $h$ by Theorem~\ref{thm:homeo_boundary}. 
 By (i) in Theorem~\ref{thm:homeo_boundary}, we have $\iota'_{2\ast}\circ\Psi_{\BM\ast}=h_\ast\circ\iota'_{1\ast}$. 
\[ \begin{tikzpicture}
	\node (a1) at (3.5,1.3) {$\pi_1(Y_1\setminus\mcC'_1)$};
	\node (a2) at (3.5,0) {$\pi_1(Y_2\setminus\mcC'_2)$};
	\node (b1) at (0,1.3) {$\pi_1(\BM(\mcC_1))$};
	\node (b2) at (0,0) {$\pi_1(\BM(\mcC_2))$};
	\node (c1) at (6,1.3) {$G$};
	\node (c2) at (6,0) {$G$};
	
	\draw [->] (a1) -- node [left] {\footnotesize $h_\ast$} (a2);
	\draw [->] (b1) -- node [left] {\footnotesize $\Psi_{\BM\ast}$} (b2);
	\draw [->] (b1) -- node [above] {\footnotesize $\iota'_{1\ast}$} (a1);
	\draw [->] (b2) -- node [above] {\footnotesize $\iota'_{2\ast}$} (a2);
	\draw[->] (a1) -- node [above] {\footnotesize $\theta_1$} (c1);
	\draw[->] (a2) -- node [above] {\footnotesize $\theta_2$} (c2);
	\draw[->] (c1) -- node [right] {\footnotesize $\tau$} (c2);
\end{tikzpicture} \]
Hence this commutative diagram induces a homeomorphism $\Psi^\theta_\BM\colon\BM^\theta(\mcC_1)\to\BM^\theta(\mcC_2)$ such that $\tilde{\phi}_2\circ\Psi^\theta_\BM=\Psi_\BM\circ\tilde{\phi}_1$ and $\Psi_\BM^\theta(g\cdot P)=\tau(g)\cdot\Psi^\theta_\BM(P)$ for any $P\in\BM^\theta(\mcC_1)$ and $g\in G$ (cf.\ \cite[Theorem~3.4]{shirane2019}). 
By (ii) in Theorem~\ref{thm:homeo_boundary}, we have $\Psi_\BM^\theta(\TT_1^\theta)=\TT_2^\theta$, where $\TT_i^\theta:=\tilde{\phi}_i^{-1}(\TT_i)$ is the graph structure of $\BM^\theta(\mcC_i)$ with respect to $\Gamma_\plumb^\modi(\mcC_i,\theta_i)$. 
Thus $\Psi_\BM^\theta$ induces an isomorphism $\varphi^\theta_G\colon\Gamma_1^\theta\to\Gamma_2^\theta$ such that 
\begin{itemize}
	\item $\red \Psi_\BM^\theta(M_{w_1}^\theta)=M_{w_2}^\theta$ for each $w_1\in\mcV(\Gamma_1^\theta)$ and $w_2:=\varphi_G^\theta(w_1)$; %, where $M_{i,w_i}^\theta\subset\BM^\theta(\mcC_i)$ is the total space of the $S^1$-bundle of $w_i$, 
	\item $\red \Psi_\BM^\theta(T_{z_1}^\theta)=T^\theta_{z_2}$ for each $z_1\in\mcY(\Gamma_{1}^\theta)$ and $z_2:=\varphi_G^\theta(z_1)$; %, where $T_{i,z_i}^\theta\subset\TT_{i}^\theta$ is the component corresponding to $z_i$, 
	\item $\varphi^\theta_G(g{\cdot}w_1)=\tau(g){\cdot}\varphi_G^\theta(w_1)$ and $\varphi^\theta_G(g{\cdot}z_1)=\tau(g){\cdot}\varphi_G^\theta(z_1)$ for each $g\in G$, $w_1\in\mcV(\Gamma_1^\theta)$ and $z_1\in\mcY(\Gamma_1^\theta)$, and
	\item $\varphi\circ\pr_1^\theta=\pr_2^\theta\circ\varphi_G^\theta$. 
\end{itemize}
Since the trivialization $\red \tau_{y_i}$ of $\red T_{y_i}\subset M_{t(y_i)}$ is determined by meridians $\red \gtm_{t(y_i)}$ and $\red \gtm_{o(y_i)}$ for each $y_i\in\mcY(\Gamma_i)$, it follows from (iv) in Theorem~\ref{thm:homeo_boundary} that $\bm{g}^\theta_1=\bm{g}_2^\theta\circ\varphi_G^\theta$, $\bm{e}_1^\theta=\bm{e}_2^\theta\circ\varphi_G^\theta$ and $\bm{m}_1^\theta=\bm{m}_2^\theta\circ\varphi_G^\theta$. 
Therefore, $\varphi_G^\theta$ is a $G$-equivalence map with respect to $\varphi$.
\end{proof}

\begin{rem}
With the same notation in Theorem~\ref{thm:invariance_G-comb}, 
we may assume that the automorphism $\tau\colon G\to G$ is the identity map of $G$ by replacing the surjection $\theta_2\colon \pi_1(\PP^2\setminus\mcC_2)\to G$  with $\tau^{-1}\circ\theta_2$, which does not change $\Comb_G(\mcC_2,\theta_2)$ except for the action~$\rho_{2}^\theta$. 
\end{rem}

\begin{rem}
Let $\mcC\subset\PP^2$ be a plane curve. 
Let $\theta\colon \pi_1(\PP^2\setminus\mcC)\to G$ be a surjection, and let $\phi\colon X\to\PP^2$ be the $G$-cover induced by $\theta$. 
Let $\sigma\colon Y\to\PP^2$ be the minimal good embedded resolution of $\mcC$. 
If there is a composition of (minimal) blowing-ups $\sigma''\colon Y''\to Y$ such that the $\CC(X)$-normalization $\tilde{\phi}''\colon\widetilde{X}''\to{Y''}$ is a $G$-cover with $\widetilde{X}''$ smooth, we may define a $G$-combinatorics $\Comb_G^\sm(\mcC,\theta)$ as the plumbing graph 
\[ \Comb_G^\sm(\mcC,\theta)=(\widetilde{\Gamma}_\mcC^\theta,\tilde{\bm{e}}_\mcC^\theta,\tilde{\bm{g}}_\mcC^\theta;\tilde\rho_\mcC^\theta, \tilde\pr_\mcC^\theta) \] 
together with maps $\tilde\rho_\mcC^\theta\colon G\times\widetilde\Gamma_\mcC^\theta\to\widetilde\Gamma_\mcC^\theta$ and $\tilde\pr_\mcC^\theta\colon \widetilde{\Gamma}_\mcC^\theta\to\widetilde\Gamma_\mcC$, 
where $\tilde{\rho}_\mcC^\theta$ is induced by the action $G\curvearrowright\widetilde X''$, $\widetilde\Gamma_\mcC$ is the plumbing graph together with the marking $\widetilde{\Str}_\mcC$ of the strict transformations of components of $\mcC$, which is constructed from $\Gamma_\mcC$ by blowing-ups (R$1^+$) corresponding to $\tilde\sigma''$, and $\tilde\pr_\mcC^\theta$ is the projection induced by $\tilde\phi$. 
By the same argument, we can check that Theorem~\ref{thm:invariance_G-comb} holds for the above $G$-combinatorics. 

If $G$ is a subgroup of the symmetric group of degree $3$, the above assumption is always true by \cite[\S 2]{horikawa_1973} for $\# G=2$ and \cite[Theorem~4.1]{tan2002} for $\# G=3,6$.

\end{rem}

%%%%%%%%%%%%%%%%%%%%%%%%%%%%
\section{Splitting invariants and $G$-combinatorics}\label{sec:splitting_G-comb}
%%%%%%%%%%%%%%%%%%%%%%%%%%%%

In this section, we show relation between $G$-combinatorics and splitting invariants; splitting type \cite{bannai2016}, splitting number \cite{shirane2017}, connected number \cite{shirane2018}, splitting graph \cite{shirane2019}. 
For this aim, we define a \textit{$G$-subcombinatorics} of a plane curve and its subcurve which has no common component with the branch locus. 

Let $\mcC\subset\PP^2$ be a plane curve, and let $G$ be a finite group. 
Assume that there is a surjection $\theta\colon\pi_1(\PP^2\setminus\mcC)\to G$. 
Let $\sigma\colon Y\to\PP^2$ be the minimal good embedded resolution of $\mcC$. 
Let $\phi\colon X\to \PP^2$ and $\tilde\phi\colon \widetilde X\to Y$ be $G$-covers induced by $\theta$, and 
let $\Delta_{\phi}\subset \PP^2$ and $\Delta_{\tilde\phi}\subset Y$ be the branch locus of $\phi$ and $\tilde\phi$, respectively. 

For a subcurve $\mcC_0\subset\mcC$ which has no common component with $\Delta_{\phi}$, 
let $\Gamma_\mcC(\mcC_0)$ be the subgraph of $\Gamma_\mcC$ such that 
\begin{enumerate}
	\item $\mcV(\Gamma_\mcC(\mcC_0)):=\left\{ v\in\mcV(\Gamma_\mcC) \mid \sigma({\red D_v})\subset\mcC_0, \ \sigma({\red D_v})\not\subset\Delta_{\phi} \right\}$, the set of vertices of $\Gamma_{\mcC}$ corresponding to components of $\mcC_0$ and exceptional divisors over the singularities of $\mcC_0$ outside of $\Delta_{\phi}$; 
	\item $\mcY(\Gamma_\mcC(\mcC_0)):=\left\{ y\in\mcY(\Gamma_\mcC) \mid o(y), t(y)\in\mcV(\Gamma_\mcC(\mcC_0)) \right\}$, in other words, $\Gamma_\mcC(\mcC_0)$ is a full subgraph of $\Gamma_\mcC$. 
\end{enumerate}
Put $\Str_\mcC(\mcC_0):=\{v\in\Str_\mcC\mid {\red D_v}\subset\mcC_0\}$. 
Let $\Gamma_\mcC^\theta(\mcC_0)$ be the pull-back of $\Gamma_\mcC(\mcC_0)$ by~$\pr_\mcC^\theta$, 
\[\Gamma_\mcC^\theta(\mcC_0) := (\pr_\mcC^\theta)^{-1}(\Gamma_\mcC(\mcC_0)). \]
Put $\Comb_G(\mcC,\theta)$ as in (\ref{eq:G-combinatorics}). 
The \textit{$G$-subcombinatorics} $\Comb_G(\mcC,\theta;\mcC_0)$ of $\mcC_0\subset\mcC$ is the $6$-tuple 
\[ \Comb_G(\mcC,\theta;\mcC_0):=\left(\Gamma_\mcC^\theta(\mcC_0), \bm{g}_\mcC^\theta,\bm{e}_\mcC^\theta,\bm{m}_\mcC^\theta; \rho_\mcC^\theta, \pr_\mcC^\theta\right), \]
where, by abuse of notation, the restrictions of $\bm{g}_\mcC^\theta, \bm{e}_\mcC^\theta, \bm{m}_\mcC^\theta, \rho_\mcC^\theta$ and $\pr_\mcC^\theta$ to $\Gamma_\mcC^\theta(\mcC_0)$ are denoted by the same symbols. 
A \textit{splitting curve} ${\red D}\subset\PP^2$ for a $G$-cover $\phi\colon X\to\PP^2$ is an irreducible curve ${\red D}\subset\PP^2$ such that ${\red D}\not\subset\Delta_{\phi}$ and {\red ${\phi}^{-1}(D)$ is not irreducible}. 

The known invariants can be computed by the $G$-subcombinatorics as follows: 
\begin{enumerate}
	\item Suppose that $G=\ZZ/2\ZZ$ and a subcurve $\mcC_0\subset\mcC$ consists of two splitting curves $\red D_1$ and $\red D_2$ with $\red D_1\cap D_2\cap \Delta_{\phi}=\emptyset$, say $\red {\phi}^\ast D_i=D_i^++D_i^-$. 
	The \textit{splitting type} of $\red (D_1,D_2;\Delta_{\phi})$ is the pair of intersection numbers $\red (m_1,m_2)=(D_1^+\cdot D_2^+, D_1^+\cdot D_2^-)$. 
	Then the splitting type can be computed {\red from $\Comb_G(\mcC,\theta; \mcC_0)$ by contracting vertices of $\mcV(\Gamma_\mcC^\theta(\mcC_0))\setminus{\pr_\mcC^\theta}^{-1}(\Str_\mcC(\mcC_0))$}. 
	\item Suppose that $\mcC_0\subset\mcC$ is an irreducible component such that $\mcC_0\not\subset\Delta_{\phi}$. 
	The \textit{splitting number} $s_{\phi}(\mcC_0)$ is the number of irreducible components of ${\phi}^\ast(\mcC_0)$. 
	Then $s_{\phi}(\mcC)$ is equal to the number of vertices $w\in\mcV(\Gamma_\mcC(\mcC_0))$ with $\pr_\mcC^\theta(w)\in\Str_\mcC(\mcC_0)$. 
	\item Suppose that $\mcC_0=\mcC-\Delta_{\phi}$ as reduced divisors on $\PP^2$. 
	The \textit{connected number} $c_{\phi}(\mcC_0)$ is the number of connected components of ${\phi}^{-1}(\mcC_0\setminus\Delta_{\phi})$. 
	Then $c_{\phi}(\mcC_0)$ is equal to the number of connected components of $\Gamma_{\mcC}^\theta(\mcC_0)$. 
	\item With the same assumption of (iii), the \textit{splitting graph} of $\mcC_0$ is the bipartite graph $\mcS_{\phi,\mcC_0}$ together with action of $G$ whose vertices corresponding to irreducible components of ${\phi}^{-1}(\mcC_0)$ and singular points of ${\phi}^{-1}(\mcC_0\setminus\Delta_{\phi})$, and $\mcS_{\phi,\mcC_0}$ describes configuration of components of ${\phi}^{-1}(\mcC_0\setminus\Delta_{\phi})$ (see \cite{shirane2019} for details). 
	Put $\Str_\mcC^\theta:=(\pr_\mcC^\theta)^{-1}(\Str_\mcC)$. 
	To construct the splitting graph from $\Gamma_\mcC^\theta(\mcC_0)$, let $\Gamma_0$ be the full subgraph of $\Gamma_\mcC^\theta(\mcC_0)$ whose vertex set is $\mcV(\Gamma_\mcC^\theta(\mcC_0))\setminus\Str^\theta_\mcC$. 
	Let $\mcY^\theta_{\mcC_0}$ be the set of edges $z\in\mcY(\Gamma_\mcC^\theta(\mcC_0))$ such that $o(z),t(z)\in\Str_\mcC^\theta$. 
	Note that the disjoint union $\Conn(\Gamma_0)\sqcup \mcY^\theta_{\mcC_0}$ corresponds to the set of singularities of ${\phi}^{-1}(\mcC_0)$ outside ${\phi}^{-1}(\Delta_{\phi})$, where $\Conn(\Gamma_0)$ is the set of connected components of $\Gamma_0$.  
	Then $\mcS:=\mcS_{\phi,\mcC_0}$ is obtained from $\Gamma_\mcC^\theta(\mcC_0)$ as follows:
	\begin{enumerate}
	\item The vertex set $\mcV(\mcS)$ consists of the following two parts $\mcV_1(\mcS)$ and $\mcV_2(\mcS)$;
	\begin{align*}
	\mcV_1(\mcS)&:=\mcV(\Gamma_\mcC^\theta(\mcC_0))\cap\Str^\theta_\mcC, &
	\mcV_2(\mcS)&:=\{w_k\mid k\in  \Conn(\Gamma_0)\sqcup\mcY_{\mcC_0}^\theta\};
	\end{align*}
	\item for $v\in\mcV_1(\mcS)$ and $w_k\in\mcV_2(\mcS)$ with $k\in\Conn(\Gamma_0)\sqcup\mcY_{\mcC_0}^\theta$, there is an edge $y_{v,w_k}$ with $o(y_{v,w_k})=v$ and $t(y_{v,w_k})=w_k$ if and only if there is an edge $y\in\mcY(\Gamma_\mcC^\theta(\mcC_0))$ satisfying one of the following conditions 
	\begin{enumerate}
	\item $k\in\mcY_{\mcC_0}^\theta$, $t(y)=v$ and $y=k$, 
	\item $k\in\Conn(\mcC_0)$, $t(y)=v$ and $o(y)\in k$;
	\end{enumerate}
	\item the action of $G$ on $\mcS$ is induced from the action $\rho_\mcC^\theta$ of $G$ on $\Gamma_\mcC^\theta(\mcC_0)$. 
\end{enumerate}

\end{enumerate}

In particular, if two plane curves $\mcC_i\subset\PP^2$ ($i=1,2$) with surjections $\theta_i\colon \pi_1(\PP^2\setminus\mcC_i)\to G$ have different splitting invariants (splitting types, splitting numbers, connected numbers or splitting graphs), then they have different $G$-combinatorics with respect to $\theta_i$. 

\begin{rem}
For each $i=1,2$, let $\mcB_i\subset\PP^2$ be a plane curve of even degree, let $\red D_{i1}, D_{i2}\subset\PP^2$ be distinct two irreducible curves with $\red D_{ij}\not\subset\mcB_i$ for $j=1,2$ and $\red D_{i1}\cap D_{i2}\cap\mcB_i=\emptyset$, and 
let $\phi_i\colon X_i\to\PP^2$ be the double cover branched along $\mcB_i$. 
In \cite{bannai2016}, it was proved under the assumption $C_{i1}\pitchfork C_{i2}$ for $i=1,2$ that, if $(C_{11},C_{12};\mcB_1)$ and $(C_{21},C_{22};\mcB_2)$ have distinct splitting types, then there is no homeomorphism $h\colon \PP^2\to\PP^2$ such that $h(C_{11}+C_{12})=C_{21}+C_{22}$ and $h(\mcB_1)=\mcB_2$, where $C_{i1}\pitchfork C_{i2}$ means that $C_{i1}$ and $C_{i2}$ intersect transversally. 
By \cite[III, Theorem~21]{brieskorn_knorrer_1986}, this statement holds even if we omit the assumption $C_{i1}\pitchfork C_{i2}$. 
\end{rem}

\section{Quasi-triangular curves}\label{sec:qt_curve}

In this section, we consider a generalization of triangular curves studied in \cite{artal_cogolludo_martin_2019}, which is related with Artal arrangements studied in \cite{shirane2019}. 
\begin{defin}
Let $d>1$ be an integer, and let $\red D\subset\PP^2$ be a plane curve of degree $2d$, and let $L_1,L_2,L_3\subset\PP^2$ be three lines. 
\begin{enumerate}
	\item For a partition $\parti:=(e_1,\dots,e_n)$ of $d$, we call a singularity $P\in {\red D}$ a \textit{multi-cusp of type $\parti$} if the local topological type of $\red D$ at $P$ is given by 
	\[ \prod_{i=1}^n\left((x+iy)^{e_i}+y^{e_i+1}\right)=0. \] 
	\item For three partitions $\parti_1,\parti_2,\parti_3$ of $d$, we call ${\red D}+L_1+L_2+L_3$ a \textit{quasi-triangular curve of type $(\parti_1,\parti_2,\parti_3)$} if $\red \Sing(D)$ consists of three multi-cusps $P_i$ of type $\parti_i$ ($i=1,2,3$) and $L_i$ passes through $P_j$ and $P_k$ for $\{i,j,k\}=\{1,2,3\}$. 
\end{enumerate}
\end{defin}

\begin{rem}\label{rem:qt-curve}
Let $\parti_1,\parti_2,\parti_3$ be three partitions of $d>1$, and let $\mcC={\red D}+L_1+L_2+L_3$ be a quasi-triangular curve of type $(\parti_1,\parti_2,\parti_3)$. 
\begin{enumerate}
	\item Since ${\red D}\cdot L_i=2d$, ${\red D}\cap L_i$ consists of two singular points $P_j, P_k$ of $\red D$, and $L_i$ is not tangent to $\red D$ at $P_j$ and $P_k$. 
	\item We obtain an Artal arrangement $\overline{\red D}+L_1'+L_2'+L_3'$ of type $(\parti_1,\parti_2,\parti_3)$ after a quadratic transformation, i.e., blowing-up at $P_1,P_2,P_3$ and blowing-down of the strict transformations of $L_1,L_2,L_3$. Since $\overline{\red D}$ is smooth and birational to $\red D$, $\red D$ is irreducible. 
	In particular, the combinatorics of a quasi-triangular curve is determined by its type $(\parti_1,\parti_2,\parti_3)$. 
\end{enumerate}
\end{rem}

We fix three partitions $\parti_i=(e_{i,1},\dots,e_{i,n_i})$ of $d>1$ for $i=1,2,3$, and put $\Parti:=(\parti_1,\parti_2,\parti_3)$. 
Let $s_i$ be the greatest common divisor $\gcd(e_{i,1},\dots,e_{i,n_i})$ for $i=1,2,3$, and put $s:=\gcd(s_1,s_2,s_3)$. 
Let $\widetilde{\Sigma}_\Parti$ be the realization space of quasi-triangular curves of type $\Parti$.

\begin{lem}\label{lem:defining eq.}
Let $\mcC:={\red D}+\mcL$ be a quasi-triangular curve of type $\Parti=(\parti_1,\parti_2,\parti_3)$, 
where $\mcL:=L_1+L_2+L_3$. 
Put $\mu_{i,j}:=e_{i,j}/s_i$. 
Then, after a certain projective transform of $\PP^2$, $\mcL$ defined by $xyz=0$, and 
$\red D$ is defined by $\overline{F}_\Parti(\beta,\{c_{i,j}\},g_0)=0$ with 
{\footnotesize
\begin{align*}
\overline{F}_{\Parti}(\beta,\{c_{i,j}\},g_0) :=&x^d\prod_{j=1}^{n_1}(z+c_{1,j}y)^{e_{1,j}}+y^d\prod_{j=1}^{n_2}(x+c_{2,j}z)^{e_{2,j}}+z^d\prod_{j=1}^{n_3}(y+c_{3,j}x)^{e_{3,j}} \\ & \quad-y^dz^d-z^dx^d-x^dy^d+x^2y^2z^2g_0(yz,zx,xy), \notag
\end{align*}
}
where $\beta$ is an integer with $0\leq\beta<s$, $g_0$ is a homogeneous polynomial of degree $d-3$ in $x,y,z$, and
$c_{i,j}$ are complex numbers satisfying the following conditions; 
\begin{align*}\label{eq: condition of c} 
\prod_{j=1}^{n_i}c_{i,j}^{\mu_{i,j}}=1 \ \ (i=1,2), & &
\prod_{j=1}^{n_3}c_{3,j}^{\mu_{3,j}}=\zeta_d^{\beta\mu_3}, &&
c_{i,j}\ne c_{i,j'} \ \mbox{ if $j\ne j'$}, 
\end{align*}
where $\mu_i:=d/s_i=\sum_{j=1}^{n_i}\mu_{i,j}$ for $i=1,2,3$. 
\end{lem}
\begin{proof}
We may assume that $\mcL$ is defined by $xyz=0$. 
Then the quadratic transformation given by $(x:y:z)\mapsto(yz:zx:xy)$ maps $\red D$ to a smooth curve $\overline{\red D}$ such that $\overline{\red D}+\mcL$ is an Artal arrangement of type $\Parti$. 
Hence the assertion follows from \cite[Lemma~4.4]{shirane2019}.  
\end{proof}

For an integer $\beta$ with $0\leq\beta<s$, let $\widetilde\Sigma_\Parti^\beta\subset\widetilde\Sigma_\Parti$ be the subset consisting of quasi-triangular curves projectively equivalent to a curve defined by an equation of form $xyz\overline{F}_\Parti(\beta,\{c_{i,j}\},g_0)=0$. 
By Lemma~\ref{lem:defining eq.}, we have $\widetilde{\Sigma}_\Parti=\bigcup_{\beta=0}^{s-1}\widetilde{\Sigma}_\Parti^\beta$. 
It follows from \cite[Lemma~4.8]{shirane2019} that $\widetilde{\Sigma}_\Parti^\beta$ is connected. 
For $0\leq \beta<\beta'<s$, if $\beta\ne s-\beta'$, then $\widetilde{\Sigma}_\Parti^\beta\cap\widetilde{\Sigma}_\Parti^{\beta'}=\emptyset$ by \cite[Theorem~4.3 and Lemma~4.5]{shirane2019}. 
Furthermore, we infer from \cite[Theorem~4.3]{shirane2019} and its proof that $\widetilde{\Sigma}_\Parti^\beta\cap\widetilde{\Sigma}_\Parti^{\beta'}=\emptyset$ for $0\leq\beta<\beta'<s$ if $\parti_i\ne\parti_j$ for any $i\ne j$; and $\widetilde{\Sigma}_\Parti^\beta=\widetilde{\Sigma}_\Parti^{s-\beta}=\emptyset$ for $0\leq\beta\leq s/2$ otherwise. 
By using $G$-combinatorics, we can prove the following theorem. 

\begin{thm}\label{thm:quasi-triangular}
Let $0\leq\beta_1<\beta_2<s$, and let $\mcC_i:={\red D_i}+\mcL_i$ be a member of $\widetilde\Sigma_\Parti^{\beta_i}$ for $i=1,2$.  
Then $(\mcC_1,\mcC_2)$ is a Zariski pair if and only if $\beta_1\ne s-\beta_2$. 
\end{thm}
\begin{proof}
By \cite[Theorem~4.3 (iii)]{shirane2019}, if $\beta_1= s-\beta_2$, then $\mcC_1$ and $\mcC_2$ have the same embedded topology. 
Suppose that $\beta_1\ne s-\beta_2$. 
Let $\sigma_i\colon Y_i\to\PP^2$ be the minimal good embedded resolution of $\mcC_i$. 
Note that, if the strict transform of $\mcL_i$ is contracted, then we obtain the minimal good embedded resolution of the Artal arrangement $\mcA_i$ given by the quadratic transformation of $\mcC_i$.  
Let $\theta_i\colon\pi_1(\PP^2\setminus\mcC_i)\to\ZZ/d\ZZ$ be the surjection given by the composition
\[ \pi_1(\PP^2\setminus\mcC_i)\overset{q_1}{\longrightarrow} \pi_1(\PP^2\setminus {\red D_i})\overset{q_2}{\longrightarrow} H_1(\PP^2\setminus {\red D_i};\ZZ)\cong\ZZ/2d\ZZ\overset{q_3}{\longrightarrow}\ZZ/d\ZZ, \]
where $q_1$ is induced the inclusion $\PP^2\setminus\mcC_i\to\PP^2\setminus {\red D_i}$, $q_2$ is the quotient map by the commutator subgroup, and $q_3$ is the natural surjection. 
Let $\phi_i\colon X_i\to\PP^2$ be the cyclic cover given by $\theta_i$, which is branched along $\red D_i$. 
If there is an equivalence map $\varphi_G^\theta\colon\Comb_G(\mcC_1,\theta_1)\to\Comb(\mcC_2,\theta_2)$, then $\varphi_G^\theta$ induces an equivalence map $\bar{\varphi}_G^\theta\colon\Comb_G(\mcA_1,\theta_1)\to\Comb_G(\mcA_2,\theta_2)$, which is a contradiction to \cite[Theorem~4.3]{shirane2019}. 
Therefore, $\Comb_G(\mcC_i,\theta_i)$ are not $G$-equivalent, and there is no homeomorphism $h\colon\PP^2\to\PP^2$ such that $\red h(D_1)=D_2$ and $h(\mcL_1)=\mcL_2$ by Theorem~\ref{thm:invariance_G-comb} since $\theta_i$ satisfies $\theta_i({\red \gtm_{D_i}})=\pm1$ for a meridian $\red \gtm_{D_i}$ of $\red D_i$. 
This implies that $(\mcC_1,\mcC_2)$ is a Zariski pair since $\deg {\red D_i}=2d>2$ and $\red D_i$ is irreducible. 
\end{proof}

\begin{rem}
For $\mcC_i$ and $\phi_i$ in the proof above, the splitting graphs of $\mcC_i$ for $\phi_i$ are $G$-equivalent since $\phi_i^{-1}(\mcL_i\setminus {\red D_i})$ are disjoint union of $3d$ components. 
Namely, it is impossible to prove the difference of their embedded topology by the splitting graphs of $\mcL_i$ for $\phi_i$. 
\end{rem}

\begin{cor}
There is a Zariski $(\lfloor s/2\rfloor+1)$-tuple $(\mcC_0,\dots,\mcC_{\lfloor s/2\rfloor})$ of quasi-triangular curves $\mcC_i$ of type $\Parti$. 
\end{cor}
\begin{proof}
Since $\widetilde\Sigma_{\Parti}^\beta\ne\emptyset$ for each $0\leq \beta<\lfloor s/2\rfloor$ by \cite[Lemma~4.5]{shirane2019}, the assertion follows from Theorem~\ref{thm:quasi-triangular}. 
\end{proof}

\appendix
%%%%%%%%%%%%%%%%%%%%%%%%%%%%%%%%%%%%%%
\section{$S^1$-bundles over $2$-manifolds}\label{sec:S1-bdls}
%%%%%%%%%%%%%%%%%%%%%%%%%%%%%%%%%%%%%%
In this appendix, we recall results on $S^1$-bundles over $2$-manifolds proved in \cite{waldhausen_1967}. 
{\red According to Waldhausen \cite{waldhausen_1967}, a \textit{system} of manifolds means a pairwise disjoint finite union of manifolds.}

\begin{lem}[{\cite[Theorem~2.8]{waldhausen_1967}}]\label{lem:S1-bdl_vertical}
Let $p\colon M\to B$ be an $S^1$-bundle over a $2$-manifold $B$ such that $B$ is neither the $2$-sphere $S^2$ nor the real projective plane $\PP^2_\RR$. 
Let $F$ be a {\red system} of incompressible orientable $2$-manifolds in $M$ such that no component of $F$ with boundary is boundary-parallel. 
Then there is an isotopy $\gth_u\colon M\to M$ $(0\leq u\leq 1)$ such that $\gth_0=\id_M$ and $\gth_1$ satisfies either
\begin{enumerate}
	\item $\gth_1(F)\subset M$ is vertical for $p$, i.e., $\gth_1(F)=p^{-1}(\ell)$ for some $\ell\subset B$. In particular, $F$ is an annulus or a torus; or
	\item the restriction of $p$ to $\gth_1(F)$ is a {\red finite covering}. 
\end{enumerate}
\end{lem}

\begin{rem}\label{rem:S1-bdl_vertical}
With the same assumption of Lemma~\ref{lem:S1-bdl_vertical}, if $F\cap \partial M$ is vertical, then there is an isotopy $\gth_u\colon M\to M$ ($0\leq u\leq 1$) such that $\gth_0=\id_M$, $\gth_1(F)\subset M$ is vertical for $p$, and $\gth_u|_{\partial M}=\id_{\partial M}$ for $u\in I$, i.e., the isotopy $\gth_u$ is constant on the boundary $\partial M$. 

By tracing the proof of \cite[Theorem~2.8]{waldhausen_1967}, we can check the above statement. 
Indeed, the sketch of the proof of \cite[Theorem~2.8]{waldhausen_1967} with the above assumption is as follows (see \cite{waldhausen_1967} for details): 
Take a {\red system} $C$ of vertical annuli such that $M\setminus U(C)^\circ$ is an $S^1$-bundle over a disc for a small vertical regular neighborhood $U(C)$ of $C\subset M$. 
We deform $M$ by an isotopy of $M$ constant on $\partial M$ so that the number of circles in $F\cap C$ is minimal.  
There is no $1$-manifold with boundary in $F\cap C$ since $F\cap \partial M$ is vertical. 
Furthermore, it is proved in \cite{waldhausen_1967} that there is no $1$-manifold in $F\cap C$ which bounds a disc on $F$ or $C$. 
Hence $F\cap C$ consists of only closed $1$-manifolds which are parallel to the boundary of $C$. 
There is a deformation of $M$ constant on $\partial M$ after which $F\cap C$ consists of vertical circles. 
Let $M'$ be the $3$-manifold obtained from $M$ by cutting out along $C$, i.e., $M'=M\setminus U(C)^\circ$. 
Put $F':=F\cap M'$. 
Then $M'$ is a solid torus, and $F'$ consists of boundary-parallel annuli by \cite[Theorem~1.9 and Lemma~2.3]{waldhausen_1967}. 
The assertion follows from \cite[Lemma~2.4]{waldhausen_1967}.  
\end{rem}

\begin{lem}[{\cite[Lemma~5.1]{waldhausen_1967}}]\label{lem:S1-bdl_annulus}
Put $A:=I\times S^1$, the annulus. 
Let $p_i\colon A\to I$ $(i=1,2)$ be two $S^1$-bundle projections such that $p_1|_{\partial A}=p_2|_{\partial A}$. 
Then there is an isotopy $\gth_u\colon A\to A$ $(0\leq u\leq 1)$ such that $\gth_0=\id_A$, $p_1=p_2\circ\gth_1$ and $\gth_u|_{\partial A}=\id_{\partial A}$ for any $u\in I$. 
\end{lem}

\begin{lem}[{\cite[Lemma~5.2]{waldhausen_1967}}]\label{lem:S1-bdl_torus}
Let $T$ be a torus, and let $p_i\colon T\to S^1$ $(i=1,2)$ be two $S^1$-bundles. 
If a fiber $F_1$ of $p_1$ is homotopic to a fiber $F_2$ of $p_2$, then there is an isotopy $\gth_u\colon T\to T$ $(0\leq u\leq1)$ such that $\gth_0=\id_T$ and $p_1=p_2\circ\gth_1$. 
\end{lem}

\begin{lem}[{\cite[Lemma~5.3]{waldhausen_1967}}]\label{lem:S1-bdl_solidtorus}
Put $V:={\red\DD^2}\times S^1$, the solid torus. 
Let $p_i\colon V\to {\red \DD^2}$ $(i=1,2)$ be two $S^1$-bundles. 
If $p_1|_{\partial V}=p_2|_{\partial V}$, then there is an isotopy $\gth_u\colon V\to V$ $(0\leq u\leq1)$ such that $\gth_0=\id_V$, $p_1=p_2\circ\gth_1$ and $\gth_u|_{\partial V}=\id_{\partial V}$.  
\end{lem}

\begin{thm}[{\cite[Theorem~5.5]{waldhausen_1967}}]\label{thm:S1-bdl_gene}
Let $p_i\colon M_i\to B_i$ $(i=1,2)$ be two $S^1$-bundles over $2$-manifolds $B_i$. 
Assume that $B_i$ are not homeomorphic to $S^2$, $\PP^2_\RR$, $\red \DD^2$ or $I\times S^1\times S^1$, and that $p_i$ has no section if $B_i$ is homeomorphic to either the torus or the Klein's bottle. 
If $h\colon M_1\to M_2$ is a homeomorphism, then there is an isotopy $\gth_u\colon M_1\to M_2$ $(0\leq u\leq 1)$ such that 
\begin{enumerate}
	\item $\gth_0=h$, 
	\item there is a homeomorphism $h'\colon B_1\to B_2$ such that $p_2\circ\gth_1=h'\circ p_1$. 
\end{enumerate}
\end{thm}

\begin{rem}\label{rem:S1-bdl_gene}
With the same assumption of Theorem~\ref{thm:S1-bdl_gene}, if $\partial M_1\ne \emptyset$ and a homeomorphism $h\colon M_1\to M_2$ maps each fiber in $\partial M_1$ of $p_1$ to that of $p_2$, then there is an isotopy $\gth_u\colon M_1\to M_2$ satisfying (i), (ii) in Theorem~\ref{thm:S1-bdl_gene} and 
\begin{enumerate}
	\item[(iii)] $\gth_u|_{\partial M_1}=h|_{\partial M_1}$.  
\end{enumerate}
Indeed, it can be {\red seen} by tracing the proof of \cite[Theorem~5.5]{waldhausen_1967}. 
The sketch is as follows (see \cite{waldhausen_1967} for details): 
Let $C_1\subset M_1$ be a {\red system} of vertical annuli for $p_1$ such that $M_1\setminus U(C_1)^\circ$ for a small vertical regular neighborhood $U(C_1)$ of $C_1$ in $M_1$. 
Then no component of $C_1$ is compressible or boundary-parallel. 
By Remark~\ref{rem:S1-bdl_vertical}, there is an isotopy $\gth_u'\colon M_1\to M_2$ ($0\leq u\leq 1$) such that $\gth_0'=h$, $C_2:=\gth_1'(C_1)$ is vertical for $p_2$ and $\gth_u'|_{\partial M_1}=h|_{\partial M_1}$. 
By Lemma~\ref{lem:S1-bdl_annulus}, there is an isotopy $\gth_u'\colon M_1\to M_2$ ($1\leq u\leq 2$) such that $\gth_u'|_{\partial M_1}=h|_{\partial M_1}$, $\gth'_u(C_1)=C_2$ for $1\leq u\leq 2$, and $\gth'_2$ maps each fiber in $C_1$ of $p_1$ to that in $C_2$ of $p_2$. 
By applying Lemma~\ref{lem:S1-bdl_solidtorus} to the restriction of $\gth'_2$ to $M_1\setminus U(C_1)^\circ$, we have an isotopy $\gth_u'\colon M_1\to M_2$ ($2\leq u\leq 3$) such that $\gth_u'|_{\partial M_1}=h_{\partial M_1}$ for $2\leq u\leq 3$, and $\gth'_3$ maps every fiber of $p_1$ to that of $p_2$. 
Let $\gth_u:=\gth_{3u}$ ($0\leq u\leq 1$), and let $h'\colon B_1\to B_2$ be the homeomorphism induced by $\phi_1$. 
The isotopy $\gth_u$ and the homeomorphism $h'\colon B_1\to B_2$ satisfy the assertion.  
\end{rem}

By the same argument of Remark~\ref{rem:S1-bdl_gene}, we obtain the following lemma. 

\begin{lem}\label{lem:S1-bdl_over_annulus}
Put $M:=I\times S^1\times S^1$ and $A:=I\times S^1$. 
Let $p_i\colon M\to A$ $(i=1,2)$ be two $S^1$-bundles. 
If a homeomorphism $h\colon M\to M$ maps each fiber in $\partial M$ of $p_1$ to that of $p_2$, then there is an isotopy $\gth_u\colon M\to M$ $(0\leq u\leq 1)$ such that $\gth_0=h$, $\gth_u|_{\partial M}=h|_{\partial M}$ for $0\leq u\leq 1$, and $\gth_1$ maps every fiber of $p_1$ to that of $p_2$. 
\end{lem}

Any $S^1$-bundle $p\colon M\to B$ over a $2$-manifold $B$ determines uniquely a disc bundle $\overline{p}\colon \overline{M}\to {\red B}$ up to isotopy such that $\partial \overline{M}=M$ and $p=\overline{p}|_{M}$ (cf. \cite{HNK1971}).

Let $p_i\colon M_i\to B_i$ ($i=1,2$) be two $S^1$-bundles, and let $\overline{p}_i\colon \overline{M}_i\to B_i$ be the induced disc bundles. 
Let $\partial B_{i,1},\dots,\partial B_{i,d_i}$ be the boundary components of $B_i$, and put $\overline{\partial M}_{i,j}:=\overline{p}_i^{-1}(\partial B_{i,j})$. 
Let 
\[\overline{\tau}_{i, j}\colon S^1\times {\red \DD^2}\longrightarrow \overline{\partial M}_{i, j}\]
be the trivialization of $\overline{\partial M}_{i, j}\subset\overline{M}_i$ with $\overline{p}_i\circ\overline{\tau}_{i,j}(t_1,t_2)=t_1$ for each $i=1,2$. 
We say the following operation to \textit{plumb} $\overline{M}_1$ and $\overline{M}_2$ at $\overline{\partial M}_{1,j_1}$ and $\overline{\partial M}_{2,j_2}$: 

We take the disjoint union $\overline{M}_1\sqcup ({\red \DD^2\times \DD^2})\sqcup\overline{M}_2$, and 
identify by the homeomorphism 
\begin{align*}
	&\overline{\partial M}_{1,j_1}\ni \overline{\tau}_{1,j_1}(t_1,t_2) \longmapsto (t_1,t_2)\in {\red (\partial \DD^2)\times \DD^2}, 
	\\
	&\overline{\partial M}_{2,j_2}\ni \overline{\tau}_{2,j_2}(t_1,t_2)\longmapsto (t_2,t_1)\in {\red \DD^2\times (\partial \DD^2)}
\end{align*}
for each $(t_1,t_2)\in S^1\times {\red\DD^2}$. 
Then we obtain an orientable manifold with boundary, and this manifold is differentiable except along $S^1\times S^1\subset {\red \DD^2\times \DD^2}$, where it has a ``corner". 
By Milnor \cite{milnor1959}, this corner can be smoothed, and the smoothing is essentially unique (cf. \cite{HNK1971}). 
Note that $M_i=\partial \overline{M}_i$ ($i=1,2$) are glued by \[ J:=\begin{pmatrix} 0&1\\1&0\end{pmatrix}\colon\partial (\overline{\partial M}_{1,j_1})\cong S^1\times S^1\longrightarrow S^1\times S^1\cong\partial(\overline{\partial M}_{2,j_2})\] via the trivializations $\overline{\tau}_{i,j_i}$.

%%%%%%%%%%%%%%%%%%%%%%%%%%%%%%%%%%%%%%
\section{Plumbing graphs}\label{sec:pgm} 
%%%%%%%%%%%%%%%%%%%%%%%%%%%%%%%%%%%%%%
In this appendix, we recall results on plumbing graphs in \cite{neumann1981}. 

%%%%%%%%%%%%%%%%%%%%
\subsection{Plumbing graphs and W-graphs}\label{sec:pg vs Wg}
%%%%%%%%%%%%%%%%%%%%

Let $\Gamma_\dplumb=(\Gamma,\bm{g},\bm{e},\sign; \Ver, \AH)$ be a decorated plumbing graph of F-normal form with $\AH(\Gamma)\ne\emptyset$. 
We recall \cite[Theorem~5.1]{neumann1981} which shows how to obtain a reduced graph structure of the plumbed manifold $\PM(\Gamma_\dplumb)$ from the F-normal form $\Gamma_\dplumb$. 
We obtain the following theorem from \cite[Theorem~5.1]{neumann1981}. 

\begin{thm}[{cf. \cite[Theorem~5.1]{neumann1981}}]\label{thm:red_str}
	Let $\Gamma_\dplumb=(\Gamma,\bm{g},\bm{e},\sign;\Ver,\AH)$ be a decorated plumbing graph of F-normal form with $\AH(\Gamma)\ne\emptyset$. 
	Let $\mcS$ be a subset of $\mcY(\Gamma)$ which contains precisely one edge of each maximal chain $\gtc$ of $\Gamma$ such that $\gtc$ is not incident to a boundary vertex. 
	Then $\TT_\mcS:=\bigcup_{y\in \mcS}T_y$ is a reduced graph structure of $\PM(\Gamma_\dplumb)$.  
\end{thm}
\begin{proof}
As \cite[Appendix]{neumann1981}, we regard each vertex $v$ in $\AH(\Gamma)$ as a vertex with $\bm{g}(v)=0$, $\bm{e}(v)=0$ and $r_v=1$ and $\Gamma_\dplumb$ as an original plumbing graph. 
Then each chain $\mathfrak{c}$ incident to a boundary vertex can be contract by annulus absorptions (R8) in \cite[Proposition~2.1]{neumann1981}. 
Hence $\TT_\mcS$ is a reduced graph structure of $\PM(\Gamma_\dplumb)$ by the definition of F-normal form and \cite[Theorem~5.1]{neumann1981} except the following case:
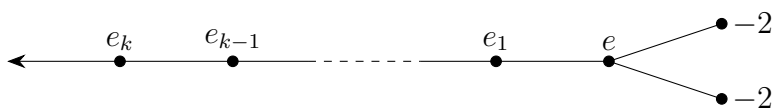
\[
\begin{tikzpicture}
	\coordinate (a1) at (0,0);
	\coordinate (a2) at (1.5,0);
	\coordinate (a3) at (3,0);
	\coordinate (a31) at (4,0);
	\coordinate (a32) at (5.5,0);
	\coordinate (a4) at (6.5,0);
	\coordinate (a5) at (8,0);
	\coordinate (a61) at (9.5,0.3);
	\coordinate (a62) at (9.5,-0.3);
	
	\draw[fill] (a2) circle [radius=2pt] node [above] {$e_k$};
	\draw[fill] (a3) circle [radius=2pt] node [above] {$e_{k-1}$};
	\draw[fill] (a4) circle [radius=2pt] node [above] {$e_1$};
	\draw[fill] (a5) circle [radius=2pt] node [above] {$e$};
	\draw[fill] (a61) circle [radius=2pt] node [right] {$-2$};
	\draw[fill] (a62) circle [radius=2pt] node [right] {$-2$};
	
	\draw[-{Stealth[length=7pt]}] (a31) to (a1);
	\draw[dashed] (a31) to (a32);
	\draw (a32) to (a5);
	\draw (a61) to (a5) to (a62);
\end{tikzpicture}
\]
If $\Gamma_\dplumb$ is the above decorated plumbing graph, then $\TT_\mcS$ consists of two tori, and the graph manifold $\PM(\Gamma_{\dplumb})$ with graph structure $\TT_\mcS$ is the graph manifold $Q$ in \cite[Section~3]{waldhausen_1967}, which is reduced. 
\end{proof}

%%%%%%%%%%%%%
%%%%%%%%%%%%%%%%%%%
\subsection{Orientation reversal}
%%%%%%%%%%%%%%%%%%%
Here we recall the results in \cite[Section~7]{neumann1981}. 
For a graph manifold $M$ with orientation, let $-M$ denote 
the graph manifold $M$ with reversed orientation. 
For a plumbing graph $\Gamma_\plumb$ of normal form, the normal form plumbing graph for $-\PM(\Gamma_\plumb)$ is described in \cite[Theorem~7.1]{neumann1981}. 
We describe the F-normal form decorated plumbing graph of $-\PM(\Gamma_\dplumb)$ from an F-normal form decorated plumbing graph $\Gamma_\dplumb$ based on \cite[Theorem~7.1]{neumann1981}. 

Let $\Gamma_\dplumb=(\Gamma,\bm{g},\bm{e};\Ver, \AH)$ be a decorated plumbing graph of F-normal form with $\sign(y)=+1$ for each $y\in\mcY(\Gamma)$. 
Let $\bm{c}\colon \Ver(\Gamma)\to\ZZ_{\geq 0}$ be the map such that $\bm{c}(v)$ is the number of maximal chains adjoining $v$ with length $>0$, where chains which both end edges are adjoining $v$ are counted twice, for each $v\in\Ver(\Gamma)$. 
We abbreviate a sequence of $n$ $2$'s ($n\geq0$) as $n\cdot 2$ in the following theorem; 
\[ n{\cdot} 2:=\underbrace{2,\dots,2}_n. \]
The following theorem can be proved by the argument of \cite[Theorem~7.1]{neumann1981}. 

\begin{thm}[{cf. \cite[Theorem~7.1]{neumann1981}}]\label{thm:ori_rev}
	With the above notation, the F-normal form decorated plumbing graph $\Gamma_\dplumb'=(\Gamma',\bm{g}',\bm{e}',\sign'; \Ver, \AH)$ for $-\PM(\Gamma_\dplumb)$ is obtained from $\Gamma_\dplumb$ as follows;  
	\begin{enumerate}
	\item $\AH(\Gamma')=\AH(\Gamma)$, 
	\item if $v\in\Ver(\Gamma)$ which is not on a chain, then $v\in\Ver(\Gamma')$ with $\bm{g}'(v)=\bm{g}(v)$ and $\bm{e}'(v):=-\bm{e}(v)-\bm{c}(v)$; 
	\item replace each maximal chain of length $k\geq 1$ on the left below by the corresponding chain on the right; 
	\[ \small
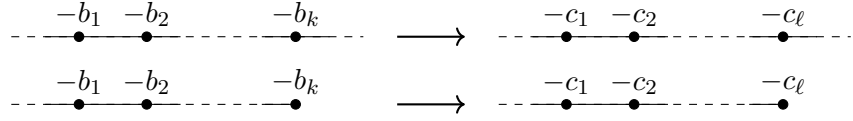
	\begin{tikzpicture}[scale=0.9]
		\coordinate (a10) at (0,0);
		\coordinate (a11) at (0.5,0);
		\coordinate (a1) at (1,0);
		\coordinate (a2) at (2,0);
		\coordinate (a20) at (2.5,0);
		\coordinate (a21) at (3.2,0);
		\coordinate (a22) at (3.7,0);
		\coordinate (a3) at (4.2,0);
		\coordinate (a30) at (4.7,0);
		\coordinate (a31) at (5.2,0);
		\coordinate (a4) at (5.7,0);
		\coordinate (a5) at (6.7,0);
		\coordinate (d1) at (0,-1);
		\coordinate (d2) at (7.2,-1);
		
		\draw[dashed] (a10) to (a31);
		\draw (a11) to (a20);
		\draw (a22) to (a30); 
		
		\draw[fill] (a1) circle [radius=2pt] node [above] {$-b_1$};
		\draw[fill] (a2) circle [radius=2pt] node [above] {$-b_2$};
		\draw[fill] (a3) circle [radius=2pt] node [above] {$-b_k$};
		
		\draw[thick, ->] (a4) to (a5);
		
		\draw[dashed] ($(a10)+(a5)+(0.5,0)$) to ($(a31)+(a5)+(0.5,0)$);
		\draw ($(a11)+(a5)+(0.5,0)$) to ($(a20)+(a5)+(0.5,0)$);
		\draw ($(a22)+(a5)+(0.5,0)$) to ($(a30)+(a5)+(0.5,0)$); 
		
		\draw[fill] ($(a1)+(a5)+(0.5,0)$) circle [radius=2pt] node [above] {$-c_1$};
		\draw[fill] ($(a2)+(a5)+(0.5,0)$) circle [radius=2pt] node [above] {$-c_2$};
		\draw[fill] ($(a3)+(a5)+(0.5,0)$) circle [radius=2pt] node [above] {$-c_\ell$};

		\draw[dashed] ($(a10)+(d1)$) to ($(a3)+(d1)$);
		\draw ($(a11)+(d1)$) to ($(a20)+(d1)$);
		\draw ($(a22)+(d1)$) to ($(a3)+(d1)$); 
		
		\draw[fill] ($(a1)+(d1)$) circle [radius=2pt] node [above] {$-b_1$};
		\draw[fill] ($(a2)+(d1)$) circle [radius=2pt] node [above] {$-b_2$};
		\draw[fill] ($(a3)+(d1)$) circle [radius=2pt] node [above] {$-b_k$};
		
		\draw[thick, ->] ($(a4)+(d1)$) to ($(a5)+(d1)$);

		\draw[dashed] ($(a10)+(d2)$) to ($(a3)+(d2)$);
		\draw ($(a11)+(d2)$) to ($(a20)+(d2)$);
		\draw ($(a22)+(d2)$) to ($(a3)+(d2)$); 
		
		\draw[fill] ($(a1)+(d2)$) circle [radius=2pt] node [above] {$-c_1$};
		\draw[fill] ($(a2)+(d2)$) circle [radius=2pt] node [above] {$-c_2$};
		\draw[fill] ($(a3)+(d2)$) circle [radius=2pt] node [above] {$-c_\ell$};
	\end{tikzpicture}\]
	where the $c_i$ are determined as follows: if 
	\[ (b_1,\dots,b_k)=\left(n_0{\cdot} 2, \ m_1+3, \ n_1{\cdot} 2,\ m_2+3,\dots,\ m_s+3,n_s{\cdot} 2\right) \]
	with $n_i\geq 0$ and $m_i\geq0$, then
	\[ (c_1,\dots, c_\ell)=\left\{\begin{array}{ll} (n_0+1) & (s=0), \\[0.3em]
	( n_0+2, \ m_1{\cdot} 2, \ n_1+3,\dots, \ m_s{\cdot} 2,\ n_s+2) & (s>0); \end{array} \right. \]
	\item if $y\in\mcY(\Gamma)$ forms a maximal chain of length $0$, then $y\in\mcY(\Gamma')$ with the same $o(y),t(y)\in\mcV(\Gamma')$ and $\sign'(y)=-1$; 
	\item Via the obvious identification $H_1(\Gamma^\ast,\AH(\Gamma))=H_1((\Gamma')^\ast,\AH(\Gamma'))$, the homomorphism $\varepsilon_{\Gamma_\dplumb'}\colon H_1((\Gamma')^\ast,\AH(\Gamma'))\to\ZZ/2\ZZ$ 
	assigns to each relative cycle $C$ of $((\Gamma')^\ast,\AH(\Gamma'))$ the number modulo $2$ of maximal chains on this cycle. 
	\end{enumerate}
In particular, for $b_i,c_j$ in {\rm (iii)}, $(b_1,\dots,b_k)=(c_1,\dots,c_k)$ if and only if $k=1$ and $b_1=c_1=2$. 
\end{thm}

%%%%%%%%%%%%%%%%%%%
\subsection{Seifert manifolds}
%%%%%%%%%%%%%%%%%%%
For an orientable $3$-manifold $M$, a \textit{Seifert fibration} is a partition of $M$ into circles $S^1$, called \textit{fibers}, such that each fiber in the interior $M^\circ$ has closed neighborhood $V$ which is a union of fibers and homeomorphic to a solid torus. 
A fiber $F$ of $M$ is said to be \textit{singular} if $F$ is not homotopic to other fibers in a small closed neighborhood $V$ as above, and is said to be \textit{regular} otherwise. 
An orientable $3$-manifold is called a \textit{Seifert manifold} if it admits a Seifert fibration. 
We assume that a Seifert manifold has a boundary, and that a trivialization is fixed for each boundary component of a Seifert manifold $M$. 
The following W-graph $\Gamma_\W$ 
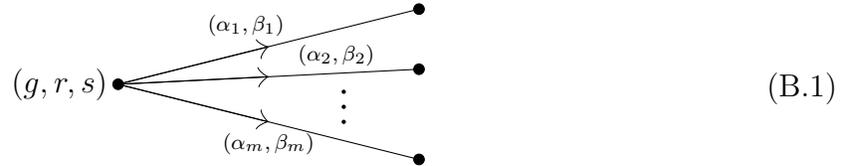
\begin{align}
\begin{aligned}\label{fig:seifert_Wgraph}
\begin{tikzpicture}
	\coordinate (o1) at (0,0); 
	\coordinate (b1) at (4,1);
	\coordinate (b2) at (4,0.2);
	\coordinate (b3) at (4,-0.6);
	\draw[fill] (o1) circle [radius=2pt] node [left] {$(g,r,s)$};
	\draw[fill] (b1) circle [radius=2pt];
	\draw[fill] (b2) circle [radius=2pt];
	\draw[fill] (b3) circle [radius=2pt];
	\draw (o1) -- node [left=3mm, above] {\tiny $(\alpha_1,\beta_1)$} (b1);
	\draw (o1) -- node [right=9mm,above] {\tiny $(\alpha_2,\beta_2)$} (b2);
	\draw (o1) -- node [below] {\tiny $(\alpha_m,\beta_m)$} (b3);
	\draw[-{Classical TikZ Rightarrow[length=4pt]}] (o1) -- (2,0.5);
	\draw[-{Classical TikZ Rightarrow[length=4pt]}] (o1) -- (2,0.1);
	\draw[-{Classical TikZ Rightarrow[length=4pt]}] (o1) -- (2,-0.3);
	\draw[fill] (3.4,-0.05) circle [radius=0.5pt]; 
	\draw[fill] (3.4,-0.2) circle [radius=0.5pt]; 
	\draw[fill] (3.4,-0.35) circle [radius=0.5pt]; 
\end{tikzpicture}
\end{aligned}
\end{align}
is called a \textit{Seifert manifold graph} since the above $\Gamma_\W$ describe the Seifert manifold with Seifert invariants $((g,r); -s; (\alpha_1,\beta_1),\dots,(\alpha_m,\beta_m))$, which corresponds to $(g;(1,-s),(\alpha_1,\beta_1),\dots,(\alpha_m,\beta_m))$ in \cite{neumann_raymond1978} if $r=0$ (cf. \cite[Corollary~5.7]{neumann1981} and \cite[Theorem~5.1]{neumann_raymond1978}). 
In \cite{neumann_raymond1978}, Neumann and Raymond defined the \textit{\red Euler number} $\bm{e}(M)$ of a Seifert manifold $M$ given by the W-graph (\ref{fig:seifert_Wgraph}) with $r=0$, i.e., without boundary, as follows: 
\[ \bm{e}(M):=s-\sum_{i=1}^m\frac{\beta_i}{\alpha_i}. \]
The {\red Euler} number $\bm{e}(M)$ can be naturally extend to a Seifert manifold $M$ with boundary having trivialization. 
The {\red Euler} number of $-M$ satisfies the following equation (see \cite[p.165]{neumann_raymond1978}):
\begin{align}\label{eq:Seiflt} 
\bm{e}(-M)=-\bm{e}(M). 
\end{align}
%For $b_i\geq 2$ ($i=1,\dots,s$), let $[b_1,\dots,b_s]$ be the continuous fraction given in Remark~\ref{rem:modified plumbing graph}. 
By \cite[Corollary~5.7]{neumann1981}, for a Seifert manifold $M$ with boundary corresponding to the above Seifert manifold graph, there is a \textit{star-shaped} decorated plumbing graph $\Gamma_\dplumb$ of F-normal form, which is in the following form such that $M\cong \PM(\Gamma_\dplumb)$: 
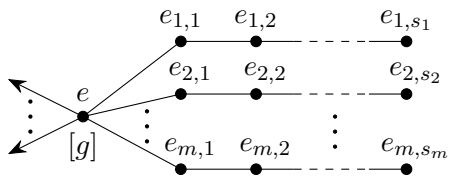
\begin{align*}\label{fig:star}
\begin{aligned}
\small \begin{tikzpicture}
	\coordinate (d1) at (1,0);
	\coordinate (d2) at (0.5,0);
	\coordinate (d3) at (0,-0.7);
	\coordinate (d4) at (0,-1);
	\coordinate (d51) at (0,0.2);
	\coordinate (d52) at (0,-0.2);
	\coordinate (a1) at (0,0);
	\coordinate (a01) at (-1,0.5);
	\coordinate (a02) at (-1,-0.5);
	\coordinate (a21) at (1.3,1);
	\coordinate (a31) at ($(a21)+(d1)$);
	\coordinate (a41) at ($(a31)+(d2)$);
	\coordinate (a51) at ($(a41)+(d1)$);
	\coordinate (a61) at ($(a51)+(d2)$);
	\coordinate (a22) at ($(a21)+(d3)$);
	\coordinate (a32) at ($(a31)+(d3)$);
	\coordinate (a42) at ($(a41)+(d3)$);
	\coordinate (a52) at ($(a51)+(d3)$);
	\coordinate (a62) at ($(a61)+(d3)$);
	\coordinate (a23) at ($(a22)+(d4)$);
	\coordinate (a33) at ($(a32)+(d4)$);
	\coordinate (a43) at ($(a42)+(d4)$);
	\coordinate (a53) at ($(a52)+(d4)$);
	\coordinate (a63) at ($(a62)+(d4)$);
	\coordinate (b1) at (-0.7,0);
	\coordinate (b2) at (0.85,-0.13);
	\coordinate (b3) at (3.35,-0.2);
	\draw[fill] (a1) circle [radius=2pt] node [above=2pt] {$e$} node [below=2pt] {$[g]$}; 
	\draw[fill] (a21) circle [radius=2pt] node [above] {$e_{1,1}$};
	\draw[fill] (a31) circle [radius=2pt] node [above] {$e_{1,2}$};
	\draw[fill] (a61) circle [radius=2pt] node [above] {$e_{1,s_1}$};
	\draw[fill] (a22) circle [radius=2pt] node [above] {$\ \ e_{2,1}$};
	\draw[fill] (a32) circle [radius=2pt] node [above] {$\ \ e_{2,2}$};
	\draw[fill] (a62) circle [radius=2pt] node [above] {$\ \ e_{2,s_2}$};
	\draw[fill] (a23) circle [radius=2pt] node [above] {$\ \ e_{m,1}$};
	\draw[fill] (a33) circle [radius=2pt] node [above] {$\ \ e_{m,2}$};
	\draw[fill] (a63) circle [radius=2pt] node [above] {$\ \ e_{m,s_m}$};
	\draw[-{Stealth[length=7pt]}] (a1) -- (a01); 
	\draw[-{Stealth[length=7pt]}] (a1) -- (a02); 
	\draw[dashed] (a41) -- (a51);
	\draw         (a1)  -- (a21) -- (a41);
	\draw         (a51) -- (a61);
	\draw[dashed] (a42) -- (a52);
	\draw         (a1)  -- (a22) -- (a42);
	\draw         (a52) -- (a62);
	\draw[dashed] (a43) -- (a53);
	\draw         (a1)  -- (a23) -- (a43);
	\draw         (a53) -- (a63);
	\draw[fill] (b1) circle [radius=0.6pt]; 
	\draw[fill] ($(b1)+(d51)$) circle [radius=0.6pt]; 
	\draw[fill] ($(b1)+(d52)$) circle [radius=0.6pt]; 
	\draw[fill] (b2) circle [radius=0.6pt]; 
	\draw[fill] ($(b2)+(d51)$) circle [radius=0.6pt]; 
	\draw[fill] ($(b2)+(d52)$) circle [radius=0.6pt]; 
	\draw[fill] (b3) circle [radius=0.6pt]; 
	\draw[fill] ($(b3)+(d51)$) circle [radius=0.6pt]; 
	\draw[fill] ($(b3)+(d52)$) circle [radius=0.6pt]; 
\end{tikzpicture}
\end{aligned}
\end{align*}
with $r$ boundary vertices, where $e=s-m$, $e_{i,j}\leq -2$ for each $i,j$ such that 
\[ \frac{\alpha_i}{\alpha_i-\beta_i}=[-e_{i,1},\dots,-e_{i,s_i}] \quad \mbox{for each $i=1,\dots,m$}. \] 
Then we have 
\[ \bm{e}(M)=e-\sum_{i=1}^m \frac{1}{[-e_{i,1},\dots,-e_{i,s_i}]}. \]
We can prove the following theorems by the same argument in the case without boundary. 

\begin{thm}[{cf. \cite[Theorem~1.1]{neumann_raymond1978} }]\label{thm:Seifert1}
	Let $M_1$ and $M_2$ be two Seifert manifolds. 
	If there is an orientation preserving homeomorphism $h\colon M_1\to M_2$ which preserves fibers, then $\bm{e}(M_1)=\bm{e}(M_2)$. 
\end{thm}

\begin{thm}[{cf. \cite[Theorem~5.2]{neumann_raymond1978}}]\label{thm:Seifelt2}
	Let $\Gamma_\dplumb$ be a star-shaped decorated plumbing graph, and let $M$ be a Seifert manifold corresponding to $\Gamma_\dplumb$. 
	If the intersection form of $\Gamma_\dplumb$ is positive definite (resp. negative definite), then $\bm{e}(M)>0$ (resp. $\bm{e}(M)<0$). 
\end{thm}

\bibliographystyle{plain}
\bibliography{biblio_3mfd.bib}

\end{document}